\documentclass[11pt, british, a4paper]{article}

\usepackage{aliascnt}
\usepackage{amsmath}
\usepackage{amssymb}
\usepackage{amsthm}
\usepackage[utf8]{inputenc}
\usepackage[style=alphabetic, maxalphanames=5, maxnames=99, giveninits=true]{biblatex}
\usepackage{bm} 
\usepackage{charter}
\usepackage{csquotes}
\usepackage{dsfont}
\usepackage{graphicx}
\usepackage{hyperref}
\usepackage{hyphenat}
\usepackage{mathtools}
\usepackage{multirow}
\usepackage[parfill]{parskip}
\usepackage{setspace}
\usepackage{subcaption}
\usepackage{thmtools}
\usepackage{tikz}
\usepackage{tikz-cd}
\usepackage{titlesec}
\usepackage{titletoc}
\usepackage{xcolor}
\usepackage{ytableau}

\usepackage[
  paper=a4paper,
  left=80pt,
  right=80pt,
  top=60pt,
  bottom=80pt,
  bindingoffset=0mm,
  includefoot,
  includehead,
  head=16.5pt,
  foot=16.5pt
]{geometry}

\renewcommand{\AA}{\mathbb{A}}

\newcommand{\CC}{\mathbb{C}}

\newcommand{\MM}{\mathbb{M}}

\renewcommand{\SS}{\mathbb{S}}

\DeclareSymbolFont{uppercaseletters}{OT1}{cmss}{m}{n}
\DeclareMathSymbol{A}{\mathalpha}{uppercaseletters}{`A}
\DeclareMathSymbol{B}{\mathalpha}{uppercaseletters}{`B}
\DeclareMathSymbol{C}{\mathalpha}{uppercaseletters}{`C}
\DeclareMathSymbol{D}{\mathalpha}{uppercaseletters}{`D}
\DeclareMathSymbol{E}{\mathalpha}{uppercaseletters}{`E}
\DeclareMathSymbol{F}{\mathalpha}{uppercaseletters}{`F}
\DeclareMathSymbol{G}{\mathalpha}{uppercaseletters}{`G}
\DeclareMathSymbol{H}{\mathalpha}{uppercaseletters}{`H}
\DeclareMathSymbol{I}{\mathalpha}{operators}{`I} 
\DeclareMathSymbol{J}{\mathalpha}{uppercaseletters}{`J}
\DeclareMathSymbol{K}{\mathalpha}{uppercaseletters}{`K}
\DeclareMathSymbol{L}{\mathalpha}{uppercaseletters}{`L}
\DeclareMathSymbol{M}{\mathalpha}{uppercaseletters}{`M}
\DeclareMathSymbol{N}{\mathalpha}{uppercaseletters}{`N}
\DeclareMathSymbol{O}{\mathalpha}{uppercaseletters}{`O}
\DeclareMathSymbol{P}{\mathalpha}{uppercaseletters}{`P}
\DeclareMathSymbol{Q}{\mathalpha}{uppercaseletters}{`Q}
\DeclareMathSymbol{R}{\mathalpha}{uppercaseletters}{`R}
\DeclareMathSymbol{S}{\mathalpha}{uppercaseletters}{`S}
\DeclareMathSymbol{T}{\mathalpha}{uppercaseletters}{`T}
\DeclareMathSymbol{U}{\mathalpha}{uppercaseletters}{`U}
\DeclareMathSymbol{V}{\mathalpha}{uppercaseletters}{`V}
\DeclareMathSymbol{W}{\mathalpha}{uppercaseletters}{`W}
\DeclareMathSymbol{X}{\mathalpha}{uppercaseletters}{`X}
\DeclareMathSymbol{Y}{\mathalpha}{uppercaseletters}{`Y}
\DeclareMathSymbol{Z}{\mathalpha}{uppercaseletters}{`Z}

\newcommand{\A}{\mathnormal{A}}
\newcommand{\B}{\mathnormal{B}}
\newcommand{\C}{\mathnormal{C}}
\newcommand{\D}{\mathnormal{D}}
\newcommand{\E}{\mathnormal{E}}
\newcommand{\F}{\mathnormal{F}}
\newcommand{\G}{\mathnormal{G}}

\newcommand{\J}{\mathnormal{J}}

\newcommand{\M}{\mathnormal{M}}
\newcommand{\N}{\mathnormal{N}}

\newcommand{\R}{\mathnormal{R}}
\renewcommand{\S}{\mathnormal{S}}
\newcommand{\T}{\mathnormal{T}}

\newcommand{\V}{\mathnormal{V}}

\newcommand{\Z}{\mathnormal{Z}}

\declaretheoremstyle[
headfont={\bfseries},
notefont=\bfseries\mathversion{bold},
notebraces={(}{)},
bodyfont=\itshape,
postheadspace=0.5em,
]{mydefinition}

\declaretheoremstyle[
headfont={\bfseries},
notefont=\bfseries\mathversion{bold},
notebraces={(}{)},
bodyfont=\itshape,
postheadspace=0.5em,
]{mytheorem}

\declaretheoremstyle[
headfont={\bfseries},
notefont=\bfseries\mathversion{bold},
notebraces={(}{)},
bodyfont=\normalfont,
postheadspace=0.5em,
qed=$\circ$
]{myremark}

\declaretheoremstyle[
headfont={\bfseries},
notefont=\bfseries\mathversion{bold},
notebraces={(}{)},
bodyfont=\normalfont,
postheadspace=0.5em,
qed=$\circ$
]{myexample}

\declaretheorem
[style=mydefinition, 
name=Definition, 
refname={Definition,Definitions},
numberwithin=section]
{definition}

\declaretheorem
[style=mytheorem,
name=Theorem,
refname={Theorem,Theorems},
sibling=definition]
{theorem}

\declaretheorem
[
style=mytheorem,
name=Theorem,
refname={Theorem,Theorems},
numbered=no
]{theorem*}

\declaretheorem
[style=mytheorem,
name=Proposition, 
refname={Proposition,Propositions},
sibling=definition]
{proposition}

\declaretheorem
[style=mytheorem,
name=Lemma, 
refname={Lemma,Lemmas},
sibling=definition]
{lemma}

\declaretheorem
[style=mytheorem,
name=Corollary, 
refname={Corollary,Corollaries},
sibling=definition]
{corollary}

\declaretheorem
[style=myremark,
name=Remark, 
refname={Remark,Remarks},
sibling=definition]
{remark}

\declaretheorem
[style=myexample,
name=Example, 
refname={Example,Examples},
sibling=definition]
{example}

\newcommand{\Aut}{\operatorname{\mathsf{Aut}}}
\renewcommand{\cal}[1]{{\normalfont\mathcal{#1}}}
\newcommand{\catname}[1]{\textup{#1}}
\newcommand*{\coleq}{\mathrel{\mathop:}=}
\newcommand{\coliff}{\mathrel{\mathop:}\Leftrightarrow}

\let\Coprod\coprod
\renewcommand{\coprod}{\amalg}

\newcommand{\del}{\partial}
\newcommand{\dfn}[1]{\textbf{\textit{\nohyphens{#1}}}}
\newcommand{\diag}{\operatorname{diag}}
\newcommand{\End}{\operatorname{\mathsf{End}}}
\newcommand*{\eqcol}{=\mathrel{\mathop:}}
\newcommand{\GL}{\operatorname{\mathsf{GL}}}
\newcommand{\Hom}{\operatorname{\mathsf{Hom}}}
\newcommand{\id}{\mathrm{id}}
\newcommand{\idd}{\mathds{1}}
\newcommand{\inj}{\lhook\joinrel\rightarrow}
\newcommand{\inner}[1]{\left\langle #1 \right\rangle} 
\newcommand{\iso}{
	\overset{\,\smash{
			\raisebox{-0.45ex}{\ensuremath{\scriptstyle\sim}}}
		\,}{\too}}
\newcommand{\mapstoo}{\longmapsto}
\newcommand{\mtx}[1]{\begin{bmatrix} #1 \end{bmatrix}}
\newcommand{\op}[1]{\operatorname{#1}}
\newcommand{\PGL}{\operatorname{\mathsf{PGL}}}
\newcommand*{\qtext}[1]{\quad\text{#1}\quad}
\newcommand*{\qqtext}[1]{\qquad\text{#1}\qquad}
\newcommand{\rank}{\operatorname{rank}}
\renewcommand{\rm}[1]{{\normalfont\mathrm{#1}}}
\newcommand{\set}[1]{\left\{ #1 \right\}} 
\newcommand{\sfop}[1]{\mathop{\mathsf{#1}}\nolimits}
\newcommand{\Sym}{\operatorname{\mathsf{Sym}}}
\newcommand{\tinyunderset}[2]{\underset{\hbox{\tiny$#2$}}{#1}}
\newcommand{\too}{\longrightarrow}

\let\oldvarphi\varphi
\let\oldphi\phi
\renewcommand{\phi}{\oldvarphi}
\renewcommand{\varphi}{\oldphi}

\numberwithin{equation}{section}

\DeclareDocumentEnvironment
{eqn}
{}	
{\csname equation*\endcsname}
{\csname endequation*\endcsname\ignorespacesafterend}

\DeclareDocumentEnvironment
{eqntag}
{}	
{\csname equation\endcsname}
{\csname endequation\endcsname\ignorespacesafterend}

\DeclareDocumentEnvironment
{eqns}
{}	
{\csname flalign*\endcsname}
{\csname endflalign*\endcsname\ignorespacesafterend}

\DeclareDocumentEnvironment
{eqnstag}
{}	
{\csname flalign\endcsname}
{\csname endflalign\endcsname\ignorespacesafterend}

\allowdisplaybreaks 

\hypersetup{hidelinks} 

\definecolor{darkblue}{RGB}{0,20,115}
\definecolor{darkgreen}{rgb}{0,0.4,0}

\titleformat{\chapter}
{\normalfont\LARGE\bfseries}{\thechapter}{1em}{} 

\titlespacing{\chapter}{0pt}{0pt}{20pt} 

\usepackage{csquotes}
\usepackage[style=alphabetic, maxalphanames=5, maxnames=99, giveninits=true]{biblatex}

\appto{\bibsetup}{\sloppy}
\AtBeginBibliography{}

\begin{document}

\title{Decorated Local Systems and Character Varieties}

\author{Benedetta Facciotti, Marta Mazzocco, Nikita Nikolaev}

\date{\today}

\vspace{-0.5cm}

\maketitle

\begin{abstract}
	\noindent The focus of this paper is the study of the moduli space of representations of fundamental groupoids of 
surfaces $\Sigma$ with boundaries with values in $G:=GL_n(\mathbb C)$. In absence of marked points on the boundary, this moduli space is realized in many equivalent ways: as the moduli space of linear local systems on $\Sigma$, as the moduli space of representations of the fundamental groupoid $\Pi_1 (\Sigma)$, as the space of monodromy data and as character variety. 
By adding marked points to the boundary of $\Sigma$ in order to capture irregular singularities, the Betti moduli space has been generalized in several ways by many authors. Although it is clear that these different approaches describe essentially the same spaces of mathematical objects, 
exactly how they fit together has not yet been established.
Motivated by the broader programme of establishing an explicit and conceptually coherent relationship between the existing approaches to the study of the  decorated Betti moduli space, in this paper, we develop a categorical framework that allows for a systematic definition of the 
 \dfn{decorated Betti moduli spaces} space, in the presence of higher order poles, designed to specialize to the different points of view encountered in the literature.
\end{abstract}


\begin{spacing}{0.9}
	\small
	\setcounter{tocdepth}{2}
	\tableofcontents
\end{spacing}

\setcounter{section}{0}
\section*{Introduction}

The focus of this paper is the study of moduli spaces of representations of fundamental groupoids of 
surfaces $\Sigma$ with boundaries with values in $G:=GL_n(\mathbb C)$. In absence of marked points on the boundary, this moduli space is realized in many equivalent ways: as the moduli space $\mathfrak{Loc}_G (\Sigma)$ of linear local systems on $\Sigma$, as the moduli space $\mathfrak{Rep}_G \big( \Pi_1 (\Sigma) \big)$ of representations of the fundamental groupoid $\Pi_1 (\Sigma)$, as the moduli space $\mathfrak{Rep}_G \big( \pi_1 (\Sigma, p) \big)$ of monodromy data and as character variety $\mathfrak{X}_G (\Sigma)$. We call all these moduli spaces\footnote{When we speak of a \textit{moduli space}, we think of an appropriate moduli stack, but in almost every instance it is sufficient to restrict our attention to the underlying set of isomorphism classes.} collectively as \dfn{Betti moduli space}. Their equivalence can be stated as follows:
\begin{eqn}
	\text{$\mathfrak{Loc}_G (\Sigma)$}
	\cong \text{$\mathfrak{Rep}_G \big( \Pi_1 (\Sigma) \big)$}
	\cong \text{$\mathfrak{Rep}_G \big( \pi_1 (\Sigma, p) \big)$}
	\cong \text{$\mathfrak{X}_G (\Sigma)$}.
\end{eqn}
For convenience of the reader, we recall these definitions and the proof of these equivalences in Sections \ref{suse:locsys}, \ref{251209141021} and \ref{251217152313}.

By adding marked points to the boundary of $\Sigma$ in order to capture irregular singularities, the Betti moduli space has been generalized in several ways by many authors as summarized below. Although it is clear that these different approaches describe essentially the same spaces of mathematical objects, 
exactly how they fit together has not yet been established.
In this paper, we develop a categorical framework that allows for a clear and systematic definition of the \dfn{decorated Betti moduli spaces} designed to specialize to the different points of view encountered in the literature.

Let us briefly recall the generalizations considered in this paper. 

The moduli space of linear local systems was enriched by adding filtrations at marked points implicitly in Deligne’s work on regular singular connections \cite{MR0417174}. The systematic construction of moduli spaces of such filtered local systems was later carried out by Simpson in the context of parabolic structures \cite{MR1040197}, and further extended to irregular singularities by Boalch. Building on this framework, Fock and Goncharov study special positive and decorated loci inside these moduli spaces, endowing them with cluster structures and giving rise to higher Teichm\"uller spaces \cite{MR2233852}. The notion of  moduli space of filtered\footnote{We note that in \cite{MR2233852} the authors use the term framed instead of filtered.} was further generalized by Goncharov and Shen \cite{250619181547} who introduced the moduli space of pinnings, namely an enrichment of the moduli space of filtered $G$–local systems on $\Sigma$ by adding decorations on some of the flags.
Note that the concept of \textit{pinning} was essentially introduced by Chevalley and later formalised by Tits, Grothendieck, and Demazure in the study of reductive algebraic groups.
In the context of filtered local systems, pinnings were used by Fock-Goncharov \cite[§2.2]{MR2233852} and Goncharov-Shen who gave an alternative definition of \textit{pinning} \cite[Def 3.7, fn.\,12]{250619181547}.
If the group is adjoint (e.g., $\PGL(\CC)$ but not $\GL_n(\CC)$), then their definition of \textit{pinning} is equivalent to the classical notion \cite[Lemma 3.30]{250619181547}.

The moduli space of representations of the fundamental groupoid was generalized by Gual-tieri, Li and Pym in \cite{gualtieri2013stokes} exploting the theory of Lie groupoids and Lie algebroids which is naturally involved in the holomorphic case, in which the Betti moduli space is naturally homeomorphic to the de Rham moduli space of 
holomorphic  connections on a Riemann surface.
Indeed, any holomorphic  connection on a Riemann surface $\Sigma$ defines a representation of the space of vector fields $\mathcal{T}_\Sigma$, which has a structure of a Lie algebroid.
A distinguished integration of such a Lie algebroid is the fundamental Lie groupoid $\Pi_1(\Sigma)$, whose representation corresponds to the parallel transport determined by the connection.
By Lie's second theorem, representations of $\mathcal{T}_\Sigma$ and $\Pi_1(\Sigma)$ are equivalent.
What they showed is that this construction can be naturally extended to the case of higher order poles by incorporating the effective divisor $D$ encoding the poles and their multiplicities.
Hence, as meromorphic connections are representations of the Lie algebroid $\mathcal{T}_\Sigma(-D)$ of vector fields vanishing along $D$, Lie’s second theorem implies they are equivalent to representations of the integrating Lie groupoid $\Pi_1(\Sigma,D)$.

The moduli space of monodromy data was generalized by Boalch in \cite{Boalch2001}, where he introduced a moduli space that was later called \textit{wild character variety} in \cite{Boalch14}.
Topologically, instead of the original Riemann surface, one considers its real oriented blow-up at the poles of the connection. 
Additional points are then removed at the intersections of circles in tubular neighborhoods around the boundaries created by higher-order poles with the so-called Stokes directions. 
Algebraically, rather than using the usual notion of representation, the construction employs Stokes representations, defined in \cite{Boalch2001}.
The moduli space is then defined taking the quotient by a specific action of the centralizer of the irregular type defined around each pole.

Finally, in \cite{MR3802126}, Chekhov, Mazzocco and Rubtsov defined the so called \textit{decorated character variety}, now renamed \textit{bordered cusped character variety}.
The construction of this moduli space involves some topological modifications to the starting surface, based on \cite{gaiottoMooreNeitzke13}.
More precisely, while one still needs to consider the real oriented blow-up of the surface at the poles, the approach here differs in that, instead of introducing additional punctures, one designates marked points on the boundaries arising from higher order poles only.
Treating the marked points as cusps, the authors assign a Borel subgroup  to each of them.
The bordered cusped character variety is then the quotient of the representation space of the groupoid of arcs with base points at the cusps by the action by mixed conjugation (see \autoref{251217183124}) of the unipotent radical of the Borel subgorups associated to the cusps.
This approach was then generalized to the $n$ dimensional case in \cite{MR3932256}.

This paper is motivated by the broader programme of establishing an explicit and conceptually coherent relationship between all the above existing approaches to the study of the  decorated Betti moduli space. As a concrete contribution to this programme, we develop a categorical framework that allows for a systematic definition of the Betti moduli space in the presence of higher order poles.
This viewpoint is also compatible with the Riemann–Hilbert correspondence, whose emphasis on the classification of geometric and algebraic objects makes a categorical perspective natural.
The first key step of our approach is therefore the introduction of suitable groupoids that accurately describe these structures.

In order to state the main result of this paper, we briefly introduce the notation for the relevant groupoids by providing a rough description of their objects, referring the reader to the appropriate sections for the precise definitions. A common feature of all these groupoids is that their objects carry certain decorations on the fibers/stalks over the marked points.

In this paper, we consider surfaces with marked boundary $(\Sigma,P)$, where $\Sigma$ is a compact connected, oriented, smooth real two-dimensional
manifold with boundary components each of which is diffeomorphic to a circle,  and $P\subset \partial \Sigma$ is a finite set of points. These points carry some extra information as described in Definition \ref{defMarkedBorderedSurface}, but we neglect it in this introduction.

Following \cite{MR2233852}, 
a     \dfn{filtered local system} on a surface with marked boundary $(\Sigma, P)$ is a pair $(\cal{E}, \cal{F})$ consisting of a local system $\cal{E}$ on $\Sigma$ endowed with a collection $\cal{F} = \set{ \cal{F}_p : p \in P}$ where $\cal{F}_p$ is a filtration near $p$ (see Definition \ref{251126214152}). We denote the groupoid of
filtered local systems, and the corresponding moduli space, respectively by 
\begin{eqn}
	\catname{Loc}_G^\rm{Fi} (\Sigma, P)
	\qtext{and}
	\mathfrak{Loc}_G^\rm{Fi} (\Sigma, P)
	\coleq \catname{Loc}_G^\rm{Fi} (\Sigma, P) \big/{\sim}.
\end{eqn}

Next, a \dfn{filtered representation} of the fundamental groupoid $\Pi_1 (\Sigma)$ of a surface with marked boundary $(\Sigma, P)$ is a triple $(\E, \F, \Phi)$ consisting of a representation $(\E, \Phi)$ of the fundamental groupoid $\Pi_1 (\Sigma)$ endowed with a collection $\F \coleq \set{ \F_p : p \in P}$ of filtrations on the fibres over the marked points which are flat in the sense specified in   Definition \ref{251127151457}. We denote the groupoid of filtered representations of rank $n$ on $(\Sigma, P)$ with filtered isomorphisms, and the corresponding moduli space (see Definition \ref{251127153641}), by 
\begin{eqn}
	\catname{Rep}_G^\rm{Fi} \big( \Pi_1 (\Sigma), P \big)
	\qtext{and}
	\mathfrak{Rep}_G^\rm{Fi} \big( \Pi_1 (\Sigma), P \big)
	\coleq \catname{Rep}_G^\rm{Fi} \big( \Pi_1 (\Sigma), P \big) \big/{\sim}.
\end{eqn}

Thanks to the fact that the decorations arise only on marked points, the introduction of  a discrete model for this groupoid is natural. Therefore, we introduce the groupoid of filtered representations of $\pi_1 (\Sigma, P)$ of rank $n$, and the corresponding moduli space (see Definitions \ref{250912195848} and \ref{250914115030}) by
\begin{eqn}
	\catname{Rep}_G^\rm{Fi} \big( \pi_1 (\Sigma, P) \big)
	\qtext{and}
	\frak{Rep}_G^\rm{Fi} \big( \pi_1 (\Sigma, P) \big)
	\coleq \catname{Rep}_G^\rm{Fi} \big( \pi_1 (\Sigma, P) \big) \big/{\sim}
\end{eqn}
respectively.

Finally we consider (see Definition
\ref{251218003508})
the \dfn{decorated representation variety} of $(\Sigma, P)$, namely  the subset $R_G (\Sigma, P)$ of $\Hom \big( \pi_1 (\Sigma, P), G \big)$ consisting of all groupoid homomorphisms $\rho : \pi_1 (\Sigma, P) \to G$ which are upper-triangular on loops based at marked points that lie on boundaries corresponding to simple poles. On this, we act by mixed conjugation by the affine algebraic subgroup of $G_P :=G^{|P|}$:
$$
B_{P} \coleq \prod_{p \in P} B = \Aut (\underline{\CC}^n, \underline{\CC}^\bullet)
$$
where $ \sfop{Aut} (\underline{\CC}^n, \underline{\CC}^\bullet)$ is the set of automorphisms of $\underline{\CC}^n$ that preserve the standard flag $\underline{\CC}^\bullet$.
In this way, we get an action groupoid on an affine algebraic variety:
\begin{eqn}
	B_P \ltimes R_G (\Sigma, P)
	\qtext{with}
	g . \rho \coleq \rm{t}^\ast g \cdot \rho \cdot \rm{s}^\ast g^{-1}.
\end{eqn}
Its moduli space (see Definition \ref{251217191739})
is called \dfn{filtered character stack}  and is defined as the following affine algebraic quotient stacks:
\begin{eqn}
	\frak{X}^\rm{Fi}_G (\Sigma, P) \coleq R_G (\Sigma, P) / B_P.
\end{eqn}

Similarly, we can introduce  framed (and projectively framed) analogues of all the above categories by framing every flag with a basis (or the equivalence class of a basis) and requiring all filtered isomorphisms to be unipotent (or projectively unipotent). We denote the corresponding groupoids and moduli spaces in the same way as above replacing the index $\rm{Fi}$ by  $\rm{Fr}$ (or $\rm{PFr}$). Therefore we will speak of \dfn{decorated} groupoids when referring to both filtered and framed (or projectively framed) groupoids.

Establishing the equivalence of these groupoids yields, at the level of objects, a bijection between the underlying sets of the corresponding moduli spaces.

\begin{theorem*}
	For any surface with marked boundary $(\Sigma,P)$, there are canonical equivalences of categories between the groupoids of decorated local systems, decorated representations of the fundamental groupoid $\Pi_1 (\Sigma)$, decorated representations of the discrete fundamental groupoid $\pi_1 (\Sigma, P)$, and suitable mixed conjugation action groupoids over the decorated representation variety $R_G (\Sigma, P)$:
	\begin{eqn}
		\catname{Loc}_G^\square (\Sigma, P) 
		\cong \catname{Rep}_G^\square \big( \Pi_1 (\Sigma), P \big)
		\cong \catname{Rep}_G^\square \big( \pi_1 (\Sigma, P) \big)
		\cong K_P^\square \ltimes R_G (\Sigma, P),
	\end{eqn}
	for any $\square \in \set{\rm{Fi}, \rm{Fr},\rm{PFr}}$ where $K^\rm{Fi}_P \coleq B_P$,  $K^\rm{Fr}_P \coleq U_P$ and  $K^\rm{Fr}_P \coleq U_P^\times$ (see \eqref{251217191045}).
	As a consequence, there are canonical bijections between the corresponding moduli sets::
	\begin{eqn}
		\frak{Loc}_G^\square (\Sigma, P) 
		\cong \frak{Rep}_G^\square \big( \Pi_1 (\Sigma), P \big)
		\cong \frak{Rep}_G^\square \big( \pi_1 (\Sigma, P) \big)
		\cong \frak{X}_G^\square (\Sigma, P).
	\end{eqn}
\end{theorem*}

Since $\mathfrak{Loc}_G^\rm{Fi} (\Sigma, P)$ carries a natural cluster variety structure \cite{MR2233852}, this theorem shows that this cluster structure is intrinsic to the filtered Betti moduli space. 

In the case when all the boundaries coming from simple poles are unmarked, i.e. when $P=P_{\rm I}$ (see Definition \ref{defMarkedBorderedSurface}), and the group is $SL_n(\mathbb C)$, the framed character variety is of the form $\frak{X}^\rm{Fr}_{SL_n(\mathbb C)} (\Sigma, P_{\rm I}) = \Hom \big( \pi_1 (\Sigma, P_{\rm I}), SL_n(\mathbb C) \big) / U_{P_{\rm I}}$. This was introduced by Chekhov, Mazzocco and Rubtsov in \cite[§4]{MR3802126} (who called it the \textit{decorated character variety}) based on the work of Li-Bland and Ševera in \cite[§6]{MR3424475}.
These papers have served as one of the main inspirations for this project. In particular, we wanted to prove that 
the Fock-Goncharov-Shen variables can be used to coordinatize any of the above moduli sets even when some decorations are forgotten - see Remark \ref{250918103523}. Motivated by this, 
in Section \ref{se:5nuova}, we show that the map induced by the forgetful functor
\begin{eqn}
	\catname{Loc}_G^{\rm{Fi}} (\Sigma, P) \too \catname{Loc}_G^{\rm{Fi}} (\Sigma, P_\rm{I})
\end{eqn}
is a ramified covering. 

This paper is organized as follows. In Section \ref{re:1},
we give a careful treatment of the
classical setting of linear local systems on surfaces with boundary, representations of the fundamental groupoid, and the relationship between the corresponding character variety and the moduli space of monodromy data. In Section 
\ref{se:2} we motivate and define surfaces with marked boundaries, and give a finite presentation  for the fundamental groupoid $\pi_1(\Sigma,P)$. In Section \ref{se:3}, we introduce various notions of local systems, which we collectively call \dfn{decorated local systems}, and we discuss the relationships between them. In Section \ref{se:4}, we reformulate decorated local systems as \textit{decorated} representations of the fundamental groupoid and the discrete fundamental groupoid. In Section \ref{se:5nuova}, we study the forgetful functors that drop the secondary marked points. In Section \ref{se:5}, we give an explicit description of the moduli space of decorated local systems.

\vskip 2mm

\noindent\textbf{Acknowledgements} The authors are grateful to Anton Alexseev, Marco Gualtieri, Gabriele Rembado and Volodya Rubtsov for many discussions on the topics of this paper. This was funded by 
the Leverhulme Trust Research Project Grant RPG-2021-047 and by the Proyecto de Generaci\'on de Conocimiento, PID2024-155686NB-I00 of the Spanish Ministry of Science, Innovation and Universities.

\section{Review: Local Systems and Character Varieties}\label{re:1}

In this section, we set the stage by reviewing the classical setting of linear local systems on surfaces with boundary, representations of the fundamental groupoid, and the relationship between the corresponding character variety and the moduli space of monodromy data.
We give a rather careful treatment of the subject for the purpose of completeness as well as to introduce notation and recall basic concepts needed in the rest of the paper.

In what follows, we put $G \coleq \GL_n (\CC)$ for arbitrary $n$, and let $B$ and   $U$ denote respectively the subgroups consisting of upper-triangular and unipotent upper-triangular matrices in $G$.
We also let $\Sigma$ be any surface by which we mean a connected, oriented, topological real two-dimensional manifold with or without boundary.

\subsection{Groupoids and Moduli Spaces}

\subsubsection{Groupoids}
First, recall that a \dfn{groupoid} is a category in which every morphism is invertible.
To spell this out a little, a groupoid consists of a set of objects $X$ and a set of morphisms $\cal{G}$ together with several structure maps including a pair of maps $\rm{s}, \rm{t} : \cal{G} \to X$ called the \textit{source} and \textit{target maps} which respectively extract the source and target objects of a morphism.
We usually say that $\cal{G}$ is a groupoid over $X$ and write $\cal{G} \rightrightarrows X$.
There is also an embedding of $X$ into $\cal{G}$ as the subset of identity morphisms and we typically identify $X$ with its embedded image inside $\cal{G}$.

A \dfn{Lie groupoid} is a groupoid $\cal{G}\rightrightarrows X$ such that $\cal{G}$ and $X$ are topological manifolds (although $\cal{G}$ may be non-Hausdorff), all the structure maps are continuous, the embedding $X \inj \cal{G}$ is closed, and the source and target maps $\rm{s}, \rm{t}$ are submersions.
Similarly, we can define \textit{holomorphic} and \textit{algebraic} Lie groupoids.
A \dfn{groupoid homomorphism} $\F : \cal{G} \to \cal{G}'$ is simply a functor and a \dfn{Lie groupoid homomorphism} is a functor which is also a continuous (resp. holomorphic or algebraic) map.

\subsubsection{Moduli spaces}
Any groupoid $\cal{G}\rightrightarrows X$ has a corresponding \textit{moduli set}: it is the set of isomorphism classes of objects in $X$ that we shall denote in any of the following ways:
\begin{eqn}
	\frak{G}
	= X/\cal{G}
	= X /{\sim}.
\end{eqn}
If $x$ is an object of $X$, we usually denote its isomorphism class by $[x]$.
If the groupoid $\cal{G} \rightrightarrows X$ carries additional structure --- such as continuous, smooth, analytic, or algebraic --- then the moduli set $\frak{G}$ may be enriched to a \textit{moduli stack} in the respective category.
However, these details will not play any appreciable role in our story.
So, from now on, when we speak of a \textit{moduli space}, we think of an appropriate moduli stack, but in almost every instance it is sufficient to restrict our attention to the underlying moduli set.

\subsubsection{Three examples of Lie groupoids}
The first example of a Lie groupoid is as follows.
If $K \times X \to X$ is a group action on a set $X$, then we can form the \dfn{action groupoid} $\cal{G} = K \ltimes X \rightrightarrows X$.
Its set of morphisms is given by $K \times X$ with source and target maps to $X$ given by the natural projection and the group action, respectively: $\rm{s} (k,x) = x$ and $\rm{t} (k,x) = k.x$.
The composition is given by $(k_2, x_2) \circ (k_1,x_1) = (k_2 k_1, x_1)$ whenever $x_2 = k_1 . x_1$.
The corresponding moduli set $X / K$ is the set of all $K$-orbits in $X$.
Moreover, if $X$ is an affine algebraic variety, $K$ is an affine algebraic group, and the action is algebraic, then $\cal{G} = K \ltimes X \rightrightarrows X$ is an algebraic Lie groupoid, and the corresponding moduli space $\frak{G}$ is an algebraic stack equal to the stack-theoretic quotient $X / K$.

Second, the \dfn{fundamental groupoid} $\Pi_1 (\Sigma) \rightrightarrows \Sigma$ of the surface $\Sigma$ is the space of all continuous paths $\gamma : [0,1] \to X$ considered up to homotopy with fixed endpoints.
It is a real four-dimensional topological manifold.
The \textit{restriction} of the groupoid $\Pi_1 (\Sigma)$ to a subset $S \subset \Sigma$ is by definition the full subgroupoid $\Pi_1 (\Sigma) |_S \coleq \rm{s}^{-1} (S) \cap \rm{t}^{-1} (S) \subset \Pi_1 (\Sigma)$ consisting of paths on $\Sigma$ that begin and end in $S$.
In particular, for any point $p \in \Sigma$, the restriction of the fundamental groupoid $\Pi_1 (\Sigma)$ to $p$ is nothing but the fundamental group:
\begin{eqn}
	\pi_1 (\Sigma, p) = \Pi_1 (\Sigma) \big|_p = \rm{s}^{-1} (p) \cap \rm{t}^{-1} (p) \subset \Pi_1 (\Sigma).
\end{eqn}
Thus, the fundamental group $\pi_1 (\Sigma, p)$ is an example of a groupoid over a single point.

Third, to any topological complex vector bundle $\E \to \Sigma$, we can naturally associate a Lie groupoid $\GL (\E) \rightrightarrows \Sigma$ called the \dfn{transport groupoid}.
Its objects are the points of $\Sigma$ and its morphisms are the linear isomorphisms between the fibres of $\E$; that is, for any pair of points $p,q \in \Sigma$, the set of morphisms in $\GL (\E)$ from $p$ to $q$ is the set $GL (\E |_p, \E |_q)$ of all linear isomorphisms $\E |_p \iso \E |_q$.
The restriction of $\GL (\E)$ to a point $p \in \Sigma$ is nothing but the automorphism group of the fibre: $\GL (\E) |_p = \GL (\E |_p) \simeq G$.

\subsection{Linear Local Systems}\label{suse:locsys}

In this paper, by a \dfn{local system} on $\Sigma$ we shall mean either a linear local system or equivalently a $G$-local system.
Let us quickly remind the reader of what these are and how to pass back and forth between the two notions.

\textit{Linear local systems} (more precisely, \textit{complex-linear local systems of finite rank}) are locally constant sheaves $\cal{E}$ on $\Sigma$ of finite-dimensional complex vector spaces.
In particular, all stalks $\cal{E}_p$ of $\cal{E}$ are complex vector spaces of the same dimension called the \dfn{rank} of $\cal{E}$.
Likewise, a $G$-\textit{local system} is a locally constant sheaf $\hat{\cal{E}}$ of sets on $\Sigma$ equipped with a right $G$-action such that the stalk $\hat{\cal{E}}_p$ of $\hat{\cal{E}}$ at every point $p \in \Sigma$ is a $G$-torsor; i.e., a set equipped with a free and transitive right $G$-action.

Linear local systems on $\Sigma$ together with linear sheaf isomorphisms (i.e., sheaf morphisms which are linear isomorphisms on stalks) form a groupoid which, along with the corresponding \dfn{moduli space of local systems}, will be respectively denoted by
\begin{eqn}
	\catname{Loc}_G (\Sigma)
	\qqtext{and}
	\frak{Loc}_G (\Sigma)
	\coleq \catname{Loc}_G (\Sigma) \big/{\sim}.
\end{eqn}
We usually draw no distinction between $G$-local systems and linear local systems because their categories are naturally equivalent thanks to the following lemma.

\begin{lemma}
	\label{251008162022}
	There is an equivalence of categories between the groupoid of linear local systems and the groupoid of $G$-local systems (where morphisms are sheaf morphisms which are $G$-equivariant maps on stalks) given by the following pair of inverse functors:
	\begin{eqn}
		\cal{E} \mapstoo \hat{\cal{E}} = GL (\CC^n, \cal{E})
		\qqtext{and}
		\hat{\cal{E}} \mapstoo \cal{E} \coleq \hat{\cal{E}} \times_G \underline{\CC}^n_\Sigma.
	\end{eqn}
\end{lemma}

The notation for this pair of functors will be explained in the course of the proof.

\begin{proof}
	If $\cal{E}$ is a linear local system of rank $n$, then for any $p \in \Sigma$ the set $\hat{\cal{E}}_p \coleq GL (\CC^n, \cal{E}_p)$ of all linear isomorphisms $\CC^n \iso \cal{E}_p$ is a $G$-torsor: it has a free and transitive right $G$-action by pre-composition; i.e., $f.g = f \circ g$ for any $f \in GL (\CC^n, \cal{E}_p)$ and any $g \in G = GL(\CC^n)$.
	Similarly, the set $GL \big( \CC^n, \cal{E} (U))$ for any open set $U \subset \Sigma$ has a right $G$-action though not necessarily free and transitive.
	Therefore, the correspondence
	\begin{eqn}
		\hat{\cal{E}} : U \mapsto GL \big( \CC^n, \mathcal{E} (U))
	\end{eqn}
	is a locally constant sheaf of sets with a right $G$-action whose stalk at any $p \in \Sigma$ is the $G$-torsor $\hat{\cal{E}}_p = \GL (\CC^n, \cal{E}_p)$.
	In other words, $\hat{\cal{E}}$ is a $G$-local system, which we call $\hat{\cal{E}} \eqcol GL (\CC^n, \cal{E})$ by a slight abuse of notation.
	
	Conversely, if $\hat{\cal{E}}$ is a $G$-local system, then $\cal{E} \coleq \hat{\cal{E}} \times_G \underline{\CC}^n_\Sigma$ is a linear local system on $\Sigma$, where $\underline{\CC}^n_\Sigma$ is the constant sheaf with value $\CC^n$ with its natural right $G$-action.
	Recall that in general, if $X$ and $Y$ are two sets equipped with a right $G$-action, then $X \times_G Y$ (sometimes called the \textit{tensor product of $G$-sets}) is defined as the quotient set $X \times Y / {\sim}$ by the equivalence relation $(x.g,y) \sim (x,y.g)$ for all $x \in X$, $y \in Y$, and $g \in G$.

	Let us now explain how these constructions extend to morphisms, thus defining a pair of functors.
	Recall that any morphism of sheaves is completely determined on stalks.
	So a morphism of linear local systems $\phi : \cal{E} \to \cal{E}'$ induces a linear map on stalks $\phi_p : \cal{E}_p \to \cal{E}'_p$ which in turn induces a $G$-equivariant map $\hat{\phi}_p : GL (\CC^n, \cal{E}_p) \to GL (\CC^n, \cal{E}'_p)$ given by post-composition with $\phi_p$ (i.e., $f \mapsto \phi_p \circ f$), which in turn induces a sheaf morphism $\hat{\phi} : GL (\CC^n, \cal{E}) \to GL (\CC^n, \cal{E}')$; i.e., $\hat{\phi} : \hat{\cal{E}} \to \hat{\cal{E}}'$.
	Conversely, any morphism of $G$-local systems $\hat{\phi} : \hat{\cal{E}} \to \hat{\cal{E}}'$ induces a $G$-equivariant map on stalks $\hat{\phi}_p : \hat{\cal{E}}_p \to \hat{\cal{E}}'_p$ which in turn induces a linear map $\phi_p : \hat{\cal{E}}_p \times_G \CC^n \to \hat{\cal{E}}'_p \times_G \CC^n$ given by $(f,v) \mapsto \big(\hat{\phi}_p (f), v \big)$ and therefore a linear morphism $\phi : \cal{E} \to \cal{E}'$.
	To see that $\phi_p$ is indeed a linear map, recall that $\hat{\cal{E}}'_p$ is a $G$-torsor, so for any $f \in \hat{\cal{E}}_p$ there is a unique element $g \in G$ such that $\hat{\phi}_p (f) = f.g = f \circ g$, and so $\big(\hat{\phi}_p (f), v \big) = (f.g,v) \sim (f, vg)$.
	
	Finally, we explain why these two functors determine an equivalence of categories. For the sake of lighter notations, let us denote by $\mathsf{C}$ and $D$ the categories of linear local systems on $\Sigma$ and the one of $G$-local systems on $\Sigma$ respectively, by $\mathcal{F}$ and $\mathcal{G}$ the functors 
	\begin{eqn}
		{\mathcal F}:\catname{C}  \mapstoo \catname{D} 
		\qqtext{and}
		{\mathcal G}:\catname{D}  \mapstoo \catname{C}.
	\end{eqn}
	
	If $\cal{E}$ is a linear local system and $\hat{\cal{E}} = GL (\CC^n, \cal{E})$ is the associated $G$-local system, then the associated linear local system $\cal{E}' \coleq \hat{\cal{E}} \times_G \underline{\CC}^n_\Sigma$ is canonically isomorphic to $\cal{E}$.
	The reason is that, at any $p \in \Sigma$, the stalk $\cal{E}'_p = \hat{\cal{E}}_p \times_G \CC^n = GL (\CC^n, \cal{E}_p) \times_G \CC^n$ is canonically isomorphic to $\cal{E}_p$ because, for any $n$-dimensional vector space $E$, there is a canonical isomorphism $GL (\CC^n, E) \times_G \CC^n \iso E$ given by $(f,v) \mapsto f (v)$.
	
	This isomorphism provides natural transformations $\eta: \operatorname{Id}_{\C} \Rightarrow \mathcal{G} \circ \mathcal{F}$ and $\zeta: \operatorname{Id}_{D} \Rightarrow \mathcal{F} \circ \mathcal{G}$.
\end{proof}

Finally, we also mention that any diffeomorphism $f : \Sigma' \to \Sigma$ --- or, more generally, any homotopy equivalence --- induces an equivalence $f^\ast : \catname{Loc}_G (\Sigma) \iso \catname{Loc}_G (\Sigma')$ by pullback.
In particular, any local system on $\Sigma$ restricts to a local system on the interior $\Sigma^\circ \subset \Sigma$, and this restriction functor is actually an equivalence of categories $\catname{Loc}_G (\Sigma) \iso \catname{Loc}_G (\Sigma^\circ)$ which has a canonical inverse.

\begin{lemma}
	\label{251111194051}
	For any surface $\Sigma$, any local system $\cal{E}^\circ$ on the interior $\Sigma^\circ$ has a canonical extension to a local system $\cal{E}$ on $\Sigma$, thus yielding an inverse pair of restriction-extension functors:
	\begin{eqn}
		\catname{Loc}_G (\Sigma) \overset{\,\smash{
				\raisebox{-0.55ex}{\ensuremath{\scriptstyle\sim}}}
			\,}{\rightleftharpoons} \catname{Loc}_G (\Sigma^\circ).
	\end{eqn}
\end{lemma}

\begin{proof}
	The canonical extension $\cal{E}$ to $\Sigma$ of a local system $\cal{E}^\circ$ on $\Sigma^\circ$ is defined as follows.
	For every boundary point $p \in \del \Sigma$, the stalk $\cal{E}_p$ of $\cal{E}$ at $p$ is defined as the restriction of $\cal{E}^\circ$ to the germ of a sectorial neighbourhood of $p$.
	Recall that, since $\Sigma$ is a manifold with boundary, a \textit{sectorial neighbourhood} of a boundary point $p$ is the intersection with $\Sigma^\circ$ of a simply connected open boundary chart containing $p$.
	More precisely, $\cal{E}_p \coleq \varinjlim \cal{E}^\circ (U)$ is the direct limit of vector spaces $\cal{E}^\circ (U)$ indexed over all sectorial neighbourhoods $U \subset \Sigma$ of $p$.
	Notice that since $\cal{E}^\circ$ is locally constant, the vector spaces $\cal{E}^\circ (U)$ stabilise in this limit, and so there is a canonical linear isomorphism $\cal{E}_p \iso \cal{E}^\circ (U)$ for any sectorial neighbourhood $U$ of $p$.
\end{proof}

\subsection{Groupoid Representations and Holonomy}
\label{251209141021}

A \dfn{representation} of a groupoid $\cal{G} \rightrightarrows X$ is by definition a pair $(\E, \Phi)$ consisting of a complex vector bundle $\E \to X$ and a groupoid morphism $\Phi : \cal{G} \to \GL (\E)$.
We usually write $\Phi_\gamma \coleq \Phi (\gamma) \in \GL (\E)$ for any $\gamma \in \cal{G}$.
A morphism $\phi : (\E, \Phi) \to (\E', \Phi')$ is a bundle map $\phi : \E \to \E'$ which intertwines the homomorphisms $\Phi, \Phi'$ in the following sense.
For any $\gamma \in \cal{G}$ with source and target $p,q \in X$, we have the following commutative diagram of maps between the fibres of $\E$ and $\E'$:
\begin{eqntag}
	\label{250911161645}
	\begin{tikzcd}
		E_p
		\ar[r, "\Phi_\gamma"]
		\ar[d, "\phi_p"']
		&	E_q
		\ar[d, "\phi_q"]
		\\
		E'_p
		\ar[r, "\Phi'_\gamma"']
		&	E'_q
	\end{tikzcd}
	\qqtext{i.e.}
	\phi_q \circ \Phi_\gamma = \Phi'_\gamma \circ \phi_p.
\end{eqntag}
Moreover, if $\mathcal{G}$ is a topological (resp. smooth, holomorphic, algebraic) Lie groupoid, then the vector bundles $\E$ and the bundle maps $\phi$ are required to be continuous (resp. smooth, holomorphic, algebraic).
We denote the groupoid of representations of $\cal{G} \rightrightarrows X$ of rank $n$ and the corresponding moduli space by
\begin{eqn}
	\catname{Rep}_G (\cal{G})
	\qqtext{and}
	\text{$\frak{Rep}_G$} (\cal{G})
	\coleq \catname{Rep}_G (\cal{G}) \big/{\sim}.
\end{eqn}

In the case of the fundamental groupoid, we give the following

\begin{definition}
	\label{251209142102}
	A \dfn{representation} (or, more precisely, a \dfn{linear representation} of rank $n$, or a \dfn{$G$-representation}) of the fundamental groupoid $\Pi_1 (\Sigma) \rightrightarrows \Sigma$ is a pair $(\E, \Phi)$ consisting of a topological complex vector bundle $\E \to \Sigma$ of rank $n$ and a (continuous) groupoid homomorphism $\Phi : \Pi_1 (\Sigma) \to \GL (\E)$.
\end{definition}


Likewise, since the fundamental group is a groupoid over a single point and a vector bundle over a single point is a vector space, we recover representations of the fundamental group as groupoid representations.
Recall that, given a point $p \in \Sigma$, a \dfn{representation} of the fundamental group $\pi_1 (\Sigma, p)$ is a pair $(E, \varphi)$ consisting of an $n$-dimensional complex vector space $E$ and a group homomorphism $\varphi : \pi_1 (\Sigma, p) \to \GL (E)$.

Since the fundamental group $\pi_1 (\Sigma, p)$ is the restriction of the fundamental groupoid $\Pi_1 (\Sigma)$ to the point $p$, any representation $(\E, \Phi)$ of $\Pi_1 (\Sigma)$ restricts to a representation of $\pi_1 (\Sigma, p)$ by taking $E$ to be the fibre $\E_p$ of $\E$ at $p$ and $\varphi$ to be the restriction of $\Phi$.
This yields a \dfn{restriction functor} on representations:
\begin{equation*}
    \R : \catname{Rep}_G \big( \Pi_1 (\Sigma) \big) \to \catname{Rep}_G \big( \pi_1 (\Sigma, p) \big)
\end{equation*}
sending
\begin{equation*}
    (\E, \Phi) \mapsto (E, \varphi) \coleq \Big( \E_p, \Phi \big|_{\pi_1 (\Sigma, p)} \Big).
\end{equation*}
This functor turns out to be an equivalence of categories with a preferred explicit inverse equivalence that gives a convenient way to extend representations from the fundamental group $\pi_1 (\Sigma, p)$ to the fundamental groupoid $\Pi_1 (\Sigma)$.
We will state and prove this equivalence as part of a more general result (\autoref{250911161735}) after discussing the holonomy representation.

\subsubsection{Holonomy representation}
Any local system $\cal{E}$ on $\Sigma$ induces a representation $(\E, \Phi)$ of the fundamental groupoid $\Pi_1 (\Sigma)$ by taking the holonomy along paths, and vice versa by taking the deck-transformation-equivariant sections of a constant sheaf on the universal cover of $\Sigma$.
Let us briefly recall how this works.

The \dfn{holonomy} of a local system $\cal{E}$ along a path $\gamma : [0,1] \to \Sigma$ from $p$ to $q \in \Sigma$ is by definition the linear isomorphism between on stalks,
\begin{eqntag}
	\label{250930122017}
	\op{hol}_\gamma : \cal{E}_p \iso \cal{E}_q,
\end{eqntag}
defined as the composition of canonical isomorphisms 
$$
{\cal{E}_p \iso \gamma^{-1} \cal{E} \big( [0,1] \big) \iso \cal{E}_q}.
$$
When $p = q$, the isomorphism $\op{hol}_\gamma$ is usually called the \dfn{monodromy} of $\cal{E}$ along a loop $\gamma$ based at $p$ and denoted by $\op{mon}_\gamma$.
The isomorphisms defining $\op{hol}_\gamma$ depend only on the homotopy class of $\gamma$ relative its endpoint and they exist by virtue of the fact that the inverse-image sheaf $\gamma^{-1} \cal{E}$ on the interval $[0,1]$ is constant because the interval is contractible.
Furthermore, observe that any sheaf morphism $\phi : \cal{E} \to \cal{E}'$ defines a linear map $\gamma^{-1} \cal{E} \big( [0,1] \big) \to \gamma^{-1} \cal{E}' \big( [0,1] \big)$, and therefore the holonomies of $\cal{E}$ and $\cal{E}'$ along any path $\gamma$ from $p$ to $q$ are related by the commutative diagram
\begin{eqntag}
	\label{250924114505}
	\begin{tikzcd}
		\cal{E}_p
		\ar[r, "\op{hol}_\gamma"]
		\ar[d, "\phi_p"']
		&		\cal{E}_q
		\ar[d, "\phi_q"]
		\\
		\cal{E}'_p
		\ar[r, "\op{hol}'_\gamma"']
		&		\cal{E}'_q.
	\end{tikzcd}
\end{eqntag}

Given a local system $\cal{E}$, let $\cal{O}_\Sigma$ be the sheaf of continuous functions on $\Sigma$ and consider the vector bundle $\E \coleq \cal{E} \otimes \cal{O}_\Sigma$ (where, in line with common practice, we treat vector bundles interchangeably with their sheaves of sections).
The fibre $\E_p$ of this bundle at any point $p \in \Sigma$ is canonically isomorphic to the stalk $\cal{E}_p$.
Then the holonomy of $\cal{E}$ defines a groupoid homomorphism 
\begin{eqn}
	\op{hol} : \Pi_1 (\Sigma) \to \GL (\E)
	\qtext{sending}
	\gamma \mapsto \op{hol}_\gamma.
\end{eqn}
The result is a representation $(\E,\Phi) \coleq (\cal{E} \otimes \cal{O}_\Sigma, \op{hol})$ of the fundamental groupoid $\Pi_1 (\Sigma)$ called the \dfn{holonomy representation} of $\cal{E}$.
If, in addition, we fix a basepoint $p \in \Sigma$, then the restriction of the holonomy representation of $\cal{E}$ to the fundamental group $\pi_1 (\Sigma, p)$ determines a representation $(E, \Phi) \coleq (\cal{E}_p, \op{mon}_p)$ with 
\begin{eqn}
	\op{mon}_p : \pi_1 (\Sigma, p) \to GL(\cal{E}_p)
	\qtext{sending}
	\gamma \mapsto \op{mon}_\gamma
\end{eqn}
called the \dfn{monodromy representation} of $\cal{E}$ at $p$.

\begin{proposition}
	\label{250911161735}
	There is an equivalence of categories
	\begin{eqn}
		\op{Hol} : \catname{Loc}_G (\Sigma) \iso \catname{Rep}_G \big( \Pi_1 (\Sigma) \big)
	\end{eqn}
	sending
	\begin{eqn}
		\cal{E} \mapstoo (\E, \Phi) \coleq (\cal{E} \otimes \cal{O}_\Sigma, \op{hol})
	\end{eqn}
	called the \dfn{holonomy representation functor} which associates a local system to its holonomy representation.
	Furthermore, for any $p \in \Sigma$, the composition of the restriction functor with the holonomy representation functor is an equivalence of categories
	\begin{eqn}
		\op{Mon}_p : \catname{Loc}_G (\Sigma) \iso \catname{Rep}_G \big( \pi_1 (\Sigma, p) \big)
	\end{eqn}
	sending
	\begin{eqn}
		\cal{E} \mapstoo (E, \varphi) \coleq (\cal{E}_p, \op{mon}_p)
	\end{eqn}
	called the \dfn{monodromy representation functor}.
	Consequently, the restriction of representations from the fundamental groupoid $\Pi_1 (\Sigma)$ to the fundamental group $\pi_1 (\Sigma, p)$ is also an equivalence of categories:
	\begin{eqn}
		\R : \catname{Rep}_G \big( \Pi_1 (\Sigma) \big) \iso \catname{Rep}_G \big( \pi_1 (\Sigma, p) \big)
	\end{eqn}
	sending
	\begin{eqn}
		\hspace{2cm} (\E, \Phi) \mapsto (E, \varphi) \coleq \Big( \E_p, \Phi \big|_{\pi_1 (\Sigma, p)} \Big).
	\end{eqn}
\end{proposition}

Explicitly, an inverse to the monodromy representation functor $\op{Mon}_p$ is given by 
\begin{eqn}
	\op{Mon}_p^{-1} : (E, \varphi) \mapsto \cal{E} \coleq \big(\tilde{\Sigma} \times E \big) / \pi_1 (\Sigma, p)
\end{eqn}
where $\tilde{\Sigma} \to \Sigma$ is the universal cover based at $p$ and the quotient is by the diagonal action of $\pi_1 (\Sigma, p)$ by deck transformations on $\tilde{\Sigma}$ and by the representation $\varphi$ on $E$.
Similarly, an inverse functor to the restriction functor $\R$ lifts a representation $(E, \varphi)$ of $\pi_1 (\Sigma, p)$ to the holonomy representation of the local system $\cal{E} = \big(\tilde{\Sigma} \times E \big) / \pi_1 (\Sigma, p)$; i.e.,
\begin{eqn}
	\op{\R}^{-1} \coleq \op{Hol} \circ \op{Mon}_p^{-1}.
\end{eqn}
These notations will be fully explained in the course of the proof (specifically, see \eqref{251008145350}) of this classical result.
For clarity, we remark that the holonomy and monodromy representation functors fit into the commutative diagram
\begin{eqn}
	\begin{tikzcd}
		\catname{Loc}_G (\Sigma)
		\ar[r, "\op{Hol}"]
		\ar[dr, "\op{Mon}_p"']
		&		\catname{Rep}_G \big( \Pi_1 (\Sigma) \big)
		\ar[d, "\R"]
		\\
		&		\catname{Rep}_G \big( \pi_1 (\Sigma, p) \big).
	\end{tikzcd}	
\end{eqn}
We also state the following immediate corollary of \autoref{250911161735}.

\begin{corollary}
	\label{250912175127}
	There are canonical bijections between the moduli spaces of local systems, of representations of the fundamental groupoid $\Pi_1 (\Sigma)$, and of representations of the fundamental group $\pi_1 (\Sigma, p)$ for any $p \in \Sigma$:
	\begin{eqn}
		\frak{Loc}_G (\Sigma) 
		\cong \frak{Rep}_G \big( \Pi_1 (\Sigma) \big)
		\cong \frak{Rep}_G \big( \pi_1 (\Sigma, p) \big),
	\end{eqn}
	given by $[\cal{E}] \mapsto [\op{Hol} (\cal{E})] \mapsto [\op{Mon}_p (\cal{E})]$.
\end{corollary}

Although \autoref{250911161735} is classical, we have found no single reference that cleanly addresses all of its aspects.
For this reason, and because we shall use parts of this proof in the next section, we supply a detailed proof in Appendix \ref{sec:A}.

\begin{remark}
	\label{250930121926}
	Let us explain in more detail the definition \eqref{250930122017} of the holonomy of a local system $\cal{E}$ along a path $\gamma$.
	First, the inverse image $\gamma^{-1} \cal{E}$ is by definition the sheaf on the interval $[0,1]$ associated with the presheaf that sends open subsets $I \subset [0,1]$ to the direct limit of $\cal{E} (U)$ over open subsets $U \subset \Sigma$ containing the image $\gamma (I)$; in symbols, $I \mapsto \varinjlim \cal{E} (U)$.
	Concretely, this implies that for any open subset $I \subset [0,1]$, there is an open neighbourhood $U \subset \Sigma$ of $\gamma (I)$ together with a linear map $\cal{E} (U) \to \gamma^{-1} \cal{E} (I)$.
	
	Since $\cal{E}$ is locally constant, it follows that the inverse image sheaf $\gamma^{-1} \cal{E}$ is locally constant, too: for any sufficiently small open subset $I \subset [0,1]$, there is an open neighbourhood $U$ of $\gamma (I)$ such that the restriction $\cal{E} |_U$ is a constant sheaf on $U$, which implies that the restriction $\gamma^{-1} \cal{E} \big|_I$ is a constant sheaf on $I$ with a canonical isomorphism $\cal{E} (U) \iso \gamma^{-1} \cal{E} (I)$.
	But since $[0,1]$ is contractible, the locally constant sheaf $\gamma^{-1} \cal{E}$ is in fact a constant sheaf on $[0,1]$, which means means there is a canonical isomorphism $\gamma^{-1} \cal{E} (I) \iso \gamma^{-1} \cal{E} \big( [0,1] \big)$ for any open subset $I \subset [0,1]$.
	
	Next, any half-open interval $[0,\epsilon)$ is an open subset of $[0,1]$, so any open neighbourhood $U$ of the image $\gamma \big( [0,\epsilon) \big)$ is in particular an open neighbourhood of $p = \gamma (0)$.
	But since $\cal{E}_p$ is the stalk of the locally constant sheaf $\cal{E}$ at $p$, there is by definition an isomorphism $\cal{E}_p \iso \cal{E} (U)$ for a sufficiently small $U$.
	Consequently, we find the canonical isomorphism $\cal{E}_p \iso \gamma^{-1} \cal{E} \big( [0,1] \big)$ used in the definition of the holonomy $\op{hol}_\gamma$.
\end{remark}

\subsection{Character Variety}
\label{251217152313}

The advantage of fixing a point $p \in \Sigma$ and reducing our attention to the representations of the fundamental group $\pi_1 (\Sigma, p)$ is that it can be used to give an explicit, finite-dimensional model of the moduli space $\frak{Loc}_G (\Sigma)$.

A \dfn{frame} (or \dfn{framing}) on an $n$-dimensional vector space $E$ is an isomorphism $\beta : E \iso \CC^n$.
Equivalently, a frame is a choice of an ordered basis of $E$ or of a group isomorphism $\GL(E) \iso G$.
Therefore, given any representation $(E, \Phi)$ of $\pi_1 (\Sigma, p)$, a frame on $E$ induces a group homomorphism $\rho : \pi_1 (\Sigma, p) \to G$ defined by $\rho : \gamma \mapsto \rho_\gamma \coleq \beta \circ \Phi_\gamma \circ \beta^{-1}$.
Conversely, any $\rho : \pi_1 (\Sigma, p) \to G$ gives rise to the representation $(\CC^n, \rho)$.

If $(E, \Phi)$ is a representation with frame $\beta$, then changing the frame on $E$ has the effect of conjugating the homomorphism $\rho$ by an element of $G$.
Namely, $\rho' : \gamma \mapsto \rho'_\gamma = g \rho_\gamma g^{-1}$ whenever $\beta' \circ \beta^{-1} = g \in G$.
This formula defines a conjugation action of $G$ on the set of group homomorphisms, called the \dfn{representation variety},
\begin{eqn}
	R_G (\Sigma, p) \coleq \Hom \big( \pi_1 (\Sigma, p), G \big).
\end{eqn}
Likewise, if $\phi : (E, \Phi) \iso (E', \Phi')$ is an isomorphism and $\beta, \beta'$ are frames on $E, E'$ that determine $\rho, \rho' \in R_G (\Sigma, p)$, then $\phi$ gives a group element $g \coleq \beta' \circ \phi \circ \beta^{-1} \in G$ which relates $\rho$ and $\rho$ by the same conjugation action: $\rho'_\gamma = g \rho_\gamma g^{-1}$ for all $\gamma \in \pi_1 (\Sigma, p)$.

This defines an inclusion functor from the action groupoid $G \ltimes R_G (\Sigma, p)$ to the groupoid of representations of $\pi_1 (\Sigma, p)$, and this functor turns out to be an equivalence.

\begin{proposition}
	\label{251009141029}
	There is a canonical equivalence of categories
	\begin{eqn}
		\catname{Rep}_G \big( \pi_1 (\Sigma, p) \big) 
		\cong G \ltimes R_G (\Sigma, p). 
	\end{eqn}
	Furthermore, for any $p, p' \in \Sigma$, there is an equivalence of categories $G \ltimes R_G (\Sigma, p) \cong G \ltimes R_G (\Sigma, p')$.
\end{proposition}

This equivalence is given by the inclusion functor
\begin{eqn}
	G \ltimes R_G (\Sigma, p) \iso \catname{Rep}_G \big( \pi_1 (\Sigma, p) \big)
	\qtext{sending}
	\rho \mapstoo (E,\Phi) \coleq (\CC^n, \rho),
\end{eqn}
and viewing any $g \in G$ as an isomorphism $g : \CC^n \to \CC^n$.
An inverse functor requires a choice of frame $\beta : E \iso \CC^n$ for every vector space $E \in \catname{Vec}_G$.
Once chosen, it is given on objects by $(E,\Phi) \mapsto \rho$ where $\rho_\gamma \coleq \beta \circ \Phi_\gamma \circ \beta^{-1}$ for all $\gamma \in \pi_1 (\Sigma, p)$, and on morphisms by $\phi \mapsto g \coleq \beta' \circ \phi \circ \beta^{-1}$.

\begin{proof}[Proof of \autoref{251009141029}]
	We just need to check that any inverse functor defined as suggested above is indeed inverse to the inclusion functor.
	That it is a left inverse is obvious especially if we choose the identity frame on $\CC^n$.
	For a right inverse, if $(E,\Phi) {\mapsto} \rho$ via a frame $\beta$, then $\beta$ in fact defines the desired isomorphism from $(E,\Phi)$ to $(\CC^n, \rho)$.
\end{proof}

If $\Sigma$ is compact, the representation variety $R_G (\Sigma, p)$ is a smooth affine algebraic variety \cite{MR1380633,MR2931326}.
If $\Sigma$ has genus $g$ and no boundary, then the representation variety has dimension $n^2 (2g-1) + 1$ when $g \geq 2$ and $(n^2 + n) g$ when $g = 0, 1$.
If $\Sigma$ has $s \geq 1$ holes then the representation variety has dimension $n^2 (2g + s - 1)$.

The affine algebraic group $G$ is acting algebraically on the affine algebraic variety $R_G (\Sigma, p)$, so the action groupoid $G \ltimes R_G (\Sigma, p)$ is algebraic.
Its moduli space is the well-known algebraic quotient stack called the \dfn{character variety}:
\begin{eqn}
	\mathfrak{X}_G (\Sigma) \coleq R_G (\Sigma, p) \big/ G.
\end{eqn}

\begin{corollary}
	\label{250912175812}
	There are canonical bijections between the moduli spaces of local systems, representations of the fundamental groupoid $\Pi_1 (\Sigma)$, representations of the fundamental group $\pi_1 (\Sigma, p)$ for any basepoint $p \in \Sigma$, and the character stack of $\Sigma$:
	\begin{eqn}
		\text{$\mathfrak{Loc}_G (\Sigma)$}
		\cong \text{$\mathfrak{Rep}_G \big( \Pi_1 (\Sigma) \big)$}
		\cong \text{$\mathfrak{Rep}_G \big( \pi_1 (\Sigma, p) \big)$}
		\cong \text{$\mathfrak{X}_G (\Sigma)$}.
	\end{eqn}
	Consequently, all these moduli spaces have a natural structure of an affine algebraic quotient stack.
\end{corollary}

Explicitly, the above bijection is defined as the composition 
\begin{eqns}
	\mathfrak{Loc}_G (\Sigma) 
	&\to \mathfrak{Rep}_G \big( \pi_1 (\Sigma, p) \big) 
	\to \text{$\Hom \big( \pi_1 (\Sigma, p), G) / G$}
	= \mathfrak{X}_G (\Sigma)
	\\
	[\cal{E}] 
	&\mapsto [\op{Mon}_p (\cal{E})] 
	= [(\cal{E}_p, \op{mon}_p)] 
	\mapsto [\rho]
\end{eqns}
where the last arrow involves choosing a frame on the vector space $\cal{E}_p$.

\subsection{Monodromy Data}
\label{251115110733}

Finally, if $\Sigma$ is compact of genus $g$, then it has at most finitely many boundary components.
Then the fundamental group $\pi_1 (\Sigma, p)$ is finitely generated by $2g + s$ generators with one relation, where $s \geq 0$ is the total number of punctures and boundary components of $\Sigma$, collectively referred to as \dfn{holes}.
In fact, if there is at least one hole (i.e., $s \geq 1$), then $\pi_1 (\Sigma, p)$ is a free group on $2g + s - 1$ generators.
A generating set can be obtained by selecting $s$ loops going round the holes of $\Sigma$ as well as $2g$ loops representing the $\rm{A}$ and $\rm{B}$-cycles.
More precisely:
\begin{enumerate}
	\item Choose any collection of $2g$ oriented loops $\alpha_1, \ldots, \alpha_g, \beta_1, \ldots, \beta_g$ based at $p$ that determine a basis of $\rm{A}$ and $\rm{B}$-cycles.
	Namely, if $\Sigma'$ denotes the closed surface obtained from $\Sigma$ by capping off all boundary components, then the fundamental group $\pi_1 (\Sigma', p)$ is generated by the $\alpha_i, \beta_i$ such that $\prod_{i=1}^g [\alpha_i, \beta_i] = 1$ where $[\alpha_i, \beta_i] \coleq \alpha_i \beta_i \alpha_i^{-1} \beta_i^{-1}$ is the commutator and $1 = 1_p$ is the constant path at $p$.
	\item If $s \geq 1$, enumerate all holes of $\Sigma$ as $1, \ldots, s$.
	For each $i = 1, \ldots, s$, let $\gamma_i$ be a positively-oriented loop based at $p$ going round only the $i$-th hole; i.e., $\gamma_i$ belongs to the conjugacy class in $\pi_1 (\Sigma, p)$ determined by the $i$-th hole.
\end{enumerate}

These generators can be chosen in such a way as to satisfy the relation
\begin{eqntag}
	\label{250922123329}
	\prod_{i=1}^{\substack{g\\\longleftarrow}} [\alpha_i, \beta_i] 
	\prod_{i=1}^{\substack{s\\\longleftarrow}} \gamma_i
	= 1.
\end{eqntag}
Here, the left arrow over the product sign indicates that the product is taken from right to left.
Therefore, the fundamental group has an explicit presentation
\begin{eqntag}
	\label{251005162302}
	\pi_1 (\Sigma, p)
	\cong \inner{ \alpha_i, \beta_i, \gamma_i ~\big|~ \text{ relation \eqref{250922123329}}}.
\end{eqntag}
The relation \eqref{250922123329} can be solved for any $\gamma_i$, but not for any of the $\alpha_i$ or $\beta_i$.
Consequently, if $s \geq 1$, then $\pi_1 (\Sigma, p)$ is isomorphic to the free group on $2g + s - 1$ generators.

Next, we introduce the following definition.

\begin{definition}
	\label{251005162057}
	For any integers $g,s \geq 0$, a \dfn{monodromy data} in $G$ of type $(g,s)$ is any collection of matrices $( \bm{\A}, \bm{\B}, \bm{\M} ) \subset G$ consisting of:
	\begin{itemize}
		\item $\bm{\A} \coleq \set{ \A_1, \ldots, \A_g}, \bm{\B} \coleq \set{ \B_1, \ldots, \B_g }$ called \dfn{period monodromies};
		\item $\bm{\M} \coleq \set{ \M_1, \ldots, \M_s}$	 called \dfn{standard monodromies};
	\end{itemize}
	and satisfying the following relation:
	\begin{eqntag}
		\label{251005161035}	
		\prod_{i=1}^{\substack{g\\\longleftarrow}} [\A_i, \B_i] 
		\prod_{i=1}^{\substack{s\\\longleftarrow}} \M_i = \idd.
	\end{eqntag}
	We denote the \dfn{space of monodromy data} in $G$ of type $(g,s)$ by 
	\begin{eqn}
		\sf{Mon}_G (g,s)
		\coleq \set{ (\bm{\A}, \bm{\B}, \bm{\M}) \in G^{2g + s} 
			~\Big|~ \textup{ relation \eqref{251005161035} } \big.}.
	\end{eqn}
\end{definition}

\begin{definition}
	\label{251216170752}
	To equip a surface $\Sigma$ of genus $g$ with $s$ holes with monodromy data is to choose a point $p\in\Sigma$ and an isomorphism
	\begin{eqn}
		R_G (\Sigma, p) \iso \sf{Mon}_G (g,s).
	\end{eqn}
	A choice of finite presentation of the fundamental group $\pi_1 (\Sigma, p)$ of the form \eqref{251005162302} determines such an isomorphism by associating monodromy data $( \bm{\A}, \bm{\B}, \bm{\M} ) \in \sf{Mon}_G (g,s)$ to the generators $(\bm{\alpha}, \bm{\beta}, \bm{\gamma})$ of $\pi_1 (\Sigma, p)$ like so:
	\begin{eqn}
		\alpha_i \mapsto \A_i ~,
		\qquad
		\beta_i \mapsto \B_i ~,
		\qquad
		\gamma_i \mapsto \M_i.
	\end{eqn}
\end{definition}


Note that if $s \geq 1$, then by eliminating one of the matrices $\M_i$ from the list $(\bm{\A}, \bm{\B}, \bm{\M})$ of generators destroys the relation \eqref{251005161035}.
This yields a bijection $\sf{Mon}_G (g,s) \cong G^{2g + s - 1}$ which shows that the set of monodromy data $\sf{Mon}_G (g,s)$ is a smooth affine algebraic variety of dimension $(2g + s - 1) n^2$ provided that $s \geq 1$.

The space of monodromy data $\sf{Mon}_G (g,s)$ is equipped with a natural left $G$-action given by simultaneous conjugation:
\begin{eqn}
	g . (\bm{\A}, \bm{\B}, \bm{\M}) = (g \bm{\A} g^{-1}, g \bm{\B} g^{-1}, g \bm{\M} g^{-1}).
\end{eqn}
Then we arrive at the following classical result.

\begin{proposition}
	\label{251109102652}
	Let $\Sigma$ be a surface of genus $g$ with $s$ holes.
	Then there is a non-canonical equivalence of categories between the groupoid of $G$-local systems on $\Sigma$ and an action groupoid over the space of monodromy data of type $(g,s)$:
	\begin{eqn}
		\catname{Loc}_G (\Sigma) \simeq G \ltimes \sf{Mon}_G (g,s).
	\end{eqn}
	To obtain such an equivalence, fix a point $p \in \Sigma$ and choose any presentation of the fundamental group $\pi_1 (\Sigma, p)$ of the form \eqref{251005162302}.
	Consequently, there is a bijection between the character variety and the moduli space of monodromy data:
	\begin{eqn}
		\text{$\mathfrak{X}_G (\Sigma)$} \cong \sf{Mon}_G (g,s) / G.
	\end{eqn}
	In particular, if $s \geq 1$, then $\mathfrak{X}_G (\Sigma) \cong G^{2g + s - 1} / G$.
	Consequently, the moduli space $\mathfrak{Loc}_G (\Sigma)$ of $G$-local systems on $\Sigma$ is an algebraic stack of dimension $(2g-2)n^2 + 2$ if $s = 0$ and $(2g+s-2)n^2 + 1$ if $s \geq 1$.
\end{proposition}

We stress that the above equivalence and bijection depend nontrivially on the chosen presentation \eqref{251005162302} of $\pi_1 (\Sigma, p)$.

\section{Surfaces with Marked Boundary }
\label{se:2}

Recall that a \dfn{bordered surface} is a connected, oriented, smooth real two-dimensional manifold-with-boundary $\Sigma$ which we also assume to be compact.
It has finitely many boundary components each of which is diffeomorphic to a circle.
A \dfn{marked bordered surface} is a pair consisting of a bordered surface equipped with a finite nonempty subset of \dfn{marked points} such that each boundary component contains at least one marked point.
An isomorphism of marked bordered surfaces is an orientation-preserving diffeomorphism that sends marked points to marked points.

\begin{remark}
	The terminology for \textit{marked bordered surfaces}  matches what seems to be a broadly accepted meaning in the literature, including Fomin-Shapiro-Thurston \cite[§2]{MR2448067}, Smith \cite{MR3416110}, Bridgeland-Smith \cite[§8]{MR3349833}, Allegretti-Bridgeland \cite[§3.6]{MR4155179}, Allegretti \cite[§2]{MR4205119}, to name a few.
	However, it corresponds to what Goncharov \cite[§1]{MR3702383} and Goncharov-Shen \cite[§2]{250619181547} call a ``decorated surface''.
\end{remark}

Marked points in the interior are called \dfn{punctures}.
For our purposes, it will be more natural to work not with marked bordered surfaces but with their real-oriented blowups along the punctures.
Concretely, the real-oriented blowup (see, e.g., \cite[§1.1]{240622121512}) of any marked bordered surface is a bordered surface $\Sigma$ obtained by replacing each puncture with a boundary circle containing no marked points whilst keeping the rest of the surface, including the boundary marked points, unchanged.
Thus, $\Sigma$ has marked points on some, but not all, boundary components.
Furthermore, sometimes it will be important for us to place an extra marked point on each unmarked boundary of $\Sigma$ yet still be able to distinguish the original boundary circles with marked points from the ones that arose from punctures.
This motivates the following definition.

\begin{definition}
	\label{defMarkedBorderedSurface}
	A \dfn{surface with marked boundary} 
	is a pair $(\Sigma, P)$ where $\Sigma$ is a compact bordered surface whose boundary components are partitioned as 
	\begin{eqn}
		I = I_\rm{sim} \sqcup I_\rm{reg} \sqcup I_\rm{irr}
	\end{eqn}
	into \dfn{simple}, \dfn{regular}, and \dfn{irregular} boundary circles, and where $P \subset \del \Sigma$ is a finite set of boundary \dfn{marked points}, itself partitioned as 
	\begin{eqn}
		P = P_\rm{I} \sqcup P_\rm{II}
	\end{eqn}
	into respectively  \dfn{primary marked points} and \dfn{secondary marked points}.
	Furthermore, every irregular boundary circle contains at least one primary marked point and no regular ones; every regular boundary circle contains exactly one secondary marked point and no primary ones; and every simple boundary circle contains no marked points at all.
	An isomorphism $f : (\Sigma, P) \to (\Sigma', P')$ of decorated surfaces is an orientation-preserving homeomorphism $f : \Sigma \to \Sigma'$ with the property that $f(P_\rm{I}) = P'_\rm{I}$ and $f(P_\rm{II}) = P'_\rm{II}$. 
\end{definition}

Often only the relative position of the marked points plays any role.
So for almost everything we say, we may allow the subset $P \subset \del \Sigma$ to be defined only up to isotopy, which concretely means that marked points can slide continuously along the boundary without crossing each other.
Hence, secondary marked points, being the sole occupiers of a regular boundary circle, are free to move around unconstrained.

\begin{remark}
	Fock-Goncharov \cite[Def 1.1]{MR2233852} use the term ``marked surface'' which in our language is a surface with marked boundary in which all boundaries are marked.
	
	Fock-Rosly \cite[§2]{MR1730456} and Fock-Goncharov \cite[§2]{MR2349682} also use the terms ``ciliated surface'', ``cilia'', and ``holes'' which in our language mean respectively a 
	surface with marked boundary with no secondary points, $P_\rm{II}=\emptyset$, \textit{primary marked points}, and \textit{simple boundary circles}.
	
	Chekhov-Mazzocco \cite{MR3746633}, Chekhov-Mazzocco-Rubtsov \cite{MR3932256,MR3802126}, and Dal Martello-Mazzocco \cite{MR4756422} use the terms ``bordered cusped Riemann surface'' or ``Riemann surface with bordered'' or ``boundary cusps'' for a manifold $\Sigma_{g,s,n}$ whose underlying topological surface is in our language a surface with marked boundary with at least one primary marked point, $P_\rm{I}\neq\emptyset$ and no secondary ones $P_\rm{II}=\emptyset$.
	
	Boalch \cite{MR4285677} considers a smooth compact complex curve $\Sigma$ with a finite subset $\bm{a} \subset \Sigma$ (using his notation).
	He then considers its real-oriented blowup $\hat{\Sigma} \to \Sigma$ and decorates its boundary $\del \hat{\Sigma}$ with two types of marked points called ``singular directions'' $\AA \subset \del \hat{\Sigma}$ and ``Stokes directions'' $\SS \subset \del \hat{\Sigma}$.
	If we drop the complex structure on $\hat{\Sigma}$ and consider just the underlying topological surface, then the pair $(\hat{\Sigma}, \AA)$ is an example of our surface with marked boundary with no secondary marked points.
	
	Jordan-Le-Schrader-Shapiro \cite[Def 1.1]{jordan2021quantumdecoratedcharacterstacks} use the term ``decorated surface'' in a quite different way but which can be directly related to ours.
	Namely, their notion involves regions in $\Sigma$ separated by simple curves, called ``walls'', that are closed or have endpoints on the boundary.
	Their attention is largely restricted to a subclass called \textit{simple decorated surfaces} \cite[Rem 1.5]{jordan2021quantumdecoratedcharacterstacks} which are equivalent to our our surface with marked boundary where all boundaries are marked.
	%
	
\end{remark}

The \dfn{boundary loop} based at a marked point $p \in P$ is the loop determined by the boundary circle containing $p$, oriented according to the natural orientation on the boundary $\del \Sigma$.
We also consider the punctured boundary $\hat{\del} \Sigma \coleq \del \Sigma \smallsetminus P_\rm{I}$, punctured along the primary marked points only.
It is a union of circles and open intervals.
A \dfn{boundary segment} $\delta \subset \del \Sigma$ is the closure of a simply connected component of the punctured boundary $\hat{\del} \Sigma$.
It has a natural orientation matching the orientation of the boundary $\del \Sigma$.
Thus, it is a path $\delta : [0,1] \to \del \Sigma$ completely determined by its endpoints which are two (not necessarily distinct) primary marked points $p_1 = \delta (0)$ and $p_2 = \delta (1) \in P_\rm{I}$.
For this reason, we may sometimes denote this boundary segment using the interval notation $[p_1, p_2]$.
We denote the set of all boundary segments of $(\Sigma, P)$ by
\begin{eqn}
	\frak{A} 
	= \frak{A} (\Sigma, P) \coleq \set{ \text{$\delta$ boundary segments} }.
\end{eqn}

\begin{remark}
	Note that Goncharov-Shen introduce the term ``coloured decorated surface'' which in our language corresponds to a triple $(\Sigma, P, \frak{A}_0)$ consisting of a surface with marked boundary $(\Sigma, P)$ and a sub-collection $\frak{A}_0 \subset \frak{A}$ of \textit{boundary segments} (which they call ``coloured boundary intervals'').
\end{remark}

A surface with marked boundary $(\Sigma, P)$ is completely determined up to isomorphism by the genus $g$ of $\Sigma$; an ordered collection $\dfn{s} \coleq (l,r,d)$ of non-negative integers equal respectively to the number of unmarked, marked by secondary points, and marked by primary points boundaries; and an unordered collection $\bm{m} \coleq \set{ m_1, \ldots, m_d }$ of strictly positive integers equal to the number of primary marked points per boundary circle marked by primary points.
In total, $(\Sigma, P)$ has $s = |\bm{s}| \coleq l + r + d$ holes, $r$ secondary marked points, $m = |\bm{m}| \coleq m_1 + \ldots + m_d$ primary marked points.
Also, $|P| = m + r$. These considerations justify the following

\begin{definition}
	A surface with marked boundary $(\Sigma, P)$ is called   of \dfn{type} $(g,\bm{s},\bm{m})$, where $\dfn{s} \coleq (l,r,d)$ and $\bm{m} \coleq (m_1 , \ldots , m_d)$, if it has $l$ simple boundaries, $r$ regular boundaries,  $d$ irregular boundaries  and $m_i$ primary marked points on the $i$-th irregular boundary.
\end{definition}

\subsection{Discrete Fundamental Groupoid}
\label{251214205432}

Now, let $(\Sigma, P)$ be a surface with marked boundary such that $P\neq\emptyset$.

\begin{definition}
	\label{250912192631}
	The \dfn{discrete fundamental groupoid} of a surface with marked boundary $(\Sigma, P)$ is the restriction of the fundamental groupoid $\Pi_1 (\Sigma)$ to $P$:
	\begin{eqn}
		\pi_1 (\Sigma, P) \coleq \Pi_1 (\Sigma) \big|_P = \rm{s}^{-1} (P) \cap \rm{t}^{-1} (P) \subset \Pi_1 (\Sigma).
	\end{eqn}
\end{definition}

Note that the restricted fundamental groupoid $\Pi_1 (\Sigma) \big|_P \coleq \rm{s}^{-1} (P) \cap \rm{t}^{-1} (P)$ makes sense as a groupoid over any subset $P \subset \Sigma$, not necessarily finite.
However, our interest lies with surfaces  with marked boundary: the upshot is that, just as every surface has a naturally associated \textit{continuous} fundamental groupoid, every  surface  with marked boundary has a naturally associated \textit{discrete} fundamental groupoid.

We call $\pi_1 (\Sigma, P)$ \textit{discrete} because, in contrast to the \textit{continuous} fundamental groupoid $\Pi_1 (\Sigma)$, it has finitely many objects $P$ and any two objects are related by a countable collection of arrows.
More accurately, the subspace topology on $\pi_1 (\Sigma, P)$ inherited from $\Pi_1 (\Sigma)$ is the discrete topology.
Even more importantly, as we argue next, the discrete fundamental groupoid $\pi_1 (\Sigma, P)$ is a \textit{finitely-generated groupoid}, which means we can find a finite collection of paths $\gamma_i \in \pi_1 (\Sigma, P)$ such that any path $\gamma \in \pi_1 (\Sigma, P)$ can be written as a finite composition of the $\gamma_i$.

\begin{remark}
	The discrete fundamental groupoid $\pi_1 (\Sigma, P)$ has been considered by Li--Bland-\v{S}evera \cite[§4]{MR3424475} who denoted it by $\Pi_1 (\Sigma, \V)$, and by Chekhov-Mazzocco-Rubtsov \cite[§4]{MR3802126} who called it the ``fundamental groupoid of arcs'' and denoted it by $\frak{U}$.
	It also appears in Dal Martello-Mazzocco \cite[§3]{MR4756422} who denote it by $\pi_\frak{a} (\Sigma_{g,s,m})$.
\end{remark}

\begin{lemma}
	\label{251217105341}
	For  surface with marked boundary $(\Sigma,P)$, where $\Sigma$ has  genus $g$ and $s$ holes, the discrete fundamental groupoid $\pi_1 (\Sigma, P)$ is finitely generated by $2g + s + |P| - 1$ elements and one relation.
	Moreover, if $s \geq 1$, then $\pi_1 (\Sigma, P)$ is freely generated by ${2g + s + |P| - 2}$ elements.
	Any choice of such free generators determines (up to isotopy) a graph $\Gamma$ embedded into $\Sigma$ with vertices $P$ such that the discrete fundamental groupoid $\pi_1 (\Sigma, P)$ is isomorphic to the free groupoid $\Pi \Gamma \rightrightarrows P$ generated by $\Gamma$.
\end{lemma}

\begin{proof}
	Select any point $p_0 \in P$.
	Recall that the fundamental group $\pi_1 (\Sigma, p_0)$ is finitely generated by $2g + s$ elements satisfying one relation, and that it is in fact isomorphic to the free group on $2g + s - 1$ generators whenever $s \geq 1$.
	Let us take any such generating set.
	Next, for every point $p \in P \smallsetminus p_0$, choose a path from $p_0$ to $p_1$.
	There are a total of $|P| - 1$ such paths.
	Then the resulting collection generates $\pi_1 (\Sigma, P)$ and satisfies one relation.
\end{proof}

\subsubsection{Finite Presentation}
\label{251218133944}

We now give a step-by-step prescription of choosing a convenient finite generating set for the discrete fundamental groupoid $\pi_1 (\Sigma, P )$ of any surface with marked boundary $(\Sigma, P)$.
This procedure specialises directly to the classical setting without marked points described in \autoref{251115110733} by simply choosing the base point of the fundamental group  $p_0$ to lie on a boundary of $\Sigma$.

Informally speaking, we first select the \textit{main} basepoint $p_0 \in P$.
Then we select one distinguished marked point from each marked boundary circle and choose paths to connect them to the main basepoint $p_0$.
Then the basis of generators is made up of these paths as well as all the boundary segments connecting the primary marked points, all the loops based at the main basepoint that go round the un-marked boundaries, and finally the loops that represent the $A$- and $B$-cycles.
More precisely, our recipe is as follows; see \autoref{251218221517}.

Let $(\Sigma,P)$ be a surface with marked boundary of type $({\bf s},{\bf m})$, then we make the following choices.

\begin{figure}[t]
	\centering
	\includegraphics[width=0.9\textwidth]{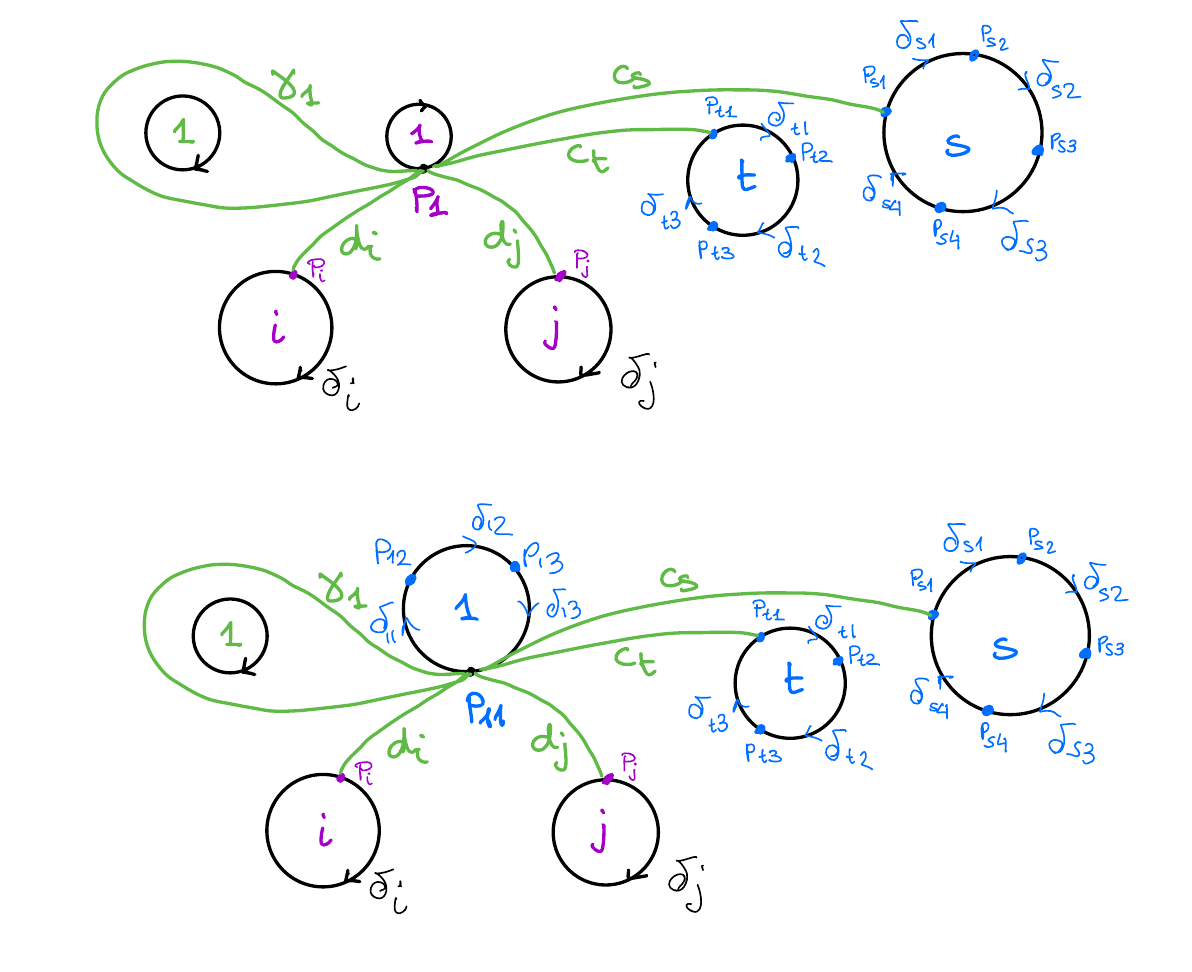}
	\caption{In the top picture the main point is a primary point. In the bottom one, the main point is a secondary point.}
	\label{251218221517}
\end{figure}

\begin{enumerate}
	\item Select any main basepoint $p_0 \in P$ and choose any collection of $2g$ oriented loops $\alpha_1, \ldots, \alpha_g, \beta_1, \ldots, \beta_g$ based at $p_0$ that determine a basis of $A$- and $B$-cycles of $\Sigma$.

	\item Enumerate all simple boundary circles as $1, \ldots, l$.
	For each $i = 1, \ldots, l$, let $\gamma_i$ be a positively-oriented loop based at $p_0$ going round only the $i$-th unmarked boundary circle; i.e., $\gamma_i$ belongs to the conjugacy class in $\pi_1 (\Sigma, p_0)$ determined by the $i$-th unmarked  boundary component.
	\item Enumerate all regular boundary circles  as $1, \ldots, r$, and label the corresponding secondary marked points as $p_1, \ldots, p_r$.
	Let $\delta_i$ be the positively-oriented boundary loop based at $p_{i}$, for each $i = 1, \ldots, r$.
	\item For each $i = 1, \ldots, r$, choose a path $d_i$ with source $p_0$ and target $p_{i}$.
	In addition, if $p_0 \in P_\rm{II}$, ensure that $p_1 = p_0$ and that $d_1 = 1$ is the constant path at $p_0$.
	\item Enumerate all irregular boundary circles as $1, \ldots, d$.
	For each $i = 1, \ldots, d$, select one distinguished marked point on the $i$-th irregular boundary circle and call it $p_{i,1}$.
	Let $m_i \geq 1$ be the number of irregular marked points on the $i$-th irregular boundary circle, so that $m_1 + \ldots + m_d = m$.
	Enumerate all the $m_i$ marked points belonging to this boundary circle as $p_{i,j}$ for $j = 1, \ldots, m_i$ in a consecutive fashion starting with $p_{i,1}$ and following the positive orientation of the boundary circle.
	Let $\delta_{i,j}$ be the positively-oriented boundary segment with source $p_{i,j}$ and target $p_{i,j+1}$, where $j$ is understood mod $m_i + 1$.
	\item For each $i = 1, \ldots, d$, choose a path $c_i$ with source $p_0$ and target $p_{i,1}$.
	In addition, if $p_0 \in P_\rm{I}$, ensure that $p_{1,1} = p_0$ and that $c_1 = 1$ is the constant path at $p_0$.
\end{enumerate}

In total, this procedure selects $2g + s +  r + m-1$ paths $\alpha_i, \beta_i, \gamma_i, \delta_i, \delta_{i,j}, c_i, d_i$ which we claim generate the discrete fundamental groupoid $\pi_1 (\Sigma, P)$.
To see this, note if we (right) conjugate each boundary loop $\delta_i$ by the connecting path $d_i$, we get a positively-oriented loop based at $p_0$ going round the $i$-th regular boundary circle.
If we concatenate the boundary segments $\delta_{i,j}$ into the boundary circle for each $i$, and then (right) conjugate by the connecting path $c_i$, we likewise obtain a positively-oriented loop based at $p_0$ going round the $i$-th irregular boundary circle.
More precisely, let us introduce the following notation:
\begin{eqn}
	\begin{aligned}
		\gamma_i^\rm{II} &\coleq d_i^{-1} \delta_i^\rm{II} d_i
		~,\qquad
		&\delta^\rm{II}_i &\coleq \delta_i
		~,\qquad
		&i &= 1, \ldots, r;
		\\
		\gamma_i^\rm{I} &\coleq c_i^{-1} \delta_i^\rm{I} c_i
		~,\qquad
		&\delta_i^\rm{I} &\coleq \delta_{i,m_i} \cdots \delta_{i,1}
		~,\qquad
		&i &= 1, \ldots, d.
	\end{aligned}
\end{eqn}
We find that the above procedure in particular yields $2g + s$ loops $\alpha_i, \beta_i, \gamma_i, \gamma^\rm{II}_i$, $\gamma^\rm{I}_i$ that generate the fundamental group $\pi_1 (\Sigma, p_0)$ and satisfy the relation
\begin{eqntag}
	\label{251025204931}
	\prod_{i=1}^{\substack{g\\\longleftarrow}} [\alpha_i, \beta_i]
	\cdot
	\prod_{i=1}^{\substack{l\\\longleftarrow}} \gamma_i
	\cdot
	\prod_{i=1}^{\substack{r\\\longleftarrow}} \gamma_i^\rm{II}
	\cdot
	\prod_{i=1}^{\substack{d\\\longleftarrow}} \gamma_i^\rm{I}
	= 1.
\end{eqntag}
More explicitly in terms of only the generators, this relation reads
\begin{eqntag}
	\label{251109164128}
	\prod_{i=1}^{\substack{g\\\longleftarrow}} [\alpha_i, \beta_i]
	\cdot
	\prod_{i=1}^{\substack{l\\\longleftarrow}} \gamma_i
	\cdot 
	\prod_{i=1}^{\substack{r\\\longleftarrow}} d_{i}^{-1} 
	\delta_{i} d_{i}
	\cdot
	\prod_{i=1}^{\substack{d\\\longleftarrow}} c_i^{-1} 
	\left( \prod_{j=1}^{\substack{m_i\\\longleftarrow}} \delta_{i,j} \right) c_i
	= 1.
\end{eqntag}

Observe that this relation can be solved uniquely for any of $\gamma_i$, $\delta_i$, $\delta_{i,j}$, but not for any of $\alpha_i$, $\beta_i$, $c_i$, $d_i$.
Therefore, we can destroy this relation by excluding any one of the paths $\gamma_i$, $\delta_i$, $\delta_{i,j}$ from the list of generators.
Moreover, we have one extra relation
because  one of the paths $c_i$ or $d_i$ is a loop based at $p_0$ which, without loss of generality, we may take from the outset to be the constant path at $p_0$ as suggested in the above procedure.
Indeed, suppose for definiteness that $c_1$ is the loop at $p_0$.
Then we can change the chosen generators to be such that $c_1 = 1$: we (left) conjugate every path $\alpha_i, \beta_i, \gamma_i$ by $c_1$, right multiply every path $c_i, d_i$ by $c_1^{-1}$, and leave all $\delta_i, \delta_{i,j}$ alone.
Notice that all the loops $\gamma_i, \gamma_i^\rm{II}, \gamma_i^\rm{I}$ and each commutator $[\alpha_i, \beta_i]$ get (left) conjugated by $c_1$.
This means the lefthand side in \eqref{251025204931} gets conjugated by $c_1$, leaving the relation unchanged.
We summarise this discussion as follows.

\begin{lemma}
	\label{251215203809}
	Let $(\Sigma, P)$ be any  surface with marked boundary
	of genus $g$ with $s \geq 1$ boundary circles, of which $l \geq 0$ are unmarked, $r \geq 0$ are regular and $d \geq 0$ are irregular. Let $m =|P_{\rm I}|$.
	Select any $p_0 \in P$.
	Then the above procedure determines a finite presentation of the discrete fundamental groupoid $\pi_1 (\Sigma, P)$ with $2g+s+r+m-1$ generators and one relation:
	\begin{eqntag}
		\label{251215204359}
		\pi_1 (\Sigma, P)
		\cong \inner{ \alpha_i, \beta_i, \gamma_i, \delta_i, \delta_{i,j}, c_i, d_i, \ell_i ~\Big|~ \text{ relation \eqref{251109164128}}}.
	\end{eqntag}
	Mind that $c_1 = 1$ if $p_0 \in P_\rm{I}$ and $d_1 = 1$ if $p_0 \in P_\rm{II}$.
	In particular, by excluding any one of the paths $\gamma_i, \delta_i, \delta_{i,j}$ from the list of generators in either case gives a finite presentation with $2g+s+r+m-2$ and no relations.
\end{lemma}

\section{Decorated Local Systems}\label{se:3}

In this section, we introduce various notions of local systems, which we collectively call \dfn{decorated local systems}, and we discuss the relationships between them.

We will sometimes write $(\cal{E}, \cal{F}, \star)$ for a decorated local system with $\star = \set{ \star_p : p \in P}$ where $\star_p$ stands for the intended decoration on the local filtration $\cal{F}_p$, including blank (or empty) decoration.
Correspondingly, we use a uniform notation $\catname{Loc}_G^\square$ for the groupoids and $\frak{Loc}_G^\square$ for the corresponding moduli spaces of decorated local systems where $\square$ indicates the type of decoration; e.g., $\rm{Fi}$ for filtered local systems, $\rm{Fr}$ for framed local systems and $\rm{PFr}$ for projectively framed local systems.

\subsection{Filtered Local Systems}
\label{251218095539}

The purpose of decorating a surface with marked points is to keep track of allowing to decorate local systems with additional local linear-algebraic data, the most fundamental of which are flags.

Recall that a \dfn{flag} $F$, by which we shall always mean a \textit{complete flag}, on an $n$-dimensional complex vector space $E$ is a nested collection of $n$ distinct subspaces
\begin{eqn}
	F = \big( 0 = F_0 \subset F_1 \subset \cdots \subset F_n = E \big)
	\qtext{where}
	\dim F_i = i.
\end{eqn}
We refer to a pair $(E,F)$ as a \dfn{filtered vector space}.
A \dfn{filtered map} $\phi : (E, F) \to (E', F')$ is a linear map $\phi : E \to E'$ that respects the flags in the sense that $\phi (F_i) \subset F'_i$.

Let $(\Sigma, P)$ be a decorated surface.
For every primary marked point $p \in P_\rm{I}$, let $U_p \subset \Sigma$ be any sectorial or arc neighbourhood of $p$; i.e., a contractible open subset of respectively $\Sigma$ or $\del \Sigma$ containing $p$.
For every secondary marked point $p \in P_\rm{II}$, let $U_p \subset \Sigma$ be either the boundary circle $C_p \subset \del \Sigma$ containing $p$ or any annular neighbourhood of $C_p$; i.e., an open subset of $\Sigma$ that contracts onto $C_p$.

If $\cal{E}$ is a local system on $\Sigma$, let us denote its restriction to a local system on the subset $U_p$ for each $p \in P$ by
\begin{eqn}
	\cal{E}_p \coleq \cal{E} |_{U_p}.
\end{eqn}
This notation is not to be confused with the stalk of $\cal{E}$ at $p$ which we denote by $E_p$.
If $p \in P_\rm{I}$, then $U_p$ is contractible onto $p$ and so there is a canonical isomorphism $\Gamma (U_p, \cal{E}_p) \cong E_p$.
However, if $p \in P_\rm{II}$, then $U_p$ is a non-contractible loop and so there is no such isomorphism in general; in fact, $\Gamma (U_p, \cal{E}_p)$ is typically the zero vector space.

\begin{definition}
	\label{251126214152}
	A \textbf{filtered local system} on a  surface with marked boundary $(\Sigma, P)$ is a pair $(\cal{E}, \cal{F})$ consisting of a local system $\cal{E}$ on $\Sigma$ endowed with a collection $\cal{F} = \set{ \cal{F}_p : p \in P}$ where $\cal{F}_p$ is a \dfn{filtration} near $p$.
	That is, for any sectorial neighbourhood $U_p$ of $p$,
	$\cal{F}_p$ is a nested collection of $n$ distinct local subsystems of the restricted local system $\cal{E}_p$ over $U_p$, called a \dfn{local filtration} near $p$:
	\begin{eqn}
		\cal{F}_p \coleq \big(0 = \cal{F}_p^0 \subset \cal{F}_p^1 \subset \cdots \subset \cal{F}_p^n = \cal{E}_p \big)
		\qtext{where}
		\rank (\cal{F}_p^i) = i.
	\end{eqn}
	A \dfn{filtered morphism} $\phi : (\cal{E},\cal{F}) \to (\cal{E}',\cal{F}')$ between filtered local systems on $(\Sigma, P)$ is a morphism of sheaves $\phi : \cal{E} \to \cal{E}'$ whose restriction $\phi_p : \cal{E}_p \to \cal{E}'_p$ to each subset $U_p$ respects the local filtrations in the sense that $\phi_p (\cal{F}_p^i) \subset \cal{F}'^i_p$ for every $i$.
	It is an isomorphism whenever $\phi : \cal{E} \to \cal{E}'$ is.
	We denote the groupoid and the corresponding moduli space of filtered local systems of rank $n$ on $(\Sigma, P)$ with filtered isomorphisms by
	\begin{eqn}
		\catname{Loc}_G^\rm{Fi} (\Sigma, P)
		\qtext{and}
		\frak{Loc}_G^\rm{Fi} (\Sigma, P)
		\coleq \catname{Loc}_G^\rm{Fi} (\Sigma, P) \big/{\sim}.
	\end{eqn}
\end{definition}

Notice that this definition specialises to the classical situation in the absence of marked points: $\catname{Loc}^\rm{Fi}_G (\Sigma, \emptyset) = \catname{Loc}_G (\Sigma)$.
Also, notice that any morphism of filtered local systems is in particular a morphism of local systems, so for any $P' \subsetneq P$ there is a forgetful functor that drops all local filtrations near the forgotten marked points $P \smallsetminus P'$:
\begin{eqn}
	\catname{Loc}_G^\rm{Fi} (\Sigma, P) \to \catname{Loc}_G^\rm{Fi} (\Sigma, P').
\end{eqn}
This functor is surjective on objects and faithful on morphisms but never full.

We  remark that in the literature, it is common to encounter local systems decorated by local filtrations only near the primary marked points.
Although the $\catname{Loc}_G^\rm{Fi} (\Sigma, P) \to \catname{Loc}_G^\rm{Fi} (\Sigma, P_\rm{I})$ fails to be an equivalence, the nature of this failure turns out to be finite up to isomorphism; cf. \autoref{251106151533}.

Next, we give a useful alternative characterisation of the groupoid of filtered local systems which concentrates all the filtration information to the stalks at marked points.

\begin{lemma}
	\label{251127102439}
	A filtered local system $(\cal{E}, \cal{F})$ on a surface with marked boundary $(\Sigma, P)$ is equivalently a pair $(\cal{E}, F)$ consisting of a local system $\cal{E}$ on $\Sigma$ endowed with a collection $F = \set{ F_p : p \in P}$ where each $F_p$ is a flag in the stalk $E_p$ of $\cal{E}$ at $p$ with the following property.
	For every secondary marked point $p \in P_\rm{II}$, the flag $F_p$ is invariant under the local monodromy $\op{mon}_\gamma$ of $\cal{E}$ along the boundary loop $\gamma \in \pi_1 (\Sigma, p)$ based at $p$; i.e, $\op{mon}_\gamma \in \Aut (E_p, F_p)$.
	
	Similarly, a morphism $\phi : (\cal{E},\cal{F}) \to (\cal{E}',\cal{F}')$ of filtered local systems on $(\Sigma, P)$ is equivalently a morphism $\phi : (\cal{E}, F) \to (\cal{E}', F')$ which is defined as a morphism of sheaves $\phi : \cal{E} \to \cal{E}'$ whose restriction to the stalks at every marked point $p \in P$ is a filtered linear map $\phi_p : (E_p, F_p) \to (E'_p, F'_p)$.
	
	In other words, the groupoid $\catname{Loc}_G^\rm{Fi} (\Sigma, P)$ of filtered local systems on a surface with marked boundary $(\Sigma, P)$ is isomorphic to the groupoid parameterising the pairs $(\cal{E},F)$ as above.
\end{lemma}

Namely, $F_p$ is the stalk of $\cal{F}$ at $p$; conversely, $\cal{F}$ is the parallel transport of $F_p$ to every point of $U_p$.
Consequently, \autoref{251126214152} is independent of the size of the neighbourhoods $U_p$
From now on, we use these characterisations of $\catname{Loc}_G^\rm{Fi} (\Sigma, P)$ interchangeably.

\begin{proof}
	For any $p \in P$, the local filtration $\cal{F}_p$ of a filtered local system $(\cal{E}, \cal{F})$ determines a flag in the stalk of $\cal{E}$ at every point of $U_p$; in particular, it determines a flag $F_p$ in the stalk $E_p$ of $\cal{E}$ at $p \in U_p$.
	Moreover, since $\cal{F}_p$ is a filtration by local subsystems of $\cal{E}_p$, the holonomy of $\cal{E}$ along paths contained in $U_p$ defines a filtered isomorphism between any two stalks and in particular a filtered isomorphism with the filtered stalk $(E_p, F_p)$.
	Notably, this means the flag $F_p$ at any secondary marked point $p \in P_\rm{II}$ is invariant under the local monodromy of $\cal{E}$ around the boundary loop based at $p$.
	We stress, however, that the flags $F_p$ at any primary marked point $p \in P_\rm{I}$ may not be invariant under the local monodromy: their purpose is only to constrain the possible morphisms of decorated local systems.
	
	Conversely, a flag $F_p$ in $E_p$ determines the local filtration $\cal{F}_p$ over $U_p$ by using the holonomy of $\cal{E}$ to transport $F_p$ to every stalk of $\cal{E}$ along $U_p$.
	When $p \in P_\rm{II}$ is a secondary marked point, this local filtration $\cal{F}_p$ is well-defined if and only if the flag $F_p$ is invariant under the local monodromy.
\end{proof}

\begin{remark}
	Our notion of \textit{filtered local systems} (\autoref{251126214152}) corresponds to what several authors call ``framed local systems''.
	This includes Fock-Goncharov \cite[Def 1.2]{MR2233852}, Goncharov \cite{MR3702383}, Allegretti \cite[Def 3.4]{MR3717948}, Goncharov-Shen \cite[§3]{250619181547}, Allegretti-Bridgeland \cite[§4]{MR4155179}, as well as Dal Martello-Mazzocco \cite[§3]{MR4756422} amongst others.
	It also corresponds to what Jordan-Le-Schrader-Shapiro \cite[Def 1.6]{jordan2021quantumdecoratedcharacterstacks} call ``decorated local systems''.
	
	In the case of a surface with marked boundary $(\Sigma, P)$with at least one marked point and no unmarked boundaries,
	our moduli space $\frak{Loc}_G^\rm{Fi} (\Sigma, P)$ corresponds to the moduli space $\cal{X}_{\G,\hat{\S}}$ in the notation of \cite[Def 1.2]{MR2233852}, to $\cal{X}_{\rm{G},\SS}$ in the notation of \cite{250619181547}, to $\cal{X} (\SS,\MM)$ in the notation of \cite[§4]{MR4155179}, and to $\rm{Ch} (\SS)$ in the notation of \cite[Def 1.8]{jordan2021quantumdecoratedcharacterstacks}.
	In contrast, we reserve the symbol ``$\frak{X}$'' for the decorated character stack/variety introduced below in \autoref{251127191014}, which we will show is isomorphic to $\frak{Loc}_G^\rm{Fi} (\Sigma, P)$ but which is defined independently.	
\end{remark}

\subsection{Framed Local Systems}

The local filtrations of filtered local systems are often further decorated by linear-algebraic data, the most important of which is a basis adapted to the local filtration.
First, let us recall some important definitions.

A \dfn{frame} on a filtered vector space $(E,F)$ of dimension $n$ is an isomorphism $\beta : E \iso \CC^n$ which takes the flag $F$ to the standard flag $\CC^\bullet$ meaning $\beta (F_i) = \CC^i$ for every $i$.
In short, it is a filtered isomorphism $\beta : (E,F) \iso (\CC^n, \CC^\bullet)$.
Equivalently, it is a choice of basis $\beta = (e_1, \ldots, e_n)$ of $E$ adapted to the flag $F$ in the sense that $F_i = \inner{ e_1, \ldots, e_i }$.
We call $(E,F,\beta)$ a \dfn{framed vector space}.

For example, the standard flag $\CC^\bullet$ on $\CC^n$ carries the \textit{standard frame}, denoted by $\idd$, given by the identity map $\id : \CC^n \to \CC^n$ or equivalently by the standard basis of $\CC^n$.

Actually, the information of the flag $F$ in a framed vector space is superfluous as it is fully determined by the frame $\beta$.
So we may equally well say that a framed vector space is a pair $(E,\beta)$ consisting of a vector space equipped with an isomorphism $\beta : E \iso \CC^n$.

A \dfn{unipotent isomorphism} $\phi : (E, F, \beta) \to (E', F', \beta')$ between two framed vector spaces is a linear isomorphism $\phi : E \iso E'$ whose matrix representation in the two frames is a unipotent upper-triangular matrix $\beta' \circ \phi \circ \beta^{-1} \in U$.
In contrast, a \dfn{strict isomorphism} $\phi : (E, F, \beta) \to (E', F', \beta')$ between two framed vector spaces is a linear isomorphism such that $\beta = \beta' \circ \phi$; i.e., the matrix representation of $\phi$ is the identity matrix.
A  strict isomorphism is necessarily unipotent, and a unipotent isomorphism is necessarily filtered.

\begin{definition}
	\label{251205224404}
	Let $(\cal{E}, \cal{F})$ be any filtered local system on a surface with marked boundary $(\Sigma, P)$.
	A \dfn{local frame} on $(\cal{E}, \cal{F})$ at $p \in P$ is a filtered trivialisation in the stalk at $p$; i.e., a filtered isomorphism $\beta_p : (\cal{E}_p, \cal{F}_p) \iso (\CC^n, \CC^\bullet)$.
	A \dfn{frame} on $(\cal{E}, \cal{F})$ over $P$ is a collection $\beta = \set{ \beta_p : p \in P}$ of local frames, one at each marked point.
	Any such triple $(\cal{E}, \cal{F}, \beta)$ is called a \dfn{framed local system} on $(\Sigma, P)$.
\end{definition}

Each local frame $\beta_p$ is equivalently a basis $\beta_p = (e_1, \ldots, e_n)$ of the vector space $E_p$ adapted to the flag $F_p$; i.e., $F_p^i = \inner{ e_1, \ldots, e_i }$.
It also extends by parallel transport to a trivialisation of the restricted filtered local system $\beta_p : (\cal{E}_p, \cal{F}_p) \iso (\underline{\CC}^n, \underline{\CC}^\bullet)$ over any sectorial or arc neighbourhood $U_p \subset \Sigma$ of $p$.
Namely, the basis $\beta_p$ of $E_p$ extends to a basis $\beta_p = (e_1, \ldots, e_n)$ of sections of the local system $\cal{E}$ over $U_p$; i.e., $e_i \in \Gamma (U_p, \cal{E})$.
This basis of sections is adapted to the local filtration $\cal{F}_p$ in the sense that $\Gamma (U_p, \cal{F}_p) = \inner{ e_1, \ldots, e_i }$.
But beware that, if $p$ is a secondary marked point, the local frame $\beta_p : (E_p, F_p) \iso (\CC^n, \CC^\bullet)$ does \textit{not} in general extend to a trivialisation of the filtered local system $(\cal{E}_p, \cal{F}_p)$ over the boundary circle based at $p$ or any annular neighbourhood of it.

\begin{definition}
	\label{251205225220}
	A \dfn{unipotent isomorphism} $\phi : (\cal{E},\cal{F},\beta) \to (\cal{E}',\cal{F}',\beta')$ between two framed local systems is an isomorphism of the underlying filtered local systems such that each restriction $\phi_p : (E_p,F_p) \to (E'_p,F'_p)$ is unipotent in the sense that its matrix representation with respect to the given frames is a unipotent upper-triangular matrix: $\beta'_p \circ \phi_p \circ \beta^{-1}_p \in U \subset G$.
	We denote the groupoid and the corresponding moduli space of framed local systems of rank $n$ on $(\Sigma, P)$ with unipotent isomorphisms by
	\begin{eqn}
		\catname{Loc}_G^\rm{Fr} (\Sigma, P)
		\qtext{and}
		\frak{Loc}_G^\rm{Fr} (\Sigma, P)
		\coleq \catname{Loc}_G^\rm{Fr} (\Sigma, P) \big/{\sim}.
	\end{eqn}
\end{definition}

In other words, a unipotent isomorphism $\phi$ takes each local frame $\beta_p$ to the local frame $\beta'_p$ up to a unipotent upper-triangular matrix.
Sometimes, the isomorphisms between framed local systems are restricted further to preserve the frames on the nose as follows.

\begin{definition}
	\label{251217232115}
	A \dfn{strict isomorphism} $\phi : (\cal{E},\cal{F},\beta) \to (\cal{E}',\cal{F}',\beta')$ between two framed local systems over a  surface with marked boundary $(\Sigma, P)$ is an isomorphism $\phi : \cal{E} \iso \cal{E}'$ of the underlying local systems which relates the frame in the sense that $\beta = \beta' \circ \phi$.
	We denote the groupoid and the corresponding moduli space of framed local systems of rank $n$ on $(\Sigma, P)$ with strict isomorphisms by
	\begin{eqn}
		\catname{Loc}_G^\rm{sFr} (\Sigma, P)
		\qtext{and}
		\frak{Loc}_G^\rm{sFr} (\Sigma, P)
		\coleq \catname{Loc}_G^\rm{sFr} (\Sigma, P) \big/{\sim}.
	\end{eqn}
	When referring to a framed local system $(\cal{E}, \cal{F}, \beta)$ as an object in this groupoid, we will call it a \dfn{strictly framed local system} on $(\Sigma, P)$.
\end{definition}

Any framed local system has an underlying filtered local system, and any isomorphism of framed local systems is in particular an isomorphism of filtered local systems.
Similarly, any strictly framed local system is in particular a framed local system, and any strict isomorphism is unipotent.
Thus, there is a chain of forgetful functors
\begin{eqn}
	\catname{Loc}_G^\rm{sFr} (\Sigma, P)
	\to \catname{Loc}_G^\rm{Fr} (\Sigma, P)
	\to \catname{Loc}_G^\rm{Fi} (\Sigma, P).
\end{eqn}
Any filtered local system can always be promoted to a framed local system by choosing local frames arbitrarily, whilst the objects of $\catname{Loc}_G^\rm{sFr}$ are literally the same as those of $\catname{Loc}_G^\rm{Fr}$.
Consequently, these functors are surjective on objects and faithful on morphisms but never full unless $P$ is empty: the groupoid $\catname{Loc}_G^\rm{sFr}$ has far fewer isomorphisms than $\catname{Loc}_G^\rm{Fr}$ which itself has fewer isomorphisms than $\catname{Loc}_G^\rm{Fi}$.
This is analogous to how the identity matrix is in particular a unipotent upper-triangular matrix which is in particular an upper-triangular matrix.
As a result, the first forgetful functor is actually an embedding.
On the other hand, although there are many possible ways to embed the objects of $\catname{Loc}_G^\rm{Fi}$ into the objects of $\catname{Loc}_G^\rm{Fr}$, it is not possible to embed the groupoid $\catname{Loc}_G^\rm{Fi}$ into $\catname{Loc}_G^\rm{Fr}$.

\begin{remark}
	Our terminology \textit{framed local systems} (\autoref{251205224404}) is more in keeping with the well-established terminology from differential geometry where traditionally a ``frame'' at a point is an ordered basis in the fibre of a given vector bundle.
	This also closely matches the use of this word by other authors including Boalch \cite{BoalchThesis,Boalch2001,MR4285677}, Hiroe-Yamakawa \cite{zbMATH06348915}, and Jordan-Le-Schrader-Shapiro \cite[Def 1.8]{jordan2021quantumdecoratedcharacterstacks}.
	Likewise, our terminology \textit{filtered local systems} closely matches the similar terminology used by Deligne \cite{MR0417174}, Malgrange \cite[§4]{MR0728430}, Simpson \cite[§3]{MR1040197}, Sabbah \cite{MR2368364,sabbah2012introduction}, Boalch \cite{MR4285677}, Huang-Sun \cite[Def 3.3]{MR4927439}, and many others.
	
	Flags endowed with a basis for the associated graded are sometimes called \textit{decorated flags} \cite[§3.1.2]{250619181547} or \textit{affine flags} \cite[p.41]{MR2233852}, although the latter term is also commonly used for flags in infinite-dimensional vector spaces.
	On the other hand, we use the word ``decorated'' as a collective adjective to describe flags with extra structure instead.
\end{remark}

\subsection{Projectively Framed Local Systems}

A \dfn{projective frame} $\bar{\beta}$ on a filtered vector space $(E,F)$ is a frame defined up to a nonzero scalar multiple.
In other words, it is a basis $\beta = (e_1, \ldots, e_n)$ of $E$ adapted to the flag $F$ but which is defined only up to overall rescaling; i.e., $\beta = (e_1, \ldots, e_n) \sim (t e_1, \ldots, t e_n) = t \beta$ for any $t \in \CC^\times$.
A \dfn{projectively unipotent map} $\phi : (E,F,\bar{\beta}) \to (E',F',\bar{\beta}')$ is a unipotent map up to a scalar factor: i.e., the matrix representation of $\phi$ with respect to any two representative frames $\beta, \beta'$ is unipotent up to a nonzero scalar factor: $\beta' \circ \phi \circ \beta^{-1} \in \CC^\times U$.

\begin{definition}
	\label{251206130050}
	Let $(\cal{E}, \cal{F})$ be a filtered local system on a surface with marked boundary $(\Sigma, P)$.
	A \dfn{local projective frame} $\hat{\beta}_p$ on $(\cal{E}, \cal{F})$ at $p \in P$ is a local frame considered up to scale; i.e., $\bar{\beta}_p$ is an equivalence class of local frames $\beta_p : (E_p, F_p) \iso (\CC^n, \CC^\bullet)$ where $\beta_p \sim t \beta_p$ for any $t \in \CC^\times$.
	A \dfn{projective frame} on $(\cal{E}, \cal{F})$ is a collection $\bar{\beta} = \set{ \bar{\beta}_p : p \in P}$ of local projective frames.
	Any such triple $(\cal{E}, \cal{F}, \bar{\beta})$ is called a \dfn{projectively framed local system}.
	
	A morphism $\phi : (\cal{E}, \cal{F}, \bar{\beta}) \to (\cal{E}', \cal{F}', \bar{\beta}')$ of two projectively framed local systems, called a \dfn{projectively unipotent morphism}, is a filtered morphism $\phi : (\cal{E}, \cal{F}) \to (\cal{E}', \cal{F}')$ whose matrix representation with respect to any representative frames $\beta, \beta'$ is projectively unipotent: $\beta' \circ \phi \circ \beta^{-1} \in \CC^\times U$.
	We denote the groupoid and the corresponding moduli space of projectively framed local systems of rank $n$ on $(\Sigma, P)$ with projectively unipotent morphisms by
	\begin{eqn}
		\catname{Loc}_G^\rm{PFr} (\Sigma, P)
		\qtext{and}
		\frak{Loc}_G^\rm{PFr} (\Sigma, P)
		\coleq \catname{Loc}_G^\rm{PFr} (\Sigma, P) \big/{\sim}.
	\end{eqn}
\end{definition}

\section{Decorated Groupoid Representations}\label{se:4}

In this section we reformulate decorated local systems as \textit{decorated} representations of the fundamental groupoid and the discrete fundamental groupoid.

\subsection{Filtered Groupoid Representations}

Let $(\Sigma, P)$ be a decorated surface.
As we did at the start of \autoref{251218095539}, for every primary marked point $p \in P_\rm{I}$, let $U_p \subset \Sigma$ be any sectorial or arc neighbourhood of $p$.
For every secondary marked point $p \in P_\rm{II}$, let $U_p \subset \Sigma$ be either the boundary circle $C_p \subset \del \Sigma$ containing $p$ or any annular neighbourhood of $C_p$.

If $\E$ is a vector bundle on $\Sigma$, let us denote its restriction to a vector bundle over the subset $U_p$ for each $p \in P$ by
\begin{eqn}
	\E_p \coleq \E |_{U_p}.
\end{eqn}
This notation is not to be confused with the fibre of $\E$ at $p$ which we denote by $E_p$.

Back in \autoref{251209141021} (specifically, \autoref{251209142102}), we recalled the notion of representations of the fundamental groupoid $\Pi_1 (\Sigma) \rightrightarrows \Sigma$.
We now enhance this notion by introducing local decorations near the marked points in the form of filtrations which are locally preserved by the representation.

\begin{definition}
	\label{251127151457}
	A \dfn{filtered representation} of the fundamental groupoid $\Pi_1 (\Sigma)$ of a surface with marked boundary $(\Sigma, P)$ is a triple $(\E, \F, \Phi)$ consisting of a representation $(\E, \Phi)$ of the fundamental groupoid $\Pi_1 (\Sigma)$ endowed with a collection $\F \coleq \set{ \F_p : p \in P}$ of \dfn{flat local filtrations}.
	That is, for any sectorial neighbourhood $U_p$ of $p$, $\F_p$ is a nested collection of $n$ distinct vector subbundles of the restricted vector bundle $\E_p$ over $U_p$,
	\begin{eqn}
		\F_p \coleq \big( 0 = \F_p^0 \subset \F_p^1 \subset \cdots \subset \F_p^n = \E_p \big)
		\qtext{where}
		\rank (\F_p^i) = i,
	\end{eqn}
	which is \dfn{flat} with respect to $\Phi$ in the following sense.
	For every path $\gamma : [0,1] \to U_p$, the groupoid homomorphism $\Phi$ induces a filtered isomorphism between the corresponding fibres of the vector bundle $\E_p$:
	\begin{eqn}
		\Phi_\gamma : (\E_p, \F_p) \big|_{\gamma (0)} \iso (\E_p, \F_p) \big|_{\gamma (1)}.
	\end{eqn}
\end{definition}

We stress that in this definition, the homomorphism $\Phi : \Pi_1 (\Sigma) \to \GL (\E)$ is \textit{not} assumed to relate in any way the local filtrations near different marked points.
These filtrations are merely local decorations on a groupoid representation in the ordinary sense.
In light of this, perhaps a more appropriate name would be ``locally filtered representations'' or ``filtration-decorated representations'', but we prefer the shorter one.
The primary purpose of these filtrations is instead to constrain the set of possible morphisms between the given representations as defined next.

\begin{definition}
	\label{251127153641}
	A \dfn{filtered morphism} $\phi : (\E, \F, \Phi) \to (\E', \F', \Phi')$ of two filtered representations of $\Pi_1 (\Sigma)$ on $(\Sigma, P)$ is a morphism $\phi : (\E, \Phi) \to (\E', \Phi')$ of the underlying representations which restricts over every $U_p$ to a map of filtered vector bundles $\phi_p : (\E_p, \F_p) \to (\E'_p, \F'_p)$.
	We denote the groupoid and the corresponding moduli space of filtered representations of rank $n$ on $(\Sigma, P)$ with filtered isomorphisms by
	\begin{eqn}
		\catname{Rep}_G^\rm{Fi} \big( \Pi_1 (\Sigma), P \big)
		\qtext{and}
		\frak{Rep}_G^\rm{Fi} \big( \Pi_1 (\Sigma), P \big)
		\coleq \catname{Rep}_G^\rm{Fi} \big( \Pi_1 (\Sigma), P \big) \big/{\sim}.
	\end{eqn}
\end{definition}

Next, we give an alternative characterisation of decorated groupoid representations, analogous to \autoref{251127102439} for filtered local systems.

\begin{lemma}
	\label{251127155249}
	A filtered representation $(\E, \F, \Phi)$ of the fundamental groupoid $\Pi_1 (\Sigma)$ of a surface with marked boundary $(\Sigma, P)$ is equivalently a triple $(\E, F, \Phi)$ consisting of a representation $(\E, \Phi)$ of $\Pi_1 (\Sigma)$ decorated by a collection $F \coleq \set{ F_p : p \in P}$ where $F_p$ is a flag in the fibre $E_p$ of $\E$ at $p$ with the following addition property.
	For every secondary marked point $p \in P_\rm{II}$, the flag $F_p$ is invariant under the isomorphism $\Phi_\gamma : E_p \to E_p$ where $\gamma \in \pi_1 (\Sigma, p)$ is the boundary loop.
\end{lemma}

Again, it is worth stressing that the homomorphism $\Phi : \Pi_1 (\Sigma) \to \GL (\E)$ is \textit{not} assumed to relate in any way the local flags at different marked points.
The homomorphism $\Phi$ is, so to speak, `ignorant' of the flags $F_p$: if $\gamma$ is any path between $p,p' \in P$, we do \textit{not} require the corresponding isomorphism $\Phi_\gamma : E_p \iso E_{p'}$ to respect the flags $F_p, F_{p'}$.
The only exception to this are the flags at secondary marked points which must be invariant under the local monodromy.

\begin{lemma}
	\label{251127155720}
	A morphism of two filtered representations of $\Pi_1 (\Sigma)$ on $(\Sigma, P)$ is equivalently a morphism $\phi : (\E, \Phi, F) \to (\E', \Phi', F')$ consisting of a morphism of representations $\phi : (\E, \Phi) \to (\E', \Phi')$ which is filtered at every marked point; i.e., the restriction of $\phi$ to any $p \in P$ is a filtered map $\phi_p : (E_p, F_p) \to (E'_p, F'_p)$.
\end{lemma}

The key fact about filtered representations of the fundamental groupoid is the following proposition regarding their close relationship with filtered local systems.

\begin{proposition}
	\label{250910194000}
	The holonomy representation functor extends to an equivalence of categories between filtered local systems and filtered fundamental groupoid representations:
	\begin{eqn}
		\op{Hol} : \catname{Loc}_G^\rm{Fi} (\Sigma, P) \iso \catname{Rep}_G^\rm{Fi} \big( \Pi_1 (\Sigma), P \big).
	\end{eqn}
	It sends $(\cal{E}, \cal{F}) \mapsto (\E, \F, \Phi) \coleq (\cal{E} \otimes \cal{O}_\Sigma,\cal{F} \otimes \cal{O}_\Sigma, \op{hol})$.
\end{proposition}

\begin{proof}
	The holonomy representation functor $\op{Hol} : \catname{Loc}_G (\Sigma) \iso \catname{Rep}_G (\Sigma)$ was explicitly constructed in \autoref{250911161735}.
	Let $(\cal{E}, \cal{F})$ be a filtered local system.
	The holonomy functor yields a representation $(\E, \Phi) = (\cal{E} \otimes \cal{O}_\Sigma, \op{hol})$.
	For every $p \in P$, the local filtration $\cal{F}_p$ induces a local filtration $\F_p \coleq \cal{F}_p \otimes \cal{O}_{U_p}$ on the bundle $\E_p$ over $U_p$.
	This filtration is flat because $\Phi$ is the holonomy of $\cal{E}$ and $\cal{F}_p$ is invariant under the holonomy along paths in $U_p$.
	If we regard $\cal{F}_p$ as a sheaf on $\Sigma$ with support on $U_p$, then we can simply write $\F_p = \cal{F}_p \otimes \cal{O}_{\Sigma}$ as in the statement.
\end{proof}

As for filtered local systems, the above definitions specialise to the classical situation without marked points: $\catname{Rep}_G^\rm{Fi} \big( \Pi_1 (\Sigma), \emptyset \big) = \catname{Rep}_G \big( \Pi_1 (\Sigma) \big)$.

\subsection{Framed and Projectively Framed Groupoid Representations}

\begin{definition}
	\label{251209155014}
	Let $(\E, \F, \Phi)$ be a filtered representation of the fundamental groupoid $\Pi_1 (\Sigma)$ on a surface with marked boundary $(\Sigma, P)$.
	A \dfn{local frame} $\beta_p$ on $(\E, \F, \Phi)$ at $p \in P$ is a filtered flat trivialisation at $p$; i.e., a filtered isomorphism $\beta_p : (E_p, F_p) \iso (\CC^n, \CC^\bullet)$ of vector bundles over $U_p$ with the following property called \dfn{flatness} with respect to $\Phi$.
	For every path $\gamma : [0,1] \to U_p$, the groupoid homomorphism $\Phi$ satisfies
	\begin{eqn}
		\beta_p \big|_{\gamma (1)} \circ \Phi_\gamma = \beta_p \big|_{\gamma (0)}.
	\end{eqn}
	A \dfn{frame} on $(\E, \F, \Phi)$ is a collection $\beta = \set{ \beta_p : P \in P }$ of local frames, one at each marked point.
	Any such triple $(\E, \F, \Phi)$ is a called a \dfn{framed} \dfn{representation} of the fundamental groupoid $\Pi_1 (\Sigma)$ on the surface with marked boundary $(\Sigma, P)$.
\end{definition}

\begin{definition}
	\label{251209164307}
	A \dfn{unipotent isomorphism} $\phi : (\E, \F, \Phi, \beta) \to (\E', \F', \Phi', \beta')$ of two framed representations of the fundamental groupoid $\Pi_1 (\Sigma)$ on the surface with marked boundary $(\Sigma, P)$ is a filtered isomorphism such that each restriction $\phi_p : (E_p, F_p) \to (E'_p, F'_p)$ is a unipotent map with respect to the given frames at $p$.
	We denote the groupoid and the corresponding moduli space of framed representations of rank $n$ on $(\Sigma, P)$ by
	\begin{eqn}
		\catname{Rep}_G^\rm{Fr} \big( \Pi_1 (\Sigma), P \big)
		\qtext{and}
		\frak{Rep}_G^\rm{Fr} \big( \Pi_1 (\Sigma), P \big)
		\coleq \catname{Rep}_G^\rm{Fr} \big( \Pi_1 (\Sigma), P \big) \big/{\sim}.
	\end{eqn}
\end{definition}

Analogously, we give the following:

\begin{definition}
	\label{def:proj-rep}
	Let $(\E, \F, \Phi)$ be a filtered representation of the fundamental groupoid $\Pi_1 (\Sigma)$ on a surface with marked boundary $(\Sigma, P)$.
	A \dfn{local projective frame} $\overline\beta_p$ on $(\E, \F, \Phi)$ at $p \in P$ is an equivalence class of filtered isomorphisms $\beta_p : (E_p, F_p) \iso (\CC^n, \mathbb P \CC^\bullet)$ of vector bundles over $U_p$ which are flat with respect to $\Phi$, where $\beta_p\sim t\beta_p$ for any $t\in C^\times$.
	A \dfn{projective frame} on $( \E, \F, \Phi)$ is a collection $\overline\beta = \set{ \beta_p : P \in P }$ of local projective frames, one at each marked point.
	Any such triple $(\E, \F, \Phi,\overline\beta)$ is a called a \dfn{filtered projectively framed representation} of the fundamental groupoid $\Pi_1 (\Sigma)$ on the surface with marked boundary $(\Sigma, P)$.
\end{definition}

\begin{definition}
	\label{def:framed-morh}
	A \dfn{projectively unipotent isomorphism} $\phi : (\E, \F, \Phi, \overline{\beta}) \to (\E', \F', \Phi',  \overline{\beta}')$ of two projectively framed representations of the fundamental groupoid $\Pi_1 (\Sigma)$ on the surface with marked boundary $(\Sigma, P)$ is a filtered morphism such that each restriction $\phi_p : (E_p, F_p) \to (E'_p, F'_p)$ is a projectively unipotent map with respect to the given projective frames at $p$.
	We denote the groupoid and the corresponding moduli space of projectively framed representations of rank $n$ on $(\Sigma, P)$ by
	\begin{eqn}
		\catname{Rep}_G^\rm{PFr} \big( \Pi_1 (\Sigma), P \big)
		\qtext{and}
		\frak{Rep}_G^\rm{PFr} \big( \Pi_1 (\Sigma), P \big)
		\coleq \catname{Rep}_G^\rm{PFr} \big( \Pi_1 (\Sigma), P \big) \big/{\sim}.
	\end{eqn}
\end{definition}

The proof of \autoref{250910194000} can be easily adapted to the framed case by adding the frame information on the fibers, leading to the following:

\begin{proposition}
	\label{251216105015}
	The holonomy representation functor extends to an equivalence of categories between decorated local systems and decorated representations of the fundamental groupoid:
	\begin{eqn}
		\op{Hol} : \catname{Loc}_G^\square (\Sigma, P) \iso \catname{Rep}_G^\square \big( \Pi_1 (\Sigma), P \big),
	\end{eqn}
	for any $\square \in \set{ \rm{Fi}, \rm{Fr}, \rm{PFr} }$.
	It sends $(\cal{E}, \cal{F}, \star) \mapsto (\E, \F, \Phi, \star) \coleq (\cal{E} \otimes \cal{O}_\Sigma, \cal{F} \otimes \cal{O}_\Sigma, \op{hol}, \star)$, where $\star= \emptyset, \beta,\overline \beta$ for  $\square= \rm{Fi}, \rm{Fr}, \rm{PFr}$ respectively.
\end{proposition}

\subsection{Discrete and Decorated Groupoid Representations}
\label{251216083423}

Now, let us recall the discrete fundamental groupoid $\pi_1 (\Sigma, P) \rightrightarrows P$ introduced in \autoref{251214205432}, defined over any  surface with marked boundary $(\Sigma,P)$.
The manifestation of its discrete nature from the point of view of representation theory is as follows.
Any representation of $\pi_1 (\Sigma, P)$ involves in particular a complex vector bundle $E \to P$ which becomes nothing but a finite collection of vector spaces $E = \set{E_p : p \in P}$.
Correspondingly, a groupoid homomorphism $\varphi : \pi_1 (\Sigma, P) \to \GL (E)$ becomes a finitely-generated countable collection of isomorphisms between these vector spaces which is completely determined by its values on the finite set of generators of $\pi_1 (\Sigma, P)$.
Furthermore, since the discrete fundamental groupoid $\pi_1 (\Sigma, P)$ is a full subgroupoid of $\Pi_1 (\Sigma)$, any representation of $\Pi_1 (\Sigma)$ restricts to a representation of $\pi_1 (\Sigma, P)$.
The advantage of this point of view is that it provides a `discrete' model for the representation category of $\Pi_1 (\Sigma)$ in the sense of the following lemma that generalises the final assertion of \autoref{250911161735}.

\begin{lemma}
	\label{250914172733}
	For any  surface with marked boundary $(\Sigma,P)$, the restriction functor
	\begin{eqntag}
		\label{251214210200}
		\catname{Rep}_G \big( \Pi_1 (\Sigma) \big) \iso \catname{Rep}_G \big( \pi_1 (\Sigma, P) \big)
	\end{eqntag}
	is an equivalence of categories.
\end{lemma}

Explicitly, this restriction functor sends
\begin{eqn}
	(\E, \Phi) \mapsto (E, \varphi) \coleq \Big( \E |_P, \Phi \big|_{\pi_1 (\Sigma, P)} \Big).
\end{eqn}
An inverse functor is obtained choosing any $p \in P$ and restricting a representation $(E,\varphi)$ of $\pi_1 (\Sigma, P)$ to a representation of $\pi_1 (\Sigma, p)$, then lifting it to the holonomy representation of the local system $\cal{E} \coleq \big(\tilde{\Sigma} \times E_{p} \big) / \pi_1 (\Sigma, p)$ where $\tilde{\Sigma} \to \Sigma$ is the universal cover based at $p$.
In particular, \autoref{250914172733} implies that, for any $p \in P$, the restriction functor
\begin{eqn}
	\catname{Rep}_G \big( \pi_1 (\Sigma, P) \big) \iso \catname{Rep}_G \big( \pi_1 (\Sigma, p) \big)
\end{eqn}
is also an equivalence of categories.

\begin{proof}[Proof of \autoref*{250914172733}]
	All the necessary ingredients for the proof of this lemma were developed in the course of the proof of \autoref{250911161735} (see Appendix \ref{sec:A}).
	In particular, select any point $p \in P$ and consider the diagram
	\begin{eqntag}
		\label{251010135020}
		\begin{tikzcd}
			\catname{Rep}_G \big( \Pi_1 (\Sigma) \big)
			\ar[r, "\R"]
			\ar[dr, "{\R'}"]
			&	\catname{Rep}_G \big( \pi_1 (\Sigma, P) \big)
			\ar[d, "{\R''}"]
			\\
			\catname{Loc}_G (\Sigma)
			\ar[u, "\op{Hol}"]
			\ar[r, "\op{Mon}_{p}"']
			&	\catname{Rep}_G \big( \pi_1 (\Sigma, p) \big),
		\end{tikzcd}
	\end{eqntag}
	where $\R, \R', \R''$ are the restriction functors.
	By \autoref{250911161735}, $\R'$ is an equivalence, and we need to show that so is $\R$.
	This will also imply that $\R''$ is an equivalence.
	
	It is clear that $\R$ is full, but to see that it is faithful we appeal to the diagram \eqref{251010135020}.
	Suppose $(\E, \Phi)$ and $(\E', \Phi')$ is a pair of representations of $\Pi_1 (\Sigma)$, and $(E,\varphi) \coleq \R (\E, \Phi)$ and $(E', \varphi') \coleq \R (\E', \Phi')$ are their restrictions to representations of $\pi_1 (\Sigma, P)$.
	Restrict them further to representations $(E_p, \varphi_p) \coleq \R'' (E,\varphi)$ and $(E'_p, \varphi'_p) \coleq \R'' (E',\varphi')$ of $\pi_1 (\Sigma, p)$.
	The map $\R' = \R'' \circ \R : \sfop{Iso} \big( (\E, \Phi), (\E', \Phi') \big) \to \sfop{Iso} \big( (E_p, \varphi_p), (E'_p, \varphi'_p) \big)$ is bijective and in particular injective by \autoref{250911161735}, so it follows that the map $\R : \sfop{Iso} \big( (\E, \Phi), (\E', \Phi') \big) \to \sfop{Iso} \big( (E, \varphi), (E', \varphi') \big)$ must be injective, too.
	
	Next, we argue that $\R$ is essentially surjective.
	Suppose $(E, \varphi)$ is a representation of $\pi_1 (\Sigma, P)$.
	Restrict it to a representation $(E_p, \varphi_p) \coleq \R'' (E, \varphi)$ of $\pi_1 (\Sigma, p)$, and lift that to the holonomy representation $(\E, \Phi) \coleq \op{Hol} (\cal{E})$ of $\Pi_1 (\Sigma)$ of the local system $\cal{E} \coleq \big(\tilde{\Sigma} \times E_p \big) / \pi_1 (\Sigma, p)$.
	Thus, $\R' (\E, \Phi) \eqcol (E'_p, \varphi'_p) \cong (E_p, \varphi_p)$.
	We claim that the restricted representation $(\E|_P, \Phi|_P) \coleq \R (\E, \Phi)$ of $\pi_1 (\Sigma, P)$ is isomorphic to the original representation $(E, \varphi)$.
	
	The first thing to observe is that the fibre $E_q$ of the bundle $E \to P$ over any $q \in P$ is isomorphic to the fibre $E_p$ via the representation $\varphi_\gamma : E_q \iso E_p$ in as many ways as there are homotopy classes of paths from $q$ to $p$.
	As a result, a vector $e$ in $E_q$ determines an infinite collection $\underline{e} = \set{e_\gamma \coleq \phi_\gamma (e)}_\gamma$ of vectors in $E_p$, one for each homotopy class $\gamma \in \pi_1 (\Sigma, P)$ with source $q$ and target $p$.
	Equivalently, the collection $\underline{e}$ is enumerated by points $\tilde{q}$ in the fibre $\pi^{-1} (q)$ of the universal cover $\pi : \tilde{\Sigma} \to \Sigma$ based at $p$.
	Any two homotopy classes of paths from $q$ to $p$ form a loop at $p$, and the components of the infinite collection $\underline{e}$ are related by the action of the fundamental group $\pi_1 (\Sigma, p)$ through the representation $\varphi_p : \pi_1 (\Sigma, p) \to GL(E)$.
	In other words, upon examining \eqref{251009191043}, we have a canonical isomorphism between $E_q$ and the local system's stalk $\cal{E}_q$:
	\begin{eqntag}
		\label{251010161315}
		\nu_q : E_q \iso \cal{E}_q = \big(\pi^{-1} (q) \times E_p \big) \big/ \pi_1 (\Sigma, p),
	\end{eqntag}
	and this isomorphism intertwines the representation $\varphi$ and the holonomy of $\cal{E}$.
	At the same time, recall from \eqref{251010134146} that the fibre $\E |_q$ of the vector bundle $\E$ is also canonically isomorphic to the stalk $\cal{E}_q$:
	\begin{eqntag}
		\label{251010161331}
		\eta_q : \E |_q \iso \cal{E}_q = \big(\pi^{-1} (q) \times E_p \big) \big/ \pi_1 (\Sigma, p),
	\end{eqntag}
	and this isomorphism also intertwines the holonomy of $\cal{E}$ with the representation $\Phi$ and hence its restriction $\varphi$ of $\Phi$.
	We get an isomorphism $\phi_q \coleq \nu_q^{-1} \circ \eta_q : \E |_q \iso E_q$ for every $q \in P$, and therefore a bundle isomorphism $\phi : \E |_P \iso E$ which intertwines the representations $\varphi'$ and $\varphi$.
\end{proof}

We can conclude from \autoref{250914172733} that just the mere addition of marked points to the surface $\Sigma$ without any decorations does not yield any new phenomena from the point of view of representation theory.
However, the important observation is that the restriction functor of \autoref{250914172733} retains any decoration that a representation of $\Pi_1 (\Sigma)$ may have along $P$, so it acts like a `discretisation' of decorated representations of $\Pi_1 (\Sigma)$.
To phrase this precisely, we introduce the following definitions.

\begin{definition}
	\label{250912195848}
	Let $(\Sigma, P)$ be any  surface with marked boundary.
	A \dfn{filtered representation} of the discrete fundamental groupoid $\pi_1 (\Sigma, P)$ is a triple $(E, F, \varphi)$ consisting of a complex vector bundle $E = \set{ E_p : p \in P } \to P$, a filtration $F = \set{ F_p : p \in P }$, and a groupoid homomorphism $\varphi : \pi_1 (\Sigma, P) \to \GL (E)$ satisfying the following additional property.
	For every secondary marked point $p \in P_\rm{II}$, the flag $F_p$ is invariant under the isomorphism $\varphi_\gamma : E_p \iso E_p$ where $\gamma \in \pi_1 (\Sigma, p)$ is the boundary loop based at $p$.
\end{definition}

Since $E$ and $E'$ are just finite collections of vector spaces, a bundle map $\phi$ is nothing but a finite collection of linear maps
\begin{eqn}
	\phi = \set{ \phi_p : (E_p, F_p) \to (E'_p, F'_p) ~\big|~ p \in P}.
\end{eqn}

\begin{definition}
	\label{250914115030}
	A morphism $\phi : (E, F, \varphi) \to (E', F', \varphi')$ of two filtered representations of $\pi_1 (\Sigma, P)$ is a filtered bundle map $\phi : (E,F) \to (E',F')$ that intertwines the representations $\varphi, \varphi'$.
	We denote the groupoids and the corresponding moduli spaces of filtered representations of $\pi_1 (\Sigma, P)$ of rank $n$ by
	\begin{eqn}
		\catname{Rep}_G^\rm{Fi} \big( \pi_1 (\Sigma, P) \big)
		\qtext{and}
		\frak{Rep}_G^\rm{Fi} \big( \pi_1 (\Sigma, P) \big)
		\coleq \catname{Rep}_G^\rm{Fi} \big( \pi_1 (\Sigma, P) \big) \big/{\sim}.
	\end{eqn}
\end{definition}

Analogously, in the framed and projectively framed case we give the following definitions.

\begin{definition}
	\label{250912195848-framed}
	Let $(\Sigma, P)$ be any with marked boundary.
	A \dfn{framed representation} of the discrete fundamental groupoid $\pi_1 (\Sigma, P)$ is a quadruple $(E, F, \varphi,\beta)$ consisting of a filtered representation $(E,F, \varphi)$ together with a collection $\beta = \set{ \beta_p : p \in P }$  of local frames at each marked point.
\end{definition}

\begin{definition}
	\label{250914115030-framed}
	A unipotent morphism $\phi : (E, F, \varphi,\beta) \to (E', F', \varphi',\beta')$ of two framed representations of $\pi_1 (\Sigma, P)$ is a 
	filtered isomorphism such that each restriction $\phi_p : (E_p, F_p) \to (E'_p, F'_p)$ is a unipotent map with respect to the given frames at $p$.
	We denote the groupoids and the corresponding moduli spaces of framed representations of $\pi_1 (\Sigma, P)$ of rank $n$ by
	\begin{eqn}
		\catname{Rep}_G^\rm{Fr} \big( \pi_1 (\Sigma, P) \big)
		\qtext{and}
		\frak{Rep}_G^\rm{Fr} \big( \pi_1 (\Sigma, P) \big)
		\coleq \catname{Rep}_G^\rm{Fr} \big( \pi_1 (\Sigma, P) \big) \big/{\sim}.
	\end{eqn}
\end{definition}

\begin{definition}
	\label{proj-framed-rep-obj}
	Let $(\Sigma, P)$ be any surface  with marked boundary.
	A \dfn{projectively framed representation} of the discrete fundamental groupoid $\pi_1 (\Sigma, P)$ is a quadruple $(E, F, \varphi,\overline\beta)$ consisting of a filtered representation $(E,F, \varphi)$ together with a collection $\overline\beta = \set{ \overline\beta_p : p \in P }$  of projective local frames at each marked point.
\end{definition}

\begin{definition}
	\label{proj-framed-rep-morph}
	A projectively unipotent morphism $\phi : (E, F, \varphi,\overline\beta) \to (E', F', \varphi',\overline\beta')$ of two framed representations of $\pi_1 (\Sigma, P)$ is a 
	filtered isomorphism such that each restriction $\phi_p : (E_p, F_p) \to (E'_p, F'_p)$ is a projectively unipotent map with respect to the given projective frames at $p$.
	We denote the groupoids and the corresponding moduli spaces of framed representations of $\pi_1 (\Sigma, P)$ of rank $n$ by
	\begin{eqn}
		\catname{Rep}_{ G}^\rm{PFr} \big( \pi_1 (\Sigma, P) \big)
		\qtext{and}
		\frak{Rep}_{G}^\rm{PFr} \big( \pi_1 (\Sigma, P) \big)
		\coleq \catname{Rep}_{ G}^\rm{PFr} \big( \pi_1 (\Sigma, P) \big) \big/{\sim}.
	\end{eqn}
\end{definition}

In respect of these definitions, the restriction functor of \autoref{250914172733} extends to an equivalence of categories for the corresponding filtered representations.
We summarise this as follows.

\begin{lemma}
	\label{251107085612}
	Let $(\Sigma, P)$ be any surface with marked boundary with nonempty $P$.
	Then the restriction functor
	\begin{eqn}
		\catname{Rep}_G^\square \big( \Pi_1 (\Sigma), P \big)
		\iso
		\catname{Rep}_G^\square \big( \pi_1 (\Sigma, P) \big)
	\end{eqn}
	is an equivalence of categories for any $\square \in \set{ \rm{Fi}, \rm{Fr}, \rm{PFr} }$.
\end{lemma}

Together with \autoref{251216105015} this lemma immediately implies the following result.

\begin{proposition}
	\label{250912193234}
	For any  surface with marked boundary $(\Sigma, P)$ with nonempty $P$, the restricted holonomy representation functor gives an equivalence of  
	\begin{eqn}
		\op{Hol} : \catname{Loc}_G^\square  (\Sigma, P) \iso \catname{Rep}_G^\square  \big( \pi_1 (\Sigma, P) \big),
	\end{eqn}
	for any $\square \in \set{ \rm{Fi}, \rm{Fr}, \rm{PFr} }$.
\end{proposition}

\begin{remark}
	By dropping the decorations, and restricting $P$ to a single point, this result reduces to the second part of \autoref{250911161735}.
\end{remark}

We can regard these equivalences as providing in a sense a \textit{discrete model} for the groupoid and moduli space of decorated local systems.

\section{Forgetting secondary points}\label{se:5nuova}

In this section, we study the forgetful functors that allow to drop all secondary marked points. In particular, we prove that all vertical arrows in the diagram
\begin{equation}\label{diagram}
    \begin{tikzcd}
		\mathfrak{Loc}_G^\rm{Fi} (\Sigma, P) 
        \arrow[r, phantom, "\cong"]
		\ar[d]
		& \mathfrak{Rep}_G^\rm{Fi} \big( \Pi_1 (\Sigma), P \big)
        \arrow[r, phantom, "\cong"]
		\ar[d]
        & \mathfrak{Rep}_G^\rm{Fi} \big( \pi_1 (\Sigma, P) \big)
		\ar[d]
		\\
        \mathfrak{Loc}_G^\rm{Fi} (\Sigma, P_{\rm I}) 
        \arrow[r, phantom, "\cong"]
		& \mathfrak{Rep}_G^\rm{Fi} \big( \Pi_1 (\Sigma), P_{\rm I} \big)
        \arrow[r, phantom, "\cong"]
        & \mathfrak{Rep}_G^\rm{Fi} \big( \pi_1 (\Sigma, P_{\rm I}) \big)
	\end{tikzcd}
\end{equation}
are ramified coverings of degree $(n!)^{r}$ where $r$ is the number of secondary marked points; i.e., $r = |P_\rm{II}|$. We also describe the fibres of such coverings. In order to do so, we need to study the stratification of flags and dive into Jordan types.

\subsection{Stratification by Relative Position of Adjacent Flags}\label{se:3.2}

Filtered local systems on a surface with marked boundary have a well-defined collection of invariants obtained by calculating the relative positions of all adjacent pairs of filtrations.
The meaning of `adjacent' here is unambiguously provided by the boundary segments of the irregular boundary circles.

Recall that, given a pair of (complete) flags $F, F'$ on an $n$-dimensional vector space, their \dfn{relative position} is defined as the permutation matrix $\sigma \coleq \op{pos} (F, F') \in W$ whose entries are given by
\begin{eqn}
	\op{pos} (F, F')_{ij} \coleq \dim \left( \frac{ F_i \cap F'_j }{ F_i \cap F'_{j-1} + F_{i-1} \cap F'_j } \right).
\end{eqn}

Before proceeding, we mention some basic properties of the relative position matrix.
If $F$ and $F'$ are two flags of lengths $m$ and $m'$ in an $n$-dimensional vector space, then $\sigma = \op{pos} (F,F')$ is an $m {\times} m'$-matrix whose entries sum up to $n$.
If $F$ and $F'$ are complete flags, then $\sigma$ is a permutation matrix.
If $\sigma$ is the skew-identity matrix, then we say that $F$ and $F'$ are \dfn{transverse} (or in \dfn{general position}) and write $F \pitchfork F'$.
Isomorphisms preserve the relative position of two flags: $\op{pos} \big(\phi (F), \phi(F')\big) = \op{pos} (F, F')$ for any linear isomorphism $\phi : E \iso E'$.
Switching the order of the flags transposes the relative position matrix: $\op{pos} (F, F') = \op{pos} (F', F)^{\top}$.

Let $(\Sigma, P)$ be any surface with marked boundary with a least one primary marked point.
Pick a positively-oriented boundary segment $\delta \in \frak{A} = \frak{A} (\Sigma, P)$ with source $\delta (0) = p_1$ and target $\delta (1) = p_2 \in P_\rm{I}$.
If $(\cal{E}, \cal{F})$ is a filtered local system on $(\Sigma, P)$, let $(E_{p_i},F_{p_i})$ be the stalk of $(\cal{E}, \cal{F})$ at $p_i$.
Then we can compare the two (not necessarily distinct) flags $F_{p_1}$ and $F_{p_2}$ by parallel transporting them to any common point $p$ in $\delta$.
Concretely, let $E_p$ be the stalk of $\cal{E}$ at $p \in \delta$, and let $\op{hol}_i : E_{p_i} \iso E_p$ be the holonomy of $\cal{E}$ along a path from $p_i$ to $p$ contained in $\delta$.

\begin{definition}
	\label{251211174256}
	The \dfn{relative position} of the filtrations $\cal{F}_{p_1}, \cal{F}_{p_2}$ on $\cal{E}$ over $\delta$, or equivalently of the flags $F_{p_1}, F_{p_2}$ over $\delta$, is defined as the relative position of their parallel transport to any common point $p \in \delta$:
	\begin{eqn}
		\op{pos} (\cal{F}_{p_1}, \cal{F}_{p_2}) =
		\op{pos} (F_{p_1}, F_{p_2}) \coleq
		\op{pos} \Big( \op{hol}_1 F_{p_1}, \op{hol}_2 F_{p_2} \Big) \in W.
	\end{eqn}
\end{definition}

The matrix $\op{pos} (\cal{F}_{p_1}, \cal{F}_{p_2})$ is independent of $p$ and hence well-defined because the relative position of two flags in a vector space is invariant under isomorphisms.
For exactly the same reason, it is also preserved by isomorphisms of filtered local systems and therefore defines an invariant.
Namely, it is a collection of permutation matrices $\sigma_\delta$ labelled by the boundary segments $\delta \in \frak{A} = \frak{A} (\Sigma, P)$; i.e., an element
\begin{eqn}
	\bm{\sigma} \coleq (\sigma_\delta)  \in W_{\frak{A}} 
	\coleq \prod_{\delta \in \frak{A}} W.
\end{eqn}

\begin{definition}
	\label{251124131709}
	We say that a filtered local system $(\cal{E}, \cal{F})$ on $(\Sigma, P)$ has \dfn{adjacent flags in relative position} $\bm{\sigma} \in W_\frak{A}$ if $\op{pos} \big( \cal{F}_{\delta (0)}, \cal{F}_{\delta (1)}) = \sigma_\delta$ for every $\delta \in \frak{A}$.
	In this case, we write $\op{pos} (\cal{E}, \cal{F}) \coleq \bm{\sigma}$.
	Such decorated local systems form a full subgroupoid of $\catname{Loc}_G^{\rm{Fi}} (\Sigma, P)$.
	We denote it and its moduli space by
	\begin{eqn}
		\catname{Loc}_G^{\rm{Fi}} (\Sigma, P ; \bm{\sigma})
		\qtext{and}
		\mathfrak{Loc}_G^{\rm{Fi}} (\Sigma, P ; \bm{\sigma})
		\coleq
		\catname{Loc}_G^{\rm{Fi}} (\Sigma, P ; \bm{\sigma})
		/ {\sim}.
	\end{eqn}
\end{definition}

The fact that these are full subgroupoids implies immediately that the groupoid of filtered local systems decomposes into the disjoint union of all these subgroupoids given by fixing relative positions of adjacent flags.

\begin{proposition}
	\label{251124134750}
	The groupoid of filtered local systems on $(\Sigma, P)$ decomposes into a disjoint union of full subgroupoids enumerated by the finite set $W_\frak{A}$ of all possible relative positions of pairs of flags at adjacent primary marked points:
	\begin{eqn}
		\catname{Loc}_G^{\rm{Fi}} (\Sigma, P)
		= \Coprod_{\bm{\sigma} \in W_\frak{A}} \catname{Loc}_G^{\rm{Fi}} (\Sigma, P ; \bm{\sigma}).
	\end{eqn}
	Consequently, the moduli space of filtered local systems admits a similar decomposition:
	\begin{eqn}
		\mathfrak{Loc}_G^\rm{Fi} (\Sigma, P) 
		= \Coprod_{\bm{\sigma} \in W_\frak{A}} \mathfrak{Loc}_G^{\rm{Fi}} (\Sigma, P ; \bm{\sigma}).
	\end{eqn}
\end{proposition}

\subsection{Invariant Flags and Shuffled Jordan Type}
\label{251106151533}

Let $J \subset G$ denote the subset of all invertible Jordan matrices.
If $\varphi \in \Aut (E)$ is an automorphism of an $n$-dimensional vector space $E$, let $J (\varphi) \subset J$ denote the subset of all Jordan matrices that have the same Jordan type as the Jordan type $\op{Jt} (\varphi)$ of $\varphi$:
\begin{eqn}
	J (\varphi) \coleq \set{ \J \in J ~\big|~ \op{Jt} (\J) = \op{Jt} (\varphi) }.
\end{eqn}
Recall that $J (\varphi)$ is a finite set whose cardinality equals the number of all possible permutations of the distinct eigenvalues of $\varphi$ and of the Jordan blocks for the same eigenvalue.
In particular, if $\varphi$ has all-distinct eigenvalues, then $| J (\varphi) | = n!$.

Here, we denote by $W\subset G$ the subgroup of permutations.

\begin{definition}
	\label{251216084311}
	An invertible \dfn{shuffled Jordan matrix} is any upper-triangular matrix of the form $\Z \coleq \tau \J \tau^{-1}$ where $\J \in J$ is an invertible Jordan matrix and $\tau \in W$ is a permutation matrix.
	Let $J_W \subset G$ denote the subset of all shuffled Jordan matrices, and let $J_W (\varphi) \subset J_W$ be the finite subset of all shuffled Jordan matrices that have the same Jordan type as a given automorphism $\varphi \in \Aut (E)$:
	\begin{eqns}
		J_W
		&\coleq \set{ \Z = \tau \J \tau^{-1} \in B ~\Big|~ \J \in J \text{ and } \tau \in W },
		\\
		J_W (\varphi) 
		&\coleq \set{ \Z = \tau \J \tau^{-1} \in B ~\Big|~ \J \in J (\varphi) \text{ and } \tau \in W }.
	\end{eqns}
\end{definition}
Equivalently, $J_W$ is the set of all pairs $(\J, \tau) \in J \times W$ such that $\Z = \tau \J \tau^{-1}$ is upper-triangular, considered up to equivalence: $(\J, \tau) \sim (\J', \tau')$ if and only if $\Z = \Z'$ which means $\J = (\tau^{-1} \tau') \J' (\tau^{-1} \tau')^{-1}$.

Note that, if a Jordan matrix $\J \in J$ is diagonal, then of course the conjugate $\Z = \tau \J \tau^{-1}$ by any permutation $\tau \in W$ remains diagonal.
In general, however, $\Z = \tau \J \tau^{-1}$ remains upper-triangular for only a subset $W (\J) \subset W$ of permutations which we call \dfn{shuffles} of $\J$.
Also, note that a shuffled Jordan matrix $\Z \in J_W$ may not be a Jordan matrix.
See \autoref{251119161726} for more details and examples.

\subsection{Forgetful functors}

We start by proving the right vertical arrow in diagram \eqref{diagram}.

Let $P_\rm{II} = \set{p_1, \ldots, p_r}$.
For each $i = 1, \ldots, r$, denote by $\delta_i \in \pi_1 (\Sigma, p_i)$ the boundary loop based at $p_i$. 

\begin{proposition}
	\label{251007202059}
	For any  surface with marked boundary $(\Sigma, P)$, the surjective map
	\begin{eqn}
		\mathfrak{Rep}_G^\rm{Fi} \big( \pi_1 (\Sigma, P) \big)
		\to
		\mathfrak{Rep}_G^\rm{Fi} \big( \pi_1 (\Sigma, P_\rm{I}) \big)
	\end{eqn}
	induced by the forgetful restriction functor is a ramified covering of degree $(n!)^{r}$ where $r$ is the number of secondary marked points; i.e., $r = |P_\rm{II}|$.
	The cardinality of the fibre can be read off from the Jordan type of the local monodromy at each secondary marked point.
	Namely, the fibre above an isomorphism class $[(E,F,\varphi)] \in \mathfrak{Rep}_G^\rm{Fi} \big( \pi_1 (\Sigma, P_\rm{I} ) \big)$ is canonically in bijection with the set $J_W (\varphi_{\delta_1}) \times \cdots \times J_W (\varphi_{\delta_r})$ where each $\delta_i$ is the boundary loop at the secondary marked point $p_i \in P_\rm{II}$.
	In particular, the complement of the branch locus is the open dense subset consisting of all filtered representations $(E,F,\varphi)$ such that each automorphism $\varphi_{\delta_i} \in \Aut (E_{p_i})$ has all-distinct eigenvalues.
\end{proposition}

\begin{proof}
	We consider the case $P_\rm{I} \neq \emptyset$; the proof in the other case is almost completely identical.
	We will explain that any filtered representation of $\pi_1 (\Sigma, P_\rm{I})$ can be lifted to a filtered representation of $\pi_1 (\Sigma, P)$ in finitely many ways up to isomorphism.
	Given a filtered representation $(E, \varphi, F)$ of $\pi_1 (\Sigma, P_\rm{I})$, the underlying undecorated representation $(E, \varphi)$ lifts to an undecorated representation $(\tilde{\E}, \tilde{\Phi})$ of $\Pi_1 (\Sigma)$ by \autoref{250914172733}.
	Then we restrict $(\tilde{\E}, \tilde{\Phi})$ to an undecorated representation $(\tilde{E}, \tilde{\varphi})$ of $\pi_1 (\Sigma, P)$.
	By \eqref{251010161315} and \eqref{251010161331}, the fibre $\tilde{E}_p$ of $\tilde{E}$ at any $p \in P_\rm{I}$ is canonically isomorphic to the fibre $E_p$ of $E$ at $p$.
	Using these canonical isomorphisms, we transport the flags $F$ in all fibres of $E \to P_\rm{I}$ to the corresponding fibres of $\tilde{E} |_{P_\rm{I}} \to P_\rm{I}$, and we denote these flags by the same letter $F$.
	Thus, we have constructed a \textit{partially} filtered representation $(\tilde{E}, \tilde{\varphi}, F)$ of $\pi_1 (\Sigma, P)$, namely a representation that is filtered only on the fibers above $P_{\rm I}$.
	
	Next, for every $p \in P \smallsetminus P_\rm{I}$, let $\gamma \in \pi_1 (\Sigma, p)$ be the boundary loop based at $p$.
	Consider the vector space $\tilde{E}_p$ with the automorphism $\tilde{\varphi}_p \coleq \tilde{\varphi}_\gamma$.
	Endow $\tilde{E}_p$ with any $\tilde{\varphi}_p$-invariant flag $\tilde{F}_p$.
	The result is a filtered representation $(\tilde{E}, \tilde{\varphi}, \tilde{F})$ of $\pi_1 (\Sigma, P)$ whose forgetful restriction to a filtered representation of $\pi_1 (\Sigma, P_\rm{I})$ is canonically isomorphic to $(E, \varphi, F)$.
	
	If for any $p \in P \smallsetminus P_\rm{I}$ we make another such choice $\tilde{F}'_p$ in $\tilde{E}_p$, then we obtain a different filtered representation $(\tilde{E}, \tilde{\varphi}, \tilde{F}')$ of $\pi_1 (\Sigma, P)$ whose forgetful restriction is canonically isomorphic to $(E, \varphi, F)$.
	Two such filtered representations are equivalent if and only if there exists, for every $p \in P \smallsetminus P_\rm{I}$, an automorphism $\tilde{\psi}_p \in \Aut (\tilde{E}_p)$ which moves $\tilde{F}_p$ to $\tilde{F}'_p$ and commutes with $\tilde{\varphi}_p$ in the sense that $\tilde{\varphi}_p \tilde{\psi}_p = \tilde{\psi}_p \tilde{\varphi}_p$.
	By \autoref{251018183337}, for every $p \in P \smallsetminus P_\rm{I}$, there are only finitely many such equivalences classes of $\tilde{F}_p$.
	
	Finally, it is a Zariski open condition that each $\varphi_{\gamma_i}$ has all-distinct eigenvalues.
\end{proof}

Thanks to \autoref{250912193234}, the following result does not require proof:

\begin{proposition}
	\label{250914113734}
	The surjective map
	\begin{eqn}
		\mathfrak{Rep}_G^\rm{Fi} \big( \Pi_1 (\Sigma), P \big) \too \mathfrak{Rep}_G^\rm{Fi} \big( \Pi_1 (\Sigma), P_\rm{I} \big)
	\end{eqn}
	induced by the forgetful functor is a ramified covering of degree $(n!)^{r}$ where $r$ is the number of secondary marked points; i.e., $r = |P_\rm{II}|$.
	The cardinality of the fibre can be read off from the Jordan type of the local monodromy at each secondary marked point.
	Namely, the fibre above an isomorphism class $[(\E,\F,\Phi)] \in \mathfrak{Rep}_G^\rm{Fi} \big( \Pi_1 (\Sigma), P_\rm{I} \big)$ is canonically in bijection with the set $J_W (\Phi_{\delta_1}) \times \cdots \times J_W (\Phi_{\delta_r})$ where each $\delta_i$ is the boundary loop at the secondary marked point $p_i \in P_\rm{II}$
	In particular, the complement of the branch locus is the open dense subset consisting of all filtered representations $(\E,\F,\Phi)$ such that each automorphism $\Phi_{\delta_i} \in \Aut (E_{p_i})$ has all-distinct eigenvalues.
\end{proposition}

Finally, thanks to \autoref{250910194000}, the following result follows:

\begin{proposition}
	\label{251119124807}
    The surjective map between the moduli spaces of filtered local systems
	\begin{eqntag}
		\label{251119125445}
		\mathfrak{Loc}_G^\rm{Fi} (\Sigma, P) \too \mathfrak{Loc}_G^\rm{Fi} (\Sigma, P_\rm{I})
	\end{eqntag}
	 is a ramified covering of degree $(n!)^{r}$ where $r$ is the number of secondary marked points; i.e., $r = |P_\rm{II}|$.
	The cardinality of the fibre can be read off from the Jordan type of the local system's local monodromy at each secondary marked point.
Namely, the fibre above an isomorphism class $[(\cal{E},\cal{F})] \in \mathfrak{Loc}_G^\rm{Fi} (\Sigma, P_\rm{I})$ is canonically in bijection with the set of $r$-tuples of shuffled Jordan matrices of the same Jordan type as the local monodromies of $\cal{E}$ around the boundary loops $\delta_1, \ldots, \delta_r$; i.e., the set
	\begin{eqn}
		J_W (\varphi_1) \times \cdots \times J_W (\varphi_r)
		\qtext{where}
		\varphi_i \coleq \op{mon}_{\delta_i} (\cal{E}) \in \Aut (E_{p_i}).
	\end{eqn}
	In particular, the complement of the branch locus is the open dense subset consisting of all filtered local systems whose local monodromy at each secondary marked point has all-distinct eigenvalues.
\end{proposition}


\section{Decorated Character Varieties}\label{se:5}

The advantage of considering representations of the discrete fundamental groupoid instead of the continuous fundamental groupoid is that this discrete model allows us to give an explicit description of the moduli space of decorated local systems.
This discrete model is the direct analogue of the character variety in the classical situation without marked points.

\subsection{Mixed Conjugation Action Groupoid}
\label{251217183124}

Let $(\Sigma,P)$ be any  surface with marked boundary and consider the groupoid $\catname{Rep}_G \big( \pi_1 (\Sigma, P) \big)$ of representations of the discrete fundamental groupoid.

Suppose $(E, \varphi)$ is such a representation.
Since the vector bundle $E \to P$ is merely a finite collection of vector spaces $E = \set{ E_p : p \in P }$, it is always possible to choose a global trivialisation $\chi : E \iso \underline{\CC}^n$ where $\underline{\CC}^n \to P$ is the trivial vector bundle over $P$.
It is enough to simply fix an isomorphism $\chi_p : E_p \iso \CC^n$ for every $p \in P$.
Then the groupoid homomorphism $\varphi : \pi_1 (\Sigma, P) \to \GL (E)$ determines a groupoid homomorphism $\rho : \pi_1 (\Sigma, P) \to G$, its \textit{matrix representation} with respect to $\chi$:
\begin{eqntag}
	\label{251217120635}
	\rho \coleq \rm{t}^\ast \chi \circ \varphi \circ \rm{s}^\ast \chi^{-1}
	\qtext{i.e.}
	\rho_\gamma 
	= \rm{t}^\ast \chi_\gamma \circ \varphi_\gamma \circ \rm{s}^\ast \chi^{-1}_\gamma 
	= \chi_{\rm{t} (\gamma)} \circ \varphi_\gamma \circ \chi^{-1}_{\rm{s} (\gamma)},
\end{eqntag}
for any $\gamma \in \pi_1 (\Sigma, P)$.

Similarly, suppose $\phi : (E, \varphi) \to (E', \varphi')$ is an isomorphism of representations and we choose trivialisations $\chi, \chi'$ of $E, E'$.
Then the linear isomorphism $\phi : E \iso E'$ determines an automorphism $g : \underline{\CC}^n \to \underline{\CC}^n$ defined by the commutative diagram
\begin{eqntag}
	\label{251217122636}
	\begin{tikzcd}
		E
		\ar[d, "\phi"']
		\ar[r, "\chi"]
		&		\underline{\CC}^n
		\ar[d, "g"]
		\\
		E'
		\ar[r, "\chi'"']
		&		\underline{\CC}^n
	\end{tikzcd}
	\qqtext{i.e.}
	g \coleq \chi' \circ \phi \circ \chi^{-1}.
\end{eqntag}
Since $\phi$ intertwines the representations $\varphi, \varphi'$ in the sense that $\varphi'_\gamma = \phi_{\rm{t} (\gamma)} \circ \varphi_\gamma \circ \phi^{-1}_{\rm{s} (\gamma)}$ for every $\gamma \in \pi_1 (\Sigma, P)$, it follows that this automorphism $g = \set{ g_p : p \in P} \in \Aut (\underline{\CC}^n)$ intertwines their matrix representations $\rho, \rho' \in \Hom \big( \pi_1 (\Sigma, P), G \big)$ in the sense that $\rho'_\gamma = g_{\rm{t} (\gamma)} \rho_\gamma g_{\rm{s} (\gamma)}^{-1}$.
Of course, an automorphism $g$ of the trivial vector bundle $\underline{\CC}^n \to P$ is nothing but an element of the affine algebraic group
\begin{eqntag}
	\label{251217122508}
	G_P \coleq \prod_{p \in P} G = \Aut (\underline{\CC}^n).
\end{eqntag}
We are therefore naturally led to the \dfn{mixed conjugation} action groupoid
\begin{eqntag}
	\label{251217132726}
	G_P \ltimes \Hom \big( \pi_1 (\Sigma, P), G \big)
	\qqtext{with}
	g . \rho \coleq \rm{t}^\ast g \cdot \rho \cdot \rm{s}^\ast g^{-1}.
\end{eqntag}
More explicitly, for any $\gamma \in \pi_1 (\Sigma, P)$, this group action is 
\begin{eqn}
	(g . \rho)_\gamma 
	= \rm{t}^\ast g_\gamma \cdot \rho_\gamma \cdot \rm{s}^\ast g^{-1}_\gamma
	= g_{\rm{t} (\gamma)} \rho_\gamma g_{\rm{s} (\gamma)}^{-1}.
\end{eqn}

The discussion above shows that a choice of trivialisation $\chi : E \iso \underline{\CC}^n$ for every rank-$n$ vector bundle $E \to P$ determines a functor 
\begin{eqntag}
	\label{251217134437}
	\catname{Rep}_G \big( \pi_1 (\Sigma, P) \big)
	\to G_P \ltimes \Hom \big( \pi_1 (\Sigma, P), G \big)
\end{eqntag}
sending representations $(E,\varphi)$ and morphisms $\phi$ to their matrix representations $\rho$ and $g$ given by \eqref{251217120635} and \eqref{251217122636}.

Making a different choice of trivialisations yields different matrix representations and therefore a different functor.
However, any two trivialisations $\chi, \tilde{\chi} : E \iso \underline{\CC}^n$ of the same bundle differ by an automorphism $g \coleq \tilde{\chi} \circ \chi^{-1} \in \Aut (\underline{\CC}^n) = G_P$ so that $\tilde{\chi} = g \chi$.
As a result, the two matrix representations $\rho, \tilde{\rho}$ of $\varphi : \pi_1 (\Sigma, P) \to \GL (E)$ are related by mixed conjugation by $g$ because $\tilde{\rho} = \rm{t}^\ast g \cdot \rho \cdot \rm{s}^\ast g^{-1}$.
Consequently, all such functors \eqref{251217134437} are naturally isomorphic, and in fact they form an equivalence of categories as stated in the following:

\begin{lemma}
	\label{251216143310}
	For any surface with marked boundary $(\Sigma, P)$, there is a canonical equivalence of categories
	\begin{eqn}
		\catname{Rep}_G \big( \pi_1 (\Sigma, P) \big) 
		\cong G_P \ltimes \Hom \big(\pi_1 (\Sigma, P), G \big).
	\end{eqn}
	Consequently, for any $p \in P$, there is a canonical bijection
	\begin{eqn}
		\text{$\Hom \big(\pi_1 (\Sigma, P), G \big) / G_P$}
		\cong \text{$\Hom \big(\pi_1 (\Sigma, p), G \big) / G$}
		= \text{$\mathfrak{X}_G (\Sigma, p)$}.
	\end{eqn}
\end{lemma}

\begin{proof}
	This equivalence of categories is given by the inclusion functor from right to left.
	Namely, it sends any homomorphism $\rho \in \Hom \big( \pi_1 (\Sigma, P), G)$ to the representation $(E,\varphi) \coleq (\underline{\CC}^n, \rho)$ on the trivial vector bundle and regards any group element $g \in G_P = \Aut (\underline{\CC}^n)$ as a bundle automorphism.
	Any inverse functor is given by matrix representation \eqref{251217134437} and requires a choice of trivialisation for every vector bundle $E \to P$.
	We just have to check that these functors are indeed inverse to one another.
	
	Going from right to left to right is obvious because the trivial vector bundle $\underline{\CC}^n \to P$ has a canonical trivialisation, so the functors compose to send $\rho \mapsto (\underline{\CC}^n, \rho) \mapsto \rho$ and similarly on morphisms.
	So the right-left-right composition of these functors is equal to the identity functor.
	
	For going from left to right to left, suppose $(E, \varphi)$ is a representation of $\pi_1 (\Sigma, P)$ and let $\rho$ be its matrix representation with respect to the chosen trivialisation $\chi$.
	Then the functors compose to send $(E, \varphi) \mapsto \rho \mapsto (\underline{\CC}^n, \rho)$.
	But then $\chi$ itself defines the desired isomorphism $(E, \varphi) \iso (\underline{\CC}^n, \rho)$.
	Similarly, given any morphism $\phi : (E, \varphi) \to (E', \varphi')$, the functors compose to send $\phi \mapsto g = \chi' \circ \phi \circ \chi^{-1} \mapsto g$.
	This means $g \in \Aut (\underline{\CC}^n)$ and $\phi$ form a commutative square
	\begin{eqn}
		\begin{tikzcd}
			(E, \varphi)
			\ar[d, "\phi"']
			\ar[r, "\chi"]
			&		(\underline{\CC}^n, \rho)
			\ar[d, "g"]
			\\
			(E', \varphi')
			\ar[r, "\chi'"']
			&		(\underline{\CC}^n, \rho').
		\end{tikzcd}
	\end{eqn}
	Thus, the left-to-right-to-left composition of the functors defined above is naturally equivalent to the identity functor.
\end{proof}

\subsection{Decorated Representation Variety}

Whereas the set of groupoid homomorphisms $\Hom \big( \Pi_1 (\Sigma), G \big)$ is an infinite-dimensional complex manifold (it is a mapping space), its `discrete' counterpart $\Hom \big( \pi_1 (\Sigma, P), G \big)$ is a finite-dimensional algebraic space.
More precisely, we have the following fact that follows immediately from \autoref{251217105341}.

\begin{lemma}
	\label{251215212039}
	For any  surface with marked boundary $(\Sigma, P)$ of genus $g$ with $s \geq 1$ holes, the set of groupoid homomorphisms $\Hom \big( \pi_1 (\Sigma, P), G \big)$ is a smooth affine algebraic variety of dimension $n^2 (2g + s + |P| - 2)$ with a non-canonical isomorphism
	\begin{eqn}
		\Hom \big( \pi_1 (\Sigma, P), G \big) \cong G^{2g + s + |P| - 2}.
	\end{eqn}
	Any such isomorphism is determined by a choice of free presentation of the discrete fundamental groupoid $\pi_1 (\Sigma, P)$.
\end{lemma}

Now we introduce the relevant subset of groupoid homomorphisms.

\begin{definition}
	\label{251218003508}
	Let $(\Sigma, P)$ be a surface with marked boundary.
	Then the \dfn{decorated representation variety} of $(\Sigma, P)$ is the subset $R_G (\Sigma, P)$ of $\Hom \big( \pi_1 (\Sigma, P), G \big)$ consisting of all groupoid homomorphisms $\rho : \pi_1 (\Sigma, P) \to G$ which are upper-triangular on the boundary loop at each secondary marked point:
	\begin{eqn}
		R_G (\Sigma, P)
		\coleq \set{ \rho ~\Big|~ \rho_{\delta_p} \in B \quad \forall p \in P_\rm{II}}
		\subset \Hom \big( \pi_1 (\Sigma, P), G \big),
	\end{eqn}
	where $\delta_p \in \pi_1 (\Sigma, p)$ denotes the boundary loop based at the marked point $p \in P_\rm{II}$.
\end{definition}

Evidently, in the absence of secondary marked points ($P_\rm{II} = \emptyset$), the defining constraint in this definition becomes vacuous, so we get a simplification:
\begin{eqn}
	R_G (\Sigma, P_{\rm I}) = \Hom \big( \pi_1 (\Sigma, P_{\rm I} ), G \big).
\end{eqn}

\begin{proposition}
	\label{251216175808}
	Let $(\Sigma, P)$ be any surface with marked boundary of genus $g$ with $s \geq 1$ holes, $r$ secondary marked points, and $m$ primary marked points.
	Then the decorated representation variety $R_G (\Sigma, P)$ is a smooth, closed, affine algebraic subvariety of $\Hom \big( \pi_1 (\Sigma, P), G \big)$, and there is a non-canonical isomorphism
	\begin{eqn}
		R_G (\Sigma, P) \cong G^{2g + s + m - 2} \times B^r.
	\end{eqn}
	Consequently, $R_G (\Sigma, P)$ has dimension $n^2 (2g + s + m - 2) + \tfrac{1}{2} r (n^2 + n)$.
	Any such isomorphism is realised by choosing a finite and free presentation of the discrete fundamental groupoid $\pi_1 (\Sigma,P )$.
\end{proposition}

We will prove this proposition by exploiting the fact that decorated representation varieties can be described completely explicitly using the convenient finite presentations of the discrete fundamental groupoid $\pi_1 (\Sigma, P)$ constructed in \autoref{251214205432}.

Let $(\Sigma, P)$ be any surface with marked boundary of genus $g$, with $\bm{s} = (l,r,d)$ holes and $\bm{m} = (m_1, \ldots, m_d)$ primary marked points.
Pick any basepoint $p_0 \in P$,  and choose a finite presentation of $\pi_1 (\Sigma, P)$ as described in \autoref{251215203809}:
\begin{eqntag}
	\label{251217165301}
	\pi_1 (\Sigma, P)
	\cong \inner{ \alpha_i, \beta_i, \gamma_i, \delta_i, \delta_{i,j}, c_i, d_i ~\Big|~ \text{ relation \eqref{251109164128}}}.
\end{eqntag}
Recall also that $c_1 = 1$ if $p_0 \in P_\rm{I}$ and $d_1 = 1$ if $p_0 \in P_\rm{II}$.
Then any homomorphism $\rho : \pi_1 (\Sigma, P) \to G$ is completely determined by its values on these generators, which we denote as follows:
\begin{eqntag}
	\label{251217165737}
	\rho :
	~
	\alpha_i \mapsto \A_i~,
	~~
	\beta_i \mapsto \B_i~,
	~~
	\gamma_i \mapsto \M_i~,
	~~
	\delta_i \mapsto \N_i~,
	~~
	\delta_{i,j} \mapsto \T_{ij}~,
	~~
	c_i \mapsto \C_i,~
	~~
	d_i \mapsto \D_i
	~.
\end{eqntag}
All these matrices are elements of $G$, and $\rho$ belongs to the decorated representation variety $R_G (\Sigma, P)$ if and only if each $\N_i \in B$.
Furthermore, we automatically have $\C_1 = \idd$ if $p_0 \in P_\rm{I}$ and $\D_1 = \idd$ if $p_0 \in P_\rm{II}$.
Let us organise this collection of matrices into the following notation:
\begin{eqntag}
	\label{251217170301}
	\begin{gathered}
		\bm{\A} \coleq \set{ \A_1, \ldots, \A_g } \subset G~,
		\qquad
		\bm{\B} \coleq \set{ \B_1, \ldots, \B_g } \subset G~,
		\\
		\bm{\M} \coleq \set{ \M_1, \ldots, \M_l } \subset G~,
		\qquad
		\bm{\N} \coleq \set{ \N_1, \ldots, \N_r} \subset B~,
		\\
		\bm{\C} \coleq \set{ \C_1 \ldots, \C_{d}} \subset G~,
		\qquad
		\bm{\D} \coleq \set{ \D_1 \ldots, \D_{r}} \subset G
		\\
		\bm{\T} \coleq \set{ \T_{ij} ~\big|~ 1 \leq i \leq d, ~ 1 \leq j \leq m_i } \subset G.
	\end{gathered}
\end{eqntag}
These collections of matrices are allowed to be empty; e.g., $\bm{\A},\bm{\B} = \emptyset$ if $g = 0$.
These matrices satisfy the following relation implied by the relation \eqref{251109164128} on the corresponding generators of $\pi_1 (\Sigma, P)$:
\begin{eqntag}
	\label{251107125603}
	\prod_{i=1}^{\substack{g\\\longleftarrow}} [\A_i, \B_i]
	\cdot
	\prod_{i=1}^{\substack{l\\\longleftarrow}} \M_i
	\cdot 
	\prod_{i={1}}^{\substack{{r}\\\longleftarrow}} \D_i^{-1} 
	\N_{i} \D_i
	\cdot
	\prod_{i=1}^{\substack{d\\\longleftarrow}} \C_i^{-1} 
	\left( \prod_{j=1}^{\substack{m_i\\\longleftarrow}} \T_{ij} \right) \C_i
	= \idd.
\end{eqntag}
It follows that the decorated representation variety $R_G (\Sigma, P)$ is in bijection with the set that parameterises the collection of all such matrices satisfying this relation.
To state this more accurately, we first introduce the following definition.

\begin{definition}
	\label{251109184820}
	Fix any nonnegative integer $g$, any vector $\bm{s} \coleq (l,r,d)$ of nonnegative integers, and any vector $\bm{m} \coleq (m_1, \ldots, m_d)$ of strictly positive integers.
	Put $s \coleq l + d + r$ and $m \coleq m_1 + \ldots + m_d$.
	Then we denote the set of all matrices \eqref{251217170301} satisfying the relation \eqref{251107125603} by
	\begin{eqn}
		\hat{R}_G (g,\bm{s},\bm{m}) \coleq \set{ (\bm{\A}, \bm{\B}, \bm{\M}, \bm{\T}, \bm{\C}, \bm{\D}; \bm{\N}) \in G^{2g + s + m} {\times} B^r
			~\Big|~ \textup{ relation \eqref{251107125603} } \big.}.
	\end{eqn}
	We also single out the following subset
	\begin{eqn}
		R_G (g, \bm{s}, \bm{m}) \coleq
		\set{ (\bm{\A}, \bm{\B}, \bm{\M}, \bm{\T}, \bm{\C}, \bm{\D}; \bm{\N})
			~\Big|~ \substack{ \textup{ $\C_1 = \idd$ if $d \geq 1 \phantom{, r\geq 1}$ } \\ \textup{ $\D_1 = \idd$ if $d = 0, r \geq 1$ }} \big.} 
		\subset \hat{R}_G (g, \bm{s}, \bm{m}).
	\end{eqn}
\end{definition}

Observe that the relation \eqref{251107125603} can be solved for any $\M_i, \N_i, \T_{ij}$, but not for any of $\A_i, \B_i, \C_i$.
Thus, eliminating any one of $\M_i, \N_i, \T_{ij}$ from the list destroys the relation.
But beware that using this identity to express any $\N_i$ in terms of all the other matrices in the list does not guarantee that $\N_i$ is upper-triangular.
On the other hand, the matrix $\M_i$ or $\T_{ij}$ (corresponding to one of the loops $\gamma_i$ or boundary segments $\delta_{i,j}$, but not a boundary loop $\delta_i$) is completely determined by the relation \eqref{251107125603}.
It follows that, if $s \geq 1$, then the sets $R_G (g,\bm{s},\bm{m}) \subset \hat{R}_G (g,\bm{s},\bm{m})$ are smooth, closed, affine algebraic subvarieties of $G^{2g+s+m+r}$ isomorphic to $G^{2g+s+m-2} \times B^r$ and $G^{2g+s+m-1} \times B^r$, respectively.
Finally, recall that $\dim G = n^2$ and $\dim B = \tfrac{1}{2} (n^2+n)$.
Therefore, we obtain the following result.

\begin{lemma}
	\label{251107114012}
	The sets $\hat{R}_G (g, \bm{s}, \bm{m})$ and $R_G (g, \bm{s}, \bm{m})$ defined above are smooth, closed, affine algebraic subvarieties of $G^{2g+s+m} \times B^r$ of dimensions $n^2 (2g + s + m - 1) + \tfrac{1}{2} r (n^2 + n)$ and $n^2 (2g + s + m - 2) + \tfrac{1}{2} r (n^2 + n)$, respectively.
	Moreover, if $l + d \geq 1$, then eliminating any one of the matrices $\M_i$ or $\T_{ij}$ (i.e., projecting this component away) determines isomorphisms
	\begin{eqn}
		\hat{R}_G (g, \bm{s}, \bm{m}) \iso G^{2g + s + m - 1} {\times} B^r
		\qtext{and}
		R_G (g, \bm{s}, \bm{m}) \iso G^{2g + s + m - 2} {\times} B^r.
	\end{eqn}
\end{lemma}

Finally, the following lemma summarises the discussion that led to \autoref{251109184820}.

\begin{lemma}
	\label{251217181238}
	Let $(\Sigma, P)$ be any surface with marked boundary of genus $g$, with $\bm{s} = (l,r,d)$ holes and $\bm{m} = (m_1, \ldots, m_d)$ primary marked points.
	Pick any basepoint $p_0 \in P$.
	Choose a finite presentation of the discrete fundamental groupoid $\pi_1 (\Sigma, P)$ of the form \eqref{251217165301}.
	Then the correspondence \eqref{251217165737} defines a bijection
	\begin{equation}
		\label{251217183742}
		R_G (\Sigma, P) \iso R_G (g, \bm{s}, \bm{m}).
	\end{equation}
\end{lemma}

\begin{proof}[Proof of \autoref{251216175808}.]
	This follows from \autoref{251107114012} and \autoref{251217181238}.
	We endow $\hat{R}_G (\Sigma, P)$ with the unique algebraic structure for which the bijective map \eqref{251217183742} is algebraic.
	Then the subset $R_G (\Sigma, P)$ inherits the same algebraic structure.
	One only has to check that if we change the finite presentation of the discrete fundamental groupoid, this algebraic structure does not change.
	But this is clear because any other set of generators of the discrete fundamental groupoid are expressible in terms of the old generators as finite concatenations.
	This means any two bijections obtained this way are related by an algebraic automorphism of $\hat{R}_G (g, \bm{s}, \bm{m})$.
\end{proof}

\subsection{Filtered, Framed and Projectively Framed Character Varieties}
\label{251127191014}

We are now ready to define decorated character varieties.
To motivate their definition, let us briefly return to the discussion we started in \autoref{251217183124} and examine the groupoid of \textit{filtered} representations $\catname{Rep}_G^\rm{Fi} \big( \pi_1 (\Sigma, P) \big)$.

Suppose $(E,F,\varphi)$ is a filtered representation of $\pi_1 (\Sigma, P)$.
It is always possible to choose an isomorphism $\chi_p : E_p \iso \CC^n$ for every $p \in P$ which sends the flag $F_p$ to the standard flag $\CC^\bullet$; i.e., a filtered isomorphism $\chi_p : (E_p, F_p) \iso (\CC^n, \CC^\bullet)$.
This defines a global trivialisation $\chi \coleq \set{ \chi_p : p \in P } : E \iso \underline{\CC}^n$ which therefore sends the flag in each fibre to the standard flag; i.e., a filtered isomorphism $\chi : (E,F) \iso (\underline{\CC}^n, \underline{\CC}^\bullet)$.
Any other such filtered isomorphism $\tilde{\chi} : (E,F) \iso (\underline{\CC}^n, \underline{\CC}^\bullet)$ differs from $\chi$ by an automorphism that preserves the standard flag in each fibre of $\underline{\CC}^n$; i.e., by $g = \tilde{\chi} \circ \chi^{-1} \in \Aut (\underline{\CC}^n, \underline{\CC}^\bullet)$.
So each component of $g = \set{ g_p : p \in P}$ is just an upper-triangular matrix, $g_p \in B$.

Similarly, suppose $\phi : (E, F, \varphi) \to (E', F', \varphi')$ is a filtered isomorphism of representations and we choose trivialisations $\chi, \chi'$ of the filtered bundles $(E,F), (E',F')$.
By definition, $\phi$ sends each flag $F_p$ to $F'_p$ in each fibre.
But this means that its matrix representation must preserve the standard flag; i.e., $g = \chi' \circ \phi \circ \chi^{-1} \in \Aut (\underline{\CC}^n, \underline{\CC}^\bullet)$.
So, again, we find that $g_p \in B$ for every $p \in P$.

Motivated by this, we introduce the following three affine algebraic subgroups of $G_P = \Aut (\underline{\CC}^n)$:
\begin{equation}
	\label{251217191045}
	\begin{aligned}
		B_{P} &\coleq \prod_{p \in P} B = \Aut (\underline{\CC}^n, \underline{\CC}^\bullet),\\
		U_P &\coleq \prod_{p \in P} U = \sfop{UAut} (\underline{\CC}^n, \underline{\CC}^\bullet),\\
		U^\times_P &\coleq (\mathbb C^\times)^{|P|} \times U_P,
	\end{aligned}
\end{equation}
where $ \sfop{UAut} (\underline{\CC}^n, \underline{\CC}^\bullet)$ is the set of unipotent automorphisms of $\CC^n$ that preserve the standard flag.

These groups act on the decorated representation variety $R_G (\Sigma, P)$ by mixed conjugation.
So we get three action groupoids on affine algebraic varieties:
\begin{align*}
	&B_P \ltimes R_G (\Sigma, P), \\
	&U_P \ltimes R_G (\Sigma, P), \qqtext{with} g . \rho \coleq \rm{t}^\ast g \cdot \rho \cdot \rm{s}^\ast g^{-1}.\\
	&U^\times_P \ltimes R_G (\Sigma, P)
\end{align*}
Their moduli spaces are the primary subject matter of this paper.

\begin{definition}
	\label{251217191739}
	The \dfn{filtered character stack}, the \dfn{framed character stack} and the \dfn{projectively framed character stack} of any surface with marked boundary $(\Sigma, P)$ with nonempty $P$ are respectively defined as the following affine algebraic quotient stacks:
	\begin{align*}
		&\frak{X}^\rm{Fi}_G (\Sigma, P) \coleq R_G (\Sigma, P) / B_P,\\ 
		&\frak{X}^\rm{Fr}_G (\Sigma, P) \coleq R_G (\Sigma, P) / U_P,\\
		&\frak{X}^\rm{PFr}_G (\Sigma, P) \coleq R_G (\Sigma, P) / U_P^\times.
	\end{align*}
	Collectively, we refer to all these moduli spaces as \dfn{decorated character stacks}.
\end{definition}

The following lemma is a straightforward consequence of the definition of $\mathbb P GL_n$ and of the action of $U_P^\times$, however it is important because it allows to use the Fock-Goncharov-Shen coordinates and their complexification to describe $ \frak{X}^\rm{PFr}_G (\Sigma, P)$:

\begin{lemma}\label{lem:FGS}
	Let 
	\begin{eqn}
		R_{\mathbb P G} (\Sigma, P)
		\coleq \set{ \rho ~\Big|~ \rho_{\delta_p} \in \mathbb P B \quad \forall p \in P_\rm{II}}
		\subset \Hom \big( \pi_1 (\Sigma, P), \mathbb P GL_n \big),
	\end{eqn}
	then we have the following canonical bijection:
	$$ 
	\frak{X}^\rm{PFr}_G (\Sigma, P) \cong R_{\mathbb P G} (\Sigma, P) / \mathbb P U_P.
	$$
\end{lemma}

\begin{proof}
	This is a straightforward consequence of the definition of $\mathbb P GL_n$ and of the action of $U_P^\times$.
\end{proof}

\begin{remark}
	\label{250918103523}
	When $P=P_{\rm I}$, namely there are no
	secondary marked points, we can write the decorated character varieties of $(\Sigma, P_{\rm I})$ more explicitly:
	\begin{align*}
		\frak{X}^\rm{Fi}_G (\Sigma, P_{\rm I}) &= \Hom \big( \pi_1 (\Sigma, P_{\rm I}), G \big) \big/ B_{P_{\rm I}},\\
		\frak{X}^\rm{Fr}_G (\Sigma, P_{\rm I}) &= \Hom \big( \pi_1 (\Sigma, P_{\rm I}), G \big) \big/ U_{P_{\rm I}},\\
		\frak{X}^\rm{PFr}_G (\Sigma, P_{\rm I}) &= \Hom \big( \pi_1 (\Sigma, P_{\rm I}), G \big) \big/ U^\times_{P_{\rm I}}.
	\end{align*}
	The framed character variety of the form 
	$$
	\frak{X}^\rm{Fr}_{SL_n(\mathbb C)} (\Sigma, P_{\rm I}) = \Hom \big( \pi_1 (\Sigma, P_{\rm I}), SL_n(\mathbb C) \big) / U_{P_{\rm I}}
	$$ 
	was introduced by Chekhov-Mazzocco-Rubtsov in \cite[§4]{MR3802126} (who called it the \textit{decorated character variety}) based on the work of Li--Bland-Ševera in \cite[§6]{MR3424475}.
	Thanks to the fact that 
	$$ 
	\Hom \big( \pi_1 (\Sigma, P_{\rm I}), SL_n(\mathbb C) \big) / U_{P_{\rm I}}\to \Hom \big( \pi_1 (\Sigma, P_{\rm I}), \mathbb P G \big) / \mathbb P U_{P_{\rm I}} =  \frak{X}^\rm{PFr}_G (\Sigma, P_{\rm I}),
	$$
	is a $n:1$ covering,
	Lemma \ref{lem:FGS} combined with the forgetful functor of Section \ref{se:5nuova}, implies that  the complexified Fock-Goncharov-Shen variables coordinatize the Chekhov-Mazzocco-Rubtsov {decorated character variety}.
\end{remark}

\begin{lemma}
	\label{251217192359}
	For any surface with marked boundary $(\Sigma,P)$, there are canonical equivalences of categories between the groupoid of decorated representations of the discrete fundamental groupoid $\pi_1 (\Sigma, P)$ and a suitable mixed conjugation action groupoid of the decorated representation variety of $(\Sigma, P)$:
	\begin{align*}
		&\catname{Rep}_G^\rm{Fi} \big( \pi_1 (\Sigma, P) \big) 
		\cong B_P \ltimes R_G (\Sigma, P), \\
		&\catname{Rep}_G^\rm{Fr} \big( \pi_1 (\Sigma, P) \big) 
		\cong U_P \ltimes R_G (\Sigma, P), \\
		&\catname{Rep}_G^\rm{PFr} \big( \pi_1 (\Sigma, P) \big) 
		\cong U^{\times}_P \ltimes R_G (\Sigma, P)
	\end{align*}
\end{lemma}

\begin{proof}
	Explicitly, these equivalences are given by the inclusion functors from right to left. Namely, 
	the equivalence $B_P \ltimes R_G (\Sigma, P) \iso \catname{Rep}_G^\rm{Fi} \big( \pi_1 (\Sigma, P) \big)$
	sends any homomorphism $\rho\in  R_G (\Sigma, P)$ to the representation $(E, F, \Phi) \coleq (\underline{\CC}^n, \underline{\CC}^\bullet, \rho)$.
	Analogously, the equivalence $U_P \ltimes R_G (\Sigma, P) \iso \catname{Rep}_G^\rm{Fr} \big( \pi_1 (\Sigma, P) \big)$
	sends any homomorphism $\rho\in  R_G (\Sigma, P)$ to the representation
	$(E, F, \Phi, \beta) \coleq (\underline{\CC}^n, \underline{\CC}^\bullet, \rho, \idd)$, and similarly the equivalence $U^\times_P \ltimes R_G (\Sigma, P) \iso \catname{Rep}_G^\rm{Fr} \big( \pi_1 (\Sigma, P) \big)$
	sends any homomorphism $\rho\in  R_G (\Sigma, P)$ to the representation
	$(E, F, \Phi, \bar{\beta}) \coleq (\underline{\CC}^n, \underline{\CC}^\bullet, \rho, [\idd])$.
	In each case, an element $g = \set{ g_p } \in B_P$, $U_P$ or $U^\times_P$ is viewed as a vector bundle automorphism $g \in \Aut (\underline{\CC}^n, \underline{\CC}^\bullet)$, $\sfop{UAut} (\underline{\CC}^n, \underline{\CC}^\bullet)$ or $\sfop{U^\times Aut} (\underline{\CC}^n, \underline{\CC}^\bullet)$.
	
	Any inverse functor is given by matrix
	representation which requires a choice of trivialisation
	$\chi : (E, F, \beta) \iso (\underline{\CC}^n, \underline{\CC}^\bullet, \beta_\rm{std})$ for every decorated vector bundle over $P$.
	Once chosen, it is given for any object $(E,\varphi,F,\beta)$ and any morphism $\phi$ by
	\begin{eqn}
		(E,\Phi,F,\beta) \mapsto \rho \coleq \rm{t}^\ast \chi \circ \varphi \circ \rm{s}^\ast \chi^{-1}
		\qqtext{and}
		\phi \mapsto g \coleq \chi' \circ \phi \circ \chi^{-1}.
	\end{eqn}
	
	Just like in the proof of \autoref{251216143310}, we only have to check that the functor suggested above defines an inverse to the inclusion functor.
	One direction is obvious: 
	$$\rho \mapsto (\underline{\CC}^n, \rho, \underline{\CC}^\bullet, \beta_\rm{std}) \mapsto \rho$$ 
	and similarly on morphisms.
	For the other direction, suppose $(E, \varphi, F, \beta)$ is a decorated representation of $\pi_1 (\Sigma, P)$ and let $\rho$ be its matrix representation with respect to the chosen trivialisation $\chi$.
	But then $\chi$ itself defines the desired isomorphism $\chi : (E, \varphi, F, \beta) \iso (\underline{\CC}^n, \rho, \underline{\CC}^\bullet, \beta_\rm{std})$.
	Lastly, an isomorphism of decorated representations $\phi : (E, \varphi, F, \beta) \to (E', \varphi', F', \beta')$ is in particular an isomorphism of representations $\phi : (E,\varphi) \to (E', \varphi')$.
	So its matrix representation defines an element $g = \set{g_p : p \in P} \in G_P$ by $g_p \coleq \chi'_p \circ \phi_p \circ \chi_p^{-1} \in G$ which necessarily belongs to $B_P$, $U_P$ or $U^\times_P$ because $\phi$ respects decorations.
\end{proof}

Combined with \autoref{250912193234}, this lemma implies the following theorem which is one of our main results.

\begin{theorem}
	\label{250913101140}
	For any surface with marked boundary $(\Sigma,P)$, there are canonical equivalences of categories between the groupoids of decorated local systems, decorated representations of the fundamental groupoid $\Pi_1 (\Sigma)$, decorated representations of the discrete fundamental groupoid $\pi_1 (\Sigma, P)$, and suitable mixed conjugation action groupoids over the decorated representation variety $R_G (\Sigma, P)$:
	\begin{eqn}
		\catname{Loc}_G^\square (\Sigma, P) 
		\cong \catname{Rep}_G^\square \big( \Pi_1 (\Sigma), P \big)
		\cong \catname{Rep}_G^\square \big( \pi_1 (\Sigma, P) \big)
		\cong K_P^\square \ltimes R_G (\Sigma, P),
	\end{eqn}
	for any $\square \in \set{\rm{Fi}, \rm{Fr}, \rm{PFr}}$ where $K^\rm{Fi}_P \coleq B_P$, $K^\rm{Fr}_P \coleq U_P$ and $K^\rm{PFr}_P \coleq U^\times_P$.
\end{theorem}

\begin{corollary}
	\label{251007125128}
	For any surface with marked boundary $(\Sigma, P)$ with nonempty $P$, there are canonical bijections between the moduli spaces of decorated local systems, decorated representations of the fundamental groupoid $\Pi_1 (\Sigma)$, decorated representations of the discrete fundamental groupoid $\pi_1 (\Sigma, P)$, and the decorated character stacks of $(\Sigma, P)$:
	\begin{eqn}
		\frak{Loc}_G^\square (\Sigma, P) 
		\cong \frak{Rep}_G^\square \big( \Pi_1 (\Sigma), P \big)
		\cong \frak{Rep}_G^\square \big( \pi_1 (\Sigma, P) \big)
		\cong \frak{X}_G^\square (\Sigma, P),
	\end{eqn}
	for any $\square \in \set{\rm{Fi}, \rm{Fr}, \rm{PFr}}$.
	Consequently, all these moduli spaces have a natural structure of an affine algebraic quotient stack.
\end{corollary}

Collecting the facts from previous sections, we recall that the first of these bijections is given by the holonomy representation functor, the second one is given by the restriction functor, and the third one is given by choosing trivialisations of all vector bundles over $P$.

Recall the ramified covering of moduli spaces of filtered local systems induced by the forgetful functor that drops the flags at secondary marked points, as described in \autoref{se:5nuova}.
The analogue for filtered character varieties reads as follows.

\begin{corollary}
	\label{251217194858}
	Let $(\Sigma, P)$ be any decorated surface.
	Then the surjective map
	\begin{eqn}
		\mathfrak{X}_G^\rm{Fi} (\Sigma, P) \too \mathfrak{X}_G^\rm{Fi} (\Sigma, P_\rm{I})
	\end{eqn}
	induced by the forgetful restriction functor is a ramified covering of degree $(n!)^{r}$ where $r$ is the number of secondary marked points; i.e., $r = |P_\rm{II}|$.
	For any groupoid homomorphism $\rho : \pi_1 (\Sigma, P_\rm{I}) \to G$, the fibre above its equivalence class $[\rho] \in \mathfrak{X}_G^\rm{Fi} (\Sigma, P_\rm{I})$ is canonically in bijection with the set $J_W (\rho_{\delta_1}) \times \cdots \times J_W (\rho_{\delta_r})$ of $r$-tuples of shuffled Jordan matrices of the same Jordan type as the value of $\rho$ on the boundary loops $\delta_1, \ldots, \delta_r$ based at the secondary marked points $P_\rm{II} = \set{p_1, \ldots, p_r}$.
	In particular, if $P$ consists of secondary marked points only ($P = P_\rm{II}$), then the filtered character variety is a ramified covering over the usual character variety:
	\begin{eqn}
		\mathfrak{X}_G^\rm{Fi} (\Sigma, P) \too \mathfrak{X}_G (\Sigma).
	\end{eqn}
\end{corollary}

\subsection{Finite-Dimensional Presentation}

In this subsection, we give an explicit finite-dimensional presentation of decorated character varieties and hence the moduli spaces of decorated local systems and representations encountered in our paper.
The key is \autoref{251107114012} and \autoref{251217181238} which give an explicit finite-dimensional presentation of the decorated representation variety.

Thus, recall the algebraic variety $R_G (g, \bm{s}, \bm{m})$ from \autoref{251109184820}.
We define a left group action on it by the group $G^{m+r}$ as follows.
Write any element $x \in G^{m+r}$ as $x = (x',x'') \in G^{m} {\times} G^r$ where 
\begin{eqn}
	x' = \big(x_{ij} : 1 \leq i \leq d, ~ 1 \leq j \leq m_i \big) \in G^m
	~,\quad
	x'' = (x_{1}, \ldots, x_{r}) \in G^r.
\end{eqn}
Then the action of $x$ on any element $(\bm{\A}, \bm{\B}, \bm{\M}, \bm{\T}, \bm{\C}, \bm{\D}; \bm{\N})$ of $\hat{R}_G (g, \bm{s}, \bm{m})$ is defined by mixed conjugation as follows:
\begin{eqntag}
	\begin{gathered}
		x.\A_i = x_0 \A_i x_0^{-1}
		~,\quad
		x.\B_i = x_0 \B_i x_0^{-1}
		~,
		\\
		x.\M_i = x_0 \M_i x_0^{-1}
		~,\quad
		x.\N_i = x_i \N_i x_i^{-1}
		~,
		\\
		x.\C_i = x_{i1} \C_i x_0^{-1}
		~,\quad
		x.\D_i = x_i \D_i x_0^{-1}
		~,
		\\
		x.\T_{ij} = x_{i,j+1} \T_{ij} x_{ij}^{-1}
		~,
	\end{gathered}
\end{eqntag}
where $j$ is understood mod $m_i + 1$.
The subgroups $U^{m+r}  \subset B^{m+r} \subset G^{m+r+1}$ inherit the same action, yielding two action groupoids
\begin{eqn}
	B^{m+r} \ltimes R_G (g, \bm{s}, \bm{m})
	\qtext{and}
	U^{m+r} \ltimes R_G (g, \bm{s}, \bm{m}).
\end{eqn}

Denote by $R_G (g, \bm{s}, \bm{m})$ the set \begin{eqn}
	R_G (g, \bm{s}, \bm{m}) \coleq
	\set{ (\bm{\A}, \bm{\B}, \bm{\M}, \bm{\T}, \bm{\C}, \bm{\D}; \bm{\N})
		~\Big|~ \substack{ \textup{ $\C_1 = \idd$ if $d \geq 1 \phantom{, r\geq 1}$ } \\ \textup{ $\D_1 = \idd$ if $d = 0, r \geq 1$ }} \big.} 
	\subset \hat{R}_G (g, \bm{s}, \bm{m}).
\end{eqn}
where $(\bm{\A}, \bm{\B}, \bm{\M}, \bm{\T}, \bm{\C}, \bm{\D}; \bm{\N})$ means that we eliminate one generator $M_{ij}$ or $T_i$ by using relation \ref{251107125603}. 

\begin{proposition}
	\label{251217210527}
	Let $(\Sigma, P)$ be any surface with marked boundary with nonempty $P$ of genus $g$ with $\bm{s} = (l,r,d)$ holes and $\bm{m} = (m_1, \ldots, m_d)$ primary marked points.
	Let $s \coleq l + r + d$ and $m \coleq m_1 + \ldots + m_d$.
	Choose any finite presentation of the discrete fundamental groupoid $\pi_1 (\Sigma, P)$ of the form \eqref{251217165301} and distinguish one of the generators $\gamma_i$ or $\delta_{i,j}$.
	Then there are canonical equivalences of categories
	\begin{eqn}
		\catname{Loc}_G^\square (\Sigma, P)
		\cong
		K^{m+r}_\square \ltimes R_G (g, \bm{s}, \bm{m})
		\cong
		K^{m+r}_\square \ltimes ( G^{2g + s + m - 2} {\times} B^r )
	\end{eqn}
	and therefore canonical bijections
	\begin{eqn}
		\mathfrak{Loc}_G^\square (\Sigma, P)
		\cong
		\mathfrak{X}_G^\square (\Sigma, P)
		\cong
		R_G(g, \bm{s}, \bm{m}) / K^{m+r}_\square
		\cong
		( G^{2g + s + m - 2} {\times} B^r ) / K^{m+r}_\square,
	\end{eqn}
	for any $\square \in \set{ \rm{Fi}, \rm{Fr}, \rm{PFr} }$ where $K_\rm{Fi} \coleq B$, $K_\rm{Fr} \coleq U$ and $K_\rm{PFr} \coleq U^\times$.
\end{proposition}

Thus, in particular, if $P$ contains no secondary marked points ($r = 0$), then
\begin{align*}
	\frak{X}_G^{\rm{Fi}} (\Sigma, P) &\cong G^{2g+s+m-2} / B^m, \\
	\frak{X}_G^{\rm{Fr}} (\Sigma, P) &\cong G^{2g+s+m-2} / U^m, \\
	\frak{X}_G^{\rm{Fr}} (\Sigma, P) &\cong G^{2g+s+m-2} / \big( U^\times \big)^m.
\end{align*}

\begin{corollary}
	\label{251109195847}
	Suppose $(\Sigma, P)$ is a surface with marked boundary with nonempty $P$ of genus $g$ with $s$ holes, $m$ primary marked points, and $r$ secondary marked points.
	Then the moduli spaces of decorated local systems $\frak{Loc}_G^\square (\Sigma, P)$ are algebraic stacks of dimension
	\begin{eqns}
		\dim \frak{Loc}_G^\rm{Fi} (\Sigma, P)
		&= (2g+s-2)n^2 + \tfrac{1}{2} m (n^2 - n),
		\\
		\dim \frak{Loc}_G^\rm{Fr} (\Sigma, P)
		&= (2g+s-2)n^2 + \tfrac{1}{2} m (n^2 + n) - rn, \\
		\dim \frak{Loc}_G^\rm{PFr} (\Sigma, P)
		&= (2g+s-2)n^2 + \tfrac{1}{2} m (n^2 + n - 2) - r(n-1).
	\end{eqns}
\end{corollary}

\begin{proof}
	The dimension is calculated as follows:
	\begin{eqns}
		\dim \frak{Loc}_G^\square (\Sigma, P) 
		&= (2g + s + m -2) \dim G + r \dim B - (m+r) \dim K_\square
		\\		&= (2g + s -2) \dim G + m (\dim G - \dim K_\square) + r (\dim B - \dim K_\square)
		\\		&= (2g+s-2)n^2 + \tfrac{1}{2} m (n^2 - n + 2 \nu_\square) - r \nu_\square,
	\end{eqns}
	where $\nu_\rm{Fi} = 0$, $\nu_\rm{Fr} = n$ and $\nu_\rm{PFr} = n-1$.
\end{proof}

Thus, recall the algebraic variety $\hat{R}_G (g, \bm{s}, \bm{m})$ from \autoref{251109184820}.
We define a left group action on it by the group $G^{m+r+1}$ as follows.
Write any element $x \in G^{m+r+1}$ as $x = (x',x'',x_0) \in G^{m} {\times} G^r {\times} G$ where 
\begin{eqn}
	x' = \big(x_{ij} : 1 \leq i \leq d, ~ 1 \leq j \leq m_i \big) \in G^m
	~,\quad
	x'' = (x_{1}, \ldots, x_{r}) \in G^r
	~,\quad
	x_0 \in G.
\end{eqn}
Then the action of $x$ on any element $(\bm{\A}, \bm{\B}, \bm{\M}, \bm{\T}, \bm{\C}, \bm{\D}; \bm{\N})$ of $\hat{R}_G (g, \bm{s}, \bm{m})$ is defined by mixed conjugation as follows:
\begin{eqntag}
	\begin{gathered}
		x.\A_i = x_0 \A_i x_0^{-1}
		~,\quad
		x.\B_i = x_0 \B_i x_0^{-1}
		~,
		\\
		x.\M_i = x_0 \M_i x_0^{-1}
		~,\quad
		x.\N_i = x_i \N_i x_i^{-1}
		~,
		\\
		x.\C_i = x_{i1} \C_i x_0^{-1}
		~,\quad
		x.\D_i = x_i \D_i x_0^{-1}
		~,
		\\
		x.\T_{ij} = x_{i,j+1} \T_{ij} x_{ij}^{-1}
		~,
	\end{gathered}
\end{eqntag}
where $j$ is understood mod $m_i + 1$.
The subgroups $U^{m+r} {\times} G  \subset B^{m+r} {\times} G \subset G^{m+r+1}$ inherit the same action, yielding two action groupoids
\begin{eqn}
	( B^{m+r} {\times} G ) \ltimes \hat{R}_G (g, \bm{s}, \bm{m})
	\qtext{and}
	( U^{m+r} {\times} G ) \ltimes \hat{R}_G (g, \bm{s}, \bm{m}).
\end{eqn}

\appendix

\section{Proof of \autoref{250911161735}}\label{sec:A}

The proof is broken down into several parts.

\textsc{Extension of $\op{Hol}$ and $\op{Mon}_p$ to morphisms.}
The correspondence $\op{Hol}$ defined above on objects extends to a functor because any morphism of sheaves $\phi : \cal{E} \to \cal{E}'$ extends to an $\cal{O}_\Sigma$-linear morphism of sheaves $\phi : \cal{E} \otimes \cal{O}_\Sigma \to \cal{E}' \otimes \cal{O}_\Sigma$ (and hence a continuous vector bundle map $\phi : \E \to \E'$) by acting on the $\cal{O}_\Sigma$ component by the identity.
Since the fibres of $\E$ are canonically isomorphic to the corresponding stalks of $\cal{E}$, the fact that the induced bundle map $\phi : \E \to \E'$ intertwines the holonomy representations (and therefore defines a morphism of representations) is the content of the commutative diagram \eqref{250924114505}.
Similarly, the correspondence $\op{Mon}_p$ defined on objects above extends to a functor because any map of sheaves $\phi : \cal{E} \to \cal{E}'$ restricts to a linear map on stalks $\phi_p : \cal{E}_p \to \cal{E}'_p$ which intertwines the monodromy representations thanks again to the commutative diagram \eqref{250924114505}.

We need to check that these functors are fully faithful and essentially surjective.

\textsc{Functors $\op{Hol}$ and $\op{Mon}_p$ are fully faithful.}
The fully faithful nature of $\op{Hol}$ follows from the fact that the fibres of $\E$ are canonically isomorphic to the corresponding stalks of $\cal{E}$, and maps of sheaves are equal if they are equal stalk-by-stalk.
For faithfulness: if $\phi_1, \phi_2 : \cal{E} \to \cal{E}'$ are two maps of sheaves which the holonomy functor sends to the continuous bundle maps $\tilde{\phi}_1, \tilde{\phi}_2 : \E \to \E'$ such that $\tilde{\phi}_1 = \tilde{\phi}_2$, then $\tilde{\phi}_1, \tilde{\phi}_2$ are in particular equal fibre-by-fibre which means $\phi_1, \phi_2$ are equal stalk-by-stalk.
For fullness: any continuous bundle map $\phi : \E = \cal{E} \otimes \cal{O}_\Sigma \to \E' = \cal{E}' \otimes \cal{O}_\Sigma$ determines linear maps $\phi_p : \E_p \to \E'_p$ fibre-by-fibre and hence linear maps $\tilde{\phi}_p : \cal{E}_p \to \cal{E}'_p$ stalk-by-stalk.
If $\phi$ intertwines the corresponding holonomy representations, then these maps glue into a map of sheaves $\tilde{\phi} : \cal{E} \to \cal{E}'$.
Similarly, $\op{Mon}_p$ is fully faithful because the vector space $\cal{E}_p$ is the stalk of $\cal{E}$ at $p$.

\textsc{Functor $\op{Mon}_p$ is essentially surjective.}
For the essential surjectivity, consider first the monodromy representation functor $\op{Mon}_p$.
Given a representation $(E, \varphi)$ of $\pi_1 (\Sigma, p)$, a local system $\cal{E}$ with monodromy representation isomorphic to $(E, \varphi)$ may be constructed as follows.
Begin by noting that the set $\GL (\CC^n, E)$ of all linear isomorphisms $\CC^n \iso E$ has a natural free and transitive right $G$-action by pre-composition: $f.g = f \circ g$ for any $f \in GL(\CC^n, E)$ and $g \in G = GL(\CC^n)$.
It also has a natural free and transitive left $\GL(E)$-action by post-composition: $h.f = h \circ f$ for any $h \in GL(E)$.
Then the representation $\varphi : \pi_1 (\Sigma, p) \to GL(E)$ induces a left action of the fundamental group $\pi_1 (\Sigma, p)$ on $\GL (\CC^n, E)$ by post-composition: $\gamma . f = \varphi_\gamma \circ f$ for any $\gamma \in \pi_1 (\Sigma, p)$.

At the same time, $\pi_1 (\Sigma, p)$ acts on the universal cover $\pi : \tilde{\Sigma} \to \Sigma$ based at $p$ by deck transformations.
This action restricts to a free and transitive action on the fibre $\pi^{-1} (q) \subset \tilde{\Sigma}$ of any point $q \in \Sigma$, and therefore the quotient
\begin{eqn}
	\hat{\cal{E}}_q 
	\coleq \big( \pi^{-1} (q) \times GL(\CC^n, E) \big) \big/ \pi_1 (\Sigma,p)
\end{eqn}
by the diagonal action of $\pi_1 (\Sigma, p)$ is a set with a free and transitive right $G$-action; i.e., $\hat{\cal{E}}_q$ is a $G$-torsor.
By the same token, $\pi_1 (\Sigma, p)$ acts on the preimage $\pi^{-1} (U) \subset \tilde{\Sigma}$ of any open subset $U \subset \Sigma$ though not necessarily freely and transitively, so the quotient of $\pi^{-1} (U) \times GL(\CC^n, E)$ by the diagonal action of $\pi_1 (\Sigma, p)$ is also a set with a right $G$-action though not necessarily a $G$-torsor.
Nonetheless, if we let $\hat{\cal{E}}$ be the sheaf associated with the presheaf
\begin{eqntag}
	\label{251008145350}
	U \mapsto \big( \pi^{-1} (U) \times GL(\CC^n, E) \big) \big/ \pi_1 (\Sigma,p),
\end{eqntag}
then $\hat{\cal{E}}$ is a locally constant sheaf of sets equipped with a right $G$-action, and its stalk at any $q \in \Sigma$ is nothing but the $G$-torsor $\hat{\cal{E}}_q$.
In other words, $\hat{\cal{E}}$ is a $G$-local system and it gives rise to the linear local system $\cal{E} \coleq \hat{\cal{E}} \times_G \underline{\CC}^n_\Sigma$ as explained in \autoref{251008162022}.
For the benefit of the reader, we remark that, by abuse of notation, the sheaves $\hat{\cal{E}}$ and $\cal{E}$ are sometimes denoted in the literature by 
\begin{eqn}
	\hat{\cal{E}} \eqcol \big(\tilde{\Sigma} \times GL (\CC^n, E) \big) / \pi_1 (\Sigma, p)
	\qtext{and}
	\cal{E} \eqcol \big(\tilde{\Sigma} \times E \big) / \pi_1 (\Sigma, p).
\end{eqn}
The reason this notation makes sense is that the universal cover $\tilde{\Sigma}$ can itself be thought of as a sheaf of sets on $\Sigma$ (in fact, a $\pi_1 (\Sigma, p)$-local system on $\Sigma$).

For future use, the stalks of $\cal{E}$ can be described more explicitly.
Using the canonical isomorphism $GL (\CC^n,E) \times_G \CC^n \iso E$ given by $(f,v) \mapsto f(v)$, we have that the stalk of $\cal{E}$ at any point $q \in \Sigma$ is the quotient by the diagonal action of $\pi_1 (\Sigma, p)$ of $\pi^{-1} (q) \times E$ which is the disjoint infinite union of copies of the vector space $E$ enumerated by the preimages of $q$ on the universal cover $\tilde{\Sigma}$:
\begin{eqn}
	\cal{E}_q 
	= \big( \pi^{-1} (q) \times E \big) \big/ \pi_1 (\Sigma, p).
\end{eqn}
In other words, a vector $e \in \cal{E}_q$ is the same thing as a collection $e = \big \{ e_{\tilde{q}} \in E : \tilde{q} \in \pi^{-1} (q) \big\}$ of infinitely many vectors in $E$, one for each preimage of $q$ in the universal cover $\tilde{\Sigma}$, whose components are related by the action of the fundamental group $\pi_1 (\Sigma, p)$ through the representation $\varphi : \pi_1 (\Sigma, p) \to GL(E)$:
\begin{eqntag}
	\label{251009191043}
	\cal{E}_q 
	= \set{ e \in \pi^{-1} (q) \times E  ~\Big|~ e_{\gamma . \tilde{q}} = \varphi_\gamma (e_{\tilde{q}}) \quad \text{for all $\tilde{q} \in \tilde{U}$ and $\gamma \in \pi_1 (\Sigma, p)$ }},
\end{eqntag}
where $\gamma . \tilde{q}$ denotes the action of the deck transformation determined by $\gamma$.
In particular, since $\pi_1 (\Sigma, p)$ acts freely and transitively on $\pi^{-1} (q)$ for any $q \in \Sigma$, the stalk $\cal{E}_q$ is isomorphic to the vector space $E$ but not canonically.
To fix such an isomorphism is to choose a point $\tilde{q}$ in the fibre $\pi^{-1} (q)$, which is the same thing as choosing a path $\gamma$ on $\Sigma$ that starts at $p$ and ends at $q$.
This path determines the desired isomorphism which is nothing but the holonomy isomorphism $\op{hol}_\gamma : E \iso \cal{E}_q$ of $\cal{E}$ along $\gamma$.
On the other hand, since the chosen universal cover $\tilde{\Sigma}$ is based at $p$, the fibre $\pi^{-1} (p)$ has a distinguished point corresponding to the constant path at $p$.
Therefore, the stalk $\cal{E}_p$ is canonically isomorphic to $E$.

We claim now that the monodromy representation of $\cal{E}$ is isomorphic to $(E, \varphi)$.
Indeed, since the stalk of $\cal{E}$ at $p$ is canonically isomorphic to $E$, the monodromy of $\cal{E}$ around a loop $\gamma \in \pi_1 (\Sigma, p)$ is given by $\gamma . e = \varphi_\gamma (e)$ for any $e \in E$.
This is because under the isomorphism $E \cong GL (\CC^n,E) \times_G \CC^n$, every vector $e \in E$ is of the form $e = f(v)$ for some $(f,v) \in GL (\CC^n,E) \times_G \CC^n$.
And the holonomy of the local system $\hat{\cal{E}}$ along $\gamma$ is computed as $\gamma . f = \varphi_\gamma \circ f$, which means $\gamma. (f,v) = (\gamma.f,v) = \big( \varphi_\gamma \circ f, v \big)$.
Consequently, the functor $\op{Mon}_p$ is essentially surjective, and so the proof that $\op{Mon}_p$ is an equivalence of categories is complete.

\textsc{Functor $\op{Hol}$ is essential surjective.}
Now we turn our attention to the essential surjectivity of the holonomy representation functor $\op{Hol}$.
Given a representation $(\E, \Phi)$ of the fundamental groupoid $\Pi_1 (\Sigma)$, let $(E,\varphi)$ be its restriction to a representation of $\pi_1 (\Sigma, p)$, and then let $\cal{E}$ be the linear local system with $\op{Mon}_p (\cal{E}) \cong (E, \varphi)$ as constructed above.
We claim that the holonomy representation $(\E',\Phi') \coleq \op{Hol} (\cal{E}) = (\cal{E} \otimes \cal{O}_\Sigma, \op{hol})$ is isomorphic to the original representation $(\E, \Phi)$.

The first fact we need is that the representation $\Phi$ determines canonical local trivialisations of the vector bundle $\E$.
More precisely, for any point $q \in \Sigma$, let $\E_q$ be the fibre of $\E$ at $q$, and choose any simply connected neighbourhood $U \subset \Sigma$ of $q$.
Then the homomorphism $\Phi$ induces a vector bundle isomorphism $\psi_q : \E |_U \iso \E_q \otimes \cal{O}_U$ where $\E_q \otimes \cal{O}_U$ is the sheaf of continuous sections of the trivial vector bundle $\E_q \times U \to U$.
Indeed, for every point $q' \in U$, there is a path from $q$ to $q'$ which is entirely contained in $U$.
Such paths are unique up to homotopy with fixed endpoints, so for every $q' \in U$ there is a unique element $\gamma \in \Pi_1 (U)$ which determines an isomorphism $\Phi_\gamma : \E_q \iso \E_{q'}$.
Since $\Phi : \Pi_1 (U) \to GL (\E|_U)$ is a continuous map, as $q'$ varies continuously in $U$, the unique element $\gamma$ varies continuously in $\Pi_1 (U)$, thereby giving a continuous identification of all fibres of $\E$ over $U$ with the fibre $\E_q$.
Furthermore, for any $q' \in U$, the change of trivialisation from $\psi_q$ to $\psi_{q'} : \E |_U \iso \E_{q'} \otimes \cal{O}_U$ is nothing but the isomorphism $\Phi_\gamma$ itself, extended to act on the $\cal{O}_U$-component by the identity:
\begin{eqntag}
	\label{251009194236}
	\psi_{q'} \circ \psi_q^{-1} = \Phi_\gamma : \E_q \otimes \cal{O}_U \iso \E_{q'} \otimes \cal{O}_U.
\end{eqntag}

Next, we observe that, for essentially the same reason, the vector bundle $\E$ is canonically the trivial bundle with fibre $E = \E_p$ once pulled back to the universal cover.
More precisely, let $\pi : \tilde{\Sigma} \to \Sigma$ be the universal cover based at $p$.
Then $\Phi$ induces a bundle isomorphism $\psi : \pi^\ast \E \iso E \otimes \cal{O}_{\tilde{\Sigma}}$ where $E \otimes \cal{O}_{\tilde{\Sigma}}$ is the sheaf of continuous sections of the trivial vector bundle $\tilde{\Sigma} \times E \to \tilde{\Sigma}$.
Indeed, any point $\tilde{q} \in \tilde{\Sigma}$ represents a homotopy class $\gamma \in \Pi_1 (\Sigma)$ of paths with source $p$ and target $q = \pi (\tilde{q})$.
Meanwhile, the fibre of $\pi^\ast \E$ at $\tilde{q}$ is by definition equal to the fibre $\E_q$, but now we have a canonical isomorphism $\Phi_\gamma : \E_p \iso \E_q$.
Since $\tilde{\Sigma}$ is simply connected, this defines a global trivialisation of $\pi^\ast \E$.

The third and final observation is that local sections of $\E$ are precisely the $\varphi$-equivariant local sections of the trivial bundle $\tilde{\Sigma} \times E \to \tilde{\Sigma}$.
More precisely, suppose $\tilde{q}_1, \tilde{q}_2 \in \tilde{\Sigma}$ are points in the same fibre over $q \in \Sigma$.
They represent paths $\gamma_1, \gamma_2 \in \Pi_1 (\Sigma)$ with source $p$ and target $q$ which together form a loop $\gamma_2^{-1} \circ \gamma_1 \in \pi_1 (\Sigma, p)$.
Now, under the global trivialisation $\psi$, two vectors $e_1 \in E \cong \pi^\ast \E_{\tilde{q}_1}$ and $e_2 \in E \cong \pi^\ast \E_{\tilde{q}_2}$ determine the same vector in the fibre $\E_q$ if and only if $\Phi_{\gamma_1} (e_1) = \Phi_{\gamma_2} (e_2)$ i.e. $e_2 = \varphi_\gamma (e_1)$.
Since $\pi^\ast \E |_{\pi^{-1} (q)} \cong \pi^{-1} (q) \times E$, a vector $e$ in $\E_q$ may be regarded equivalently as a collection $e = \big\{ e_{\tilde{q}} \in E : \tilde{q} \in \pi^{-1} (q) \big\}$ of infinitely many vectors in $E$ enumerated by the fibre $\pi^{-1} (q)$ whose components are related by the same action of the fundamental group $\pi_1 (\Sigma, p)$ as in \eqref{251009191043} through the restriction $\varphi : \pi_1 (\Sigma, p) \to GL(E)$ of the given representation $\Phi$.
In other words, for every $q$, there is a canonical isomorphism between the fibre $\E_q$ and the stalk $\cal{E}_q$:
\begin{eqntag}
	\label{251010134146}
	\eta_q : \E_q \iso \cal{E}_q = \big( \pi^{-1} (q) \times E \big) \big/ \pi_1 (\Sigma, p).
\end{eqntag}
We also note that under this isomorphism, the representation $\Phi$ becomes the holonomy of the local system $\cal{E}$; i.e., for every $\gamma \in \Pi_1 (\Sigma)$ with source $q$ and target $q'$, there is a commutative diagram:
\begin{eqntag}
	\label{251009200039}
	\begin{tikzcd}
		&	\E_q
		\ar[r, "\sim" {yshift=-1pt}, "\eta_q"']
		\ar[d, "\rotatebox{90}{$\sim$}", "\Phi_\gamma"']
		& 	\cal{E}_q
		\ar[d, "\rotatebox{90}{$\sim$}"' {xshift=1pt}, "\op{hol}_\gamma"]
		\\
		&	\E_{q'}
		\ar[r, "\sim" {yshift=-1pt}, "\eta_{q'}"']
		&	\cal{E}_{q'}.
	\end{tikzcd}
\end{eqntag}

Combining the above three observations, we conclude that, for every $q \in \Sigma$ and every simply connected neighbourhood $U \subset \Sigma$ of $q$, there is a canonical isomorphism
\begin{eqntag}
	\label{251009111901}
	\E |_U
	\underset{\psi_q}{\iso} \E_q \otimes \cal{O}_U
	\underset{\eta_q}{\iso} \cal{E}_q \otimes \cal{O}_U.
\end{eqntag}
At the same time, since $U$ is simply connected, the restriction of the local system $\cal{E}$ to $U$ is a constant sheaf isomorphic to the constant sheaf with value $\cal{E}_q$.
In particular, it gives a canonical trivialisation of the bundle $\E'$ over $U$:
\begin{eqn}
	\nu_q : \cal{E}_q \otimes \cal{O}_U \iso (\cal{E} \otimes \cal{O}_\Sigma)|_U = \E' |_U.
\end{eqn}
Altogether, we have constructed an isomorphism $\mu_q \coleq \nu_q \circ \eta_q \circ \psi_q : \E |_U \iso \E' |_U$.

We can now choose any open covering of $\Sigma$ by simply connected subsets and glue these local isomorphisms $\mu_q : \E |_U \iso \E' |_U$ into a global isomorphism $\mu : \E \iso \E'$, provided they agree on double overlaps.
To see this, it is enough to choose two points $q, q' \in U$ in an arbitrary simply connected subset $U \subset \Sigma$ and demonstrate that the two isomorphisms $\mu_{q}, \mu_{q'} : \E |_U \iso \E' |_U$ agree.
Let $\gamma \in \Pi_1 (U)$ be the unique-up-to-homotopy path from $q$ to $q'$.
Recall from \eqref{251009194236} that the trivialisations $\psi_q$ and $\psi_{q'}$ of $\E|_U$ are related by the representation $\Phi_\gamma$.
Similarly, the trivialisations $\nu_q$ and $\nu_{q'}$ of $\E'|_U$ are related by the holonomy $\op{hol}_\gamma$ of the local system $\cal{E}$ along $\gamma$.
Thus, unravelling their definition, the isomorphisms $\mu_q$ and $\mu_{q'}$ agree by virtue of the commutativity of the following diagram:
\begin{eqn}
	\begin{tikzcd}
		\E |_U
		\ar[r, "\sim" {yshift=-1pt}, "\psi_q"']
		\ar[d, equal]
		&	\E_q \otimes \cal{O}_U
		\ar[r, "\sim" {yshift=-1pt}, "\eta_q"']
		\ar[d, "\rotatebox{90}{$\sim$}"' {xshift=1pt}, "\Phi_\gamma"]
		& 	\cal{E}_q \otimes \cal{O}_U
		\ar[r, "\sim" {yshift=-1pt}, "\nu_q"']
		\ar[d, "\rotatebox{90}{$\sim$}"' {xshift=1pt}, "\op{hol}_\gamma"]
		&	\E' |_U
		\ar[d, equal]
		\\
		\E |_U
		\ar[r, "\sim" {yshift=-1pt}, "\psi_{q'}"']
		&	\E_{q'} \otimes \cal{O}_U
		\ar[r, "\sim" {yshift=-1pt}, "\eta_{q'}"']
		&	\cal{E}_{q'} \otimes \cal{O}_U
		\ar[r, "\sim" {yshift=-1pt}, "\nu_{q'}"']
		&	\E' |_U.
	\end{tikzcd}
\end{eqn}
This diagram also demonstrates the fact that the isomorphism $\mu : \E \iso \E'$ intertwines the representations $\Phi$ and $\Phi' = \op{hol}$.
This is because they are both groupoid homomorphisms and any path on $\Sigma$ is the concatenation of finitely many `short' paths contained entirely in simply connected subsets where the above diagram is valid.
This completes the proof that $\op{Hol}$ is an equivalence of categories.

\textsc{Functor $\R$ is an equivalence.}
Finally, we show that the restriction functor $\R$ is also an equivalence of categories.
Recall that the monodromy representation functor $\op{Mon}_p$ is the composition of the holonomy representation functor $\op{Hol}$ and the restriction functor.
The essential surjectivity of the restriction functor follows from that of $\op{Mon}_p$ and $\op{Hol}$.
For if $(E,\varphi)$ is any representation of $\pi_1 (\Sigma, p)$, let $\cal{E}$ be a local system such that $\op{Mon}_p (\cal{E}) \cong (E, \varphi)$, and take $(\E, \Phi) \coleq \op{Hol} (\cal{E})$.
Then the restriction of $(\E, \Phi)$ to $\pi_1 (\Sigma, p)$ is literally the monodromy representation of $\cal{E}$, i.e. $\R (\E, \Phi) = \op{Mon}_p (\cal{E})$, which we already know is isomorphic to $(E, \varphi)$.

The restriction functor is obviously full but it is less clear why it is faithful.
That is to say, if $(\E, \Phi), (\E', \Phi')$ is a pair of representations of $\Pi_1 (\Sigma)$, and $(E,\varphi), (E', \varphi')$ are their restrictions to representations of $\pi_1 (\Sigma,p)$, why does the restriction functor induce a bijection between $\R : \sfop{Iso} \big( (\E, \Phi), (\E', \Phi') \big) \iso \sfop{Iso} \big( (E,\varphi), (E', \varphi') \big)$?
This can be argued as follows.
Using the essential surjectivity of $\op{Hol}$ and $\op{Mon}_p$, let $\cal{E}$ be a local system such that $\op{Hol} (\cal{E}) \eqcol (\tilde{\E}, \tilde{\Phi}) \cong (\E, \Phi)$ and $\op{Mon}_p (\cal{E}) \eqcol (\tilde{E}, \tilde{\varphi}) \cong (E, \varphi)$.
Same for $\cal{E}'$.
By virtue of these isomorphisms, there are bijections $\sfop{Iso} \big(({\E},{\Phi}), ({\E}',{\Phi}')\big) \cong \sfop{Iso} \big((\tilde{\E},\tilde{\Phi}), (\tilde{\E}',\tilde{\Phi}')\big)$ and $\sfop{Iso} \big(({E},{\varphi}), ({E}',{\varphi}')\big) \cong \sfop{Iso} \big((\tilde{E},\tilde{\varphi}), (\tilde{E}',\tilde{\varphi}')\big)$.
At the same time, we also have bijections $\op{Hol} : \sfop{Iso} (\cal{E}, \cal{E}') \iso \big((\tilde{\E},\tilde{\Phi}), (\tilde{\E}',\tilde{\Phi}')\big)$ and $\op{Mon}_p : \sfop{Iso} (\cal{E}, \cal{E}') \iso \big((\tilde{E},\tilde{\varphi}), (\tilde{E}',\tilde{\varphi}')\big)$.

\section{Invariant Flags and Shuffled Jordan Types}\label{sec:D}
\label{251020181701}

Suppose $(E, \varphi)$ is an $n$-dimensional complex vector space equipped with an endomorphism $\varphi \in \End (E)$.
Such a pair is sometimes called a \textit{linear system}.
In this subsection, we ask the following question:
What are the possible ways to equip the vector space $E$ with a complete $\varphi$-invariant flag $F$ up to an automorphism that commutes with $\varphi$?
In other words, we want to understand the moduli set
\begin{eqn}
	\sf{Fi}_\varphi (E) \big/ {\sim}
	\qtext{where}
	\sf{Fi}_\varphi (E) \coleq \set{ F \in \sf{Fi} (E) ~\Big|~ \varphi (F_i) \subset F_i \quad \text{for all $i$ } }
\end{eqn}
and $F \sim F'$ if there is a filtered isomorphism $\psi : (E,F) \iso (E,F')$ such that $\varphi \psi = \psi \varphi$.
We remark that for $E = \CC^n$ and $\varphi$ nilpotent, the set $\sf{Fi}_\varphi (\CC^n)$ is called the \textit{Springer fibre} over $\varphi$; see, e.g., \cite{MR3666060}.
We will prove the following fact which, as far we can tell, has not yet appeared in the literature.

\begin{proposition}
	\label{251018183337}
	For any $(E, \varphi)$, the moduli set $\sf{Fi}_\varphi (E) / {\sim}$ is finite and its size can be read off from the Jordan type of $\varphi$.
	In particular, if $\varphi$ has all-distinct eigenvalues, then $\sf{Fi}_\varphi (E) / {\sim}$ has $n!$ points which are in one-to-one correspondence with total orders on the set of $n$ eigenvalues of $\varphi$.
\end{proposition}

In the case of all-distinct eigenvalues, this proposition follows from \autoref{251018210133} below.
The general assertion will follow as a consequence of \autoref{251017181852} below which uses the complete invariant of $\varphi$-invariant flags on $E$ called \textit{shuffled Jordan type}, introduced in \autoref{251017143241} below.
A shuffled Jordan type is essentially a total order on the boxes of the Young tableau of $\varphi$.
Thus, any equivalence class in $\sf{Fi}_\varphi (E)$ can be described as the set of all complete $\varphi$-invariant flags with a fixed shuffled Jordan type.
Said differently, the moduli set $\sf{Fi}_\varphi (E) / {\sim}$ is in bijection with the set of all possible shuffled Jordan types on $(E,\varphi)$, and the size of this set can be deduced from the Jordan type of $\varphi$.
Let us begin by examining a few illustrative examples.

\begin{example}
	\label{251019172215}
	If $E = \CC^2$ and $\varphi = \diag (1,2)$, then the only two possible $\varphi$-invariant flags on $E$ are the standard flag $\CC^\bullet = \big( \inner{e_1} \subset \CC^2 \big)$ and its opposite $\bar{\CC^{\bullet}} = \big( \inner{e_2} \subset \CC^2 \big)$.
	Any automorphism of $\CC^2$ that commutes with $\varphi$ must be a diagonal matrix, but a diagonal matrix cannot move the standard flag to the opposite flag.
	Therefore, the moduli set $\sf{Fi}_\varphi (\CC^2) / {\sim}$ for $\varphi = \diag (1,2)$ has only two elements because the set $\sf{Fi}_\varphi (\CC^2)$ itself is finite and consists of just two inequivalent flags $\CC^\bullet$ and $\bar{\CC^\bullet}$.
	This is not generally the case: if $\varphi = \diag (1,1)$, then any flag $F$ is $\varphi$-invariant, so $\sf{Fi}_\varphi (\CC^2) = \sf{Fi} (\CC^2)$ which is far from a finite set.
	However, all these flags are equivalent because any automorphism of $\CC^2$ commutes with $\varphi$, so $\sf{Fi}_\varphi (\CC^2) / {\sim}$ is a single point.
\end{example}

\begin{example}
	\label{251019183631}
	Next, consider the automorphism of $\CC^3$ given by a single Jordan block:
	\begin{eqn}
		\varphi = \mtx{ 1 & 1 & 0 \\ 0 & 1 & 1 \\ 0 & 0 & 1 }.
	\end{eqn}
	Then the only $\varphi$-invariant flag is the standard flag $\CC^\bullet$, meaning $\sf{Fi}_\varphi (\CC^3)$ is a single point.
	Indeed, if $F$ is a $\varphi$-invariant flag, its first step $F_1$ is a one-dimensional $\varphi$-invariant subspace that must therefore equal the unique one-dimensional $\varphi$-invariant subspace, namely $\inner{e_1}$.
	Towards a contradiction, suppose $F_2$ is not the standard flag, so that $e_2 \not\in F_2$ and we must have $ae_2 + e_3 \in F_2$ for some $a \in \CC$.
	Since $F_2$ is a $\varphi$-invariant subspace, it follows that $\varphi (ae_2 + e_3) = a e_1 + (a+1) e_2 + e_3$ also belongs to $F_2$.
	But this means that both $ae_2 + e_3$ and $(a+1) e_2 + e_3$ must also belong to $F_2$, so it follows that $e_2 \in F_2$, a contradiction.
\end{example}

\begin{example}
	\label{251019183645}
	Finally, consider the automorphism of $\CC^3$ with triple eigenvalue $1$ having two Jordan blocks of size $1$ and $2$:
	\begin{eqn}
		\varphi = \mtx{ 1 & 0 & 0 \\ 0 & 1 & 1 \\ 0 & 0 & 1 }.
	\end{eqn}
	Then the set $\sf{Fi}_\varphi (\CC^3)$ has infinitely-many flags which we organise into three separate groups including a continuous one-parameter family of flags $\set{F^c}_{c \in \CC}$ and two other flags $F'$ and $F''$ defined as follows:
	\begin{eqns}
		F^c &= \big( \inner{e_1 + c e_2} \subset \inner{e_1, e_2} \subset \CC^3 \big),
		\\
		F' &= \big( \inner{e_2} \subset \inner{e_1, e_2} \subset \CC^3 \big),
		\\
		F'' &= \big( \inner{e_2} \subset \inner{e_2, e_3} \subset \CC^3 \big).
	\end{eqns}
	Indeed, $\varphi$ has two eigenvectors $e_1, e_2$ so $F_1$ of any $\varphi$-invariant flag $F$ must be generated by some linear combination $a e_1 + b e_2$.
	If $F_2$ contains the generalised eigenvector $e_3$, then by the same argument as in \autoref{251019183631} it follows that $F_1$ must be generated by $e_2$, yielding $F_2 = \inner{e_2, e_3}$.
	Otherwise, $F_2 = \inner{e_1, e_2}$.
	
	Any invertible matrix $\psi \in GL_3$ that satisfies $\psi \varphi = \varphi \psi$ must be of the following form:
	\begin{eqn}
		\psi = \mtx{ x & 0 & y \\ z & u & v \\ 0 & 0 & u }
	\end{eqn}
	for some $x,y,z,u,v \in \CC$ with $x,u \neq 0$.
	First, we observe that all flags $F^c$ are equivalent to each other because the matrix $\psi$ with $y,v = 0$, $u,x = 1$, and $z = c$
	commutes with $\varphi$ and moves the flag $F^0 = \big( \inner{e_1} \subset \inner{e_1, e_2} \subset \CC^3 \big)$ to the flag $F^c$.
	On the other hand, no $\psi$ that commutes with $\varphi$ can give an equivalence between any two of $F^0, F', F''$.
	Consequently, the moduli set $\sf{Fi}_\varphi (\CC^3) / {\sim}$ consists of three points.
\end{example}

\subsection{Review of Linear Algebra and Jordan Types}
\label{251119120056}

To state and prove \autoref{251017181852}, we first recall some elementary facts from linear algebra.
Suppose $F$ is any complete $\varphi$-invariant flag on $E$.
If $\varphi$ has all-distinct eigenvalues (so that $E$ decomposes as a direct sum of one-dimensional eigenspaces), then in particular $F_1$ must be an eigenspace.
But then $F_2$ must contain an eigenvector of $\varphi$ not in $F_1$, so it decomposes into a direct sum of $F_1$ and another eigenspace.
And so on.
We summarise this as follows.

\begin{lemma}
	\label{251018210133}
	For any $n$-dimensional linear system $(E, \varphi)$ with all-distinct eigenvalues, any $\varphi$-invariant flag $F$ is necessarily of the form
	\begin{eqn}
		F_i = E_{\lambda_1} \oplus \cdots \oplus E_{\lambda_i}
	\end{eqn}
	where $\lambda_1, \ldots, \lambda_n$ are the eigenvalues of $\varphi$ and $E_{\lambda_i}$ is the $\lambda_i$-eigenspace.
	Thus, in this case, $\sf{Fi}_\varphi (E) / {\sim} = \sf{Fi}_\varphi (E)$ is in bijection with the finite set of size $n!$ of all possible orders $\tau = (\lambda_1, \ldots, \lambda_n)$ of the set of eigenvalues of $\varphi$.
\end{lemma}

More generally, for any eigenvalue $\lambda$ with multiplicity $m$, we consider the generalised $\lambda$-eigenspace
\begin{eqn}
	E_\lambda \coleq \ker \big( (\varphi - \lambda \id)^m \big) \subset E
	\qtext{and let}
	\nu_\lambda \coleq (\varphi - \lambda \id) \big|_{E_\lambda} : E_\lambda \to E_\lambda.
\end{eqn}
Recall that $E_\lambda$ is an $m$-dimensional vector space equipped with a canonical $m$-step filtration by kernels of powers of the nilpotent part $\nu_\lambda$:
\begin{eqn}
	0 \subset \ker ( \nu_\lambda ) \subset \ker ( \nu_\lambda^2 ) \subset \cdots \subset \ker ( \nu_\lambda^m ) = E_\lambda.
\end{eqn}
This is not a flag in general because this sequence typically stabilises to $E_\lambda$ sooner than step $m$.
However, we can always reduce this filtration in a unique way to a partial flag of some length $l \leq m$ by deleting all duplicate subspaces:
\begin{eqntag}
	\label{251018193025}
	E_{\lambda}^\bullet \coleq \Big(
	0 \subset E_{\lambda}^1 \subset \cdots \subset E_{\lambda}^l = E_\lambda
	\Big)
	\qtext{where}
	E_{\lambda}^k = \ker ( \nu_\lambda^{k} ).
\end{eqntag}

Now, suppose $e_0 \in E_\lambda^k \smallsetminus E_\lambda^{k-1}$ is a generalised eigenvector of $\varphi$ of some rank $k \geq 1$.
Recall that this means $\nu_\lambda^k (e_0) = 0$ but $e_j \coleq \nu^j_\lambda (e_0) \neq 0$ for all $j = 0, \ldots, k-1$.
These vectors form a \textit{Jordan chain} $\set{e_{k-1}, \ldots, e_1, e_0} \subset E_\lambda$ of length $k$; i.e., an ordered set of linearly independent vectors, spanning a $k$-dimensional $\varphi$-invariant subspace $E_{\lambda, 0} \coleq \inner{ e_{k-1}, \ldots, e_1, e_0 } \subset E_\lambda$.
This Jordan chain is called \textit{maximal} if $e_0 \not\in \sf{im} (\nu_\lambda)$; i.e., there does not exist $e'_0 \in E_\lambda$ such that $e_0 = \nu_\lambda (e'_0)$.
The subspace $E_{\lambda, 0}$ spanned by a maximal Jordan chain is sometimes called a \textit{Jordan block}.

Any two Jordan blocks are either equal or disjoint, so every generalised eigenspace decomposes into a direct sum of Jordan blocks $E_\lambda = \bigoplus_{b} E_{\lambda,b}$.
But beware that this decomposition is not unique, for if $e_0$ and $e'_0$ are two linearly independent generalised eigenvectors of rank $k$, then $e_0 + e'_0$ and $e_0 - e'_0$ are also generalised eigenvectors of rank $k$.
The vectors $e_0$ and $e'_0$ generate two linearly independent Jordan chains spanning a pair of Jordan blocks $E_{\lambda, b}$ and $E_{\lambda, b'}$.
So do the vectors $e_0 + e'_0$ and $e_0 - e'_0$ thus giving another pair of disjoint Jordan blocks $E'_{\lambda, b}$ and $E'_{\lambda, b'}$.
But $E_{\lambda, b} \oplus E_{\lambda, b'} = E'_{\lambda, b} \oplus E'_{\lambda, b'}$.

The length $l$ in \eqref{251018193025} is the size of the largest Jordan block in $E_\lambda$, whilst the dimension of each $E^k_\lambda$ is the number of Jordan blocks of size $\geq k$.
Thus, the dimension of $E$ is partitioned by the dimensions of its generalised eigenspaces, and the dimension of each generalised eigenspace is further partitioned by the sizes of the Jordan blocks in that generalised eigenspace.
Together, these partitions are known as the \textit{Jordan type} of $\varphi$.
More specifically, let $\lambda_1, \ldots, \lambda_d$ be the distinct eigenvalues of $\varphi$ with respective multiplicities $m_1, \ldots, m_d$.
For each $i = 1, \ldots, d$, let $n_i$ be the total number of Jordan blocks in the generalised $\lambda_i$-eigenspace and let $r_{i1} \geq r_{i2} \geq \cdots \geq r_{in_i}$ be their lengths.
They define a partition $\bm{m}_i \coleq ( m_i = r_{i1} + \ldots + r_{in_i} )$.
Then the \textit{Jordan type} of $\varphi$ is the unordered collection of these partitions labelled by the distinct eigenvalues of $\varphi$:
\begin{eqn}
	\op{Jt} (\varphi) \coleq \set{ (\lambda_1, \bm{m}_{1}), \ldots, (\lambda_d, \bm{m}_{d}) }.
\end{eqn}

Suppose $E_{\lambda,0}$ and $E'_{\lambda,0}$ are two Jordan blocks of the same dimension in the same generalised eigenspace, spanned by Jordan chains $\{ e_{k-1}, \ldots, e_1, e_0 \}$ and $\{ e'_{k-1}, \ldots, e'_1, e'_0 \}$, respectively.
Let $\psi \in \Aut (E)$ be the automorphism that sends $e_i \mapsto e'_i$ and vice versa for all $i = 0, \ldots, k-1$ and leaves all other Jordan blocks untouched.
Then $\psi$ commutes with $\varphi$ because $\psi \varphi (e_i) = \psi (\lambda e_i + e_{i+1}) = \lambda e'_i + e'_{i+1} = \varphi (e'_i) = \varphi \psi (e_i)$ for all $i = 0, \ldots, k-1$ where we have set $e_k = e'_k \coleq 0$ for convenience.

Every Jordan block has its own canonical complete flag given by any Jordan chain that spans it:
\begin{eqn}
	E_{\lambda, 0}^\bullet \coleq \Big( 0 \subset E_{\lambda, 0}^1 \subset \cdots \subset E_{\lambda, 0}^k = E_{\lambda,0} \Big)
	\qtext{where}
	E_{\lambda, 0}^j \coleq \inner{e_{k-1}, \ldots, e_{k-j}}.
\end{eqn}
Meanwhile, the complete flag $F$ also restricts to a complete flag on any Jordan block $E_{\lambda, 0}$ given by intersecting the subspaces of $F$ with $E_{\lambda, 0}$ and deleting all duplicate subspaces.
The point is that the relation $\nu_\lambda e_j = e_{j+1}$ implies $\varphi e_j = \lambda e_j + e_{j+1}$, and this means if $e_j \in F_i$ for some $i$, then we must have $e_{k-1}, \ldots, e_{j+1}, e_j \subset F_i$ because $F_i$ is a $\varphi$-invariant subspace.
In fact, by linear independence, all the vectors in the Jordan chain generated by $e_0$ necessarily land in lower and lower levels of the flag $F$; i.e., if $e_j \in F_i \smallsetminus F_{i-1}$, then we must have $e_{j+1} \in F_{i'} \smallsetminus F_{i'-1}$ for some $i' < i$, and $e_{j+2} \in F_{i''} \smallsetminus F_{i''-1}$ for some $i'' < i'$, and so on.
These observations can be summarised as follows.

\begin{lemma}
	\label{251018210001}
	For any linear system $(E, \varphi)$, the restriction of any $\varphi$-invariant flag $F$ to any Jordan block $E_{\lambda, 0}$ must coincide with the canonical flag $E_{\lambda, 0}^\bullet$.
	Moreover, any automorphism $\psi \in \Aut (E)$ that swaps two Jordan chains of the same length within the same generalised eigenspace must necessarily commute with $\varphi$.
	More precisely, suppose $E_{\lambda,0}$ and $E'_{\lambda,0}$ are any two Jordan blocks of the same dimension spanned by Jordan chains $\{ e_{k-1}, \ldots, e_1, e_0 \}$ and $\{ e'_{k-1}, \ldots, e'_1, e'_0 \}$, respectively.
	Let $\psi \in \Aut (E)$ be the automorphism that sends $e_i \mapsto e'_i$ and vice versa for all $i = 0, \ldots, k-1$ and is the identity on all other Jordan blocks.
	Then $\psi \varphi = \varphi \psi$.
\end{lemma}

\subsection{Shuffled Jordan Types}

To design the complete invariant parameterising the isomorphism classes of $\sf{Fi}_\varphi (E)$, we review the Young tableau representing the Jordan type of an endomorphism.
Jordan blocks in a generalised eigenspace $E_\lambda$ are often organised into a \textit{Young tableau} (using the English convention) such as
\ytableausetup{smalltableaux,aligntableaux=center}
\begin{eqn}
	\begin{ytableau}
		\none & \none & \none \\
		{} & {} & {} \\
		{} & {} \\
		{} & {} \\
		{}
	\end{ytableau}
	\qquad \quad \text{or} \qquad \quad
	\begin{ytableau}
		\none[\scriptsize \lambda] & \none[\scriptsize \lambda] & \none[\scriptsize \lambda] \\
		{11} & {12} & {13} \\
		{21} & {22} \\
		{31} & {33} \\
		{41}
	\end{ytableau}
\end{eqn}
where each row represents a Jordan block whose size equals the number of boxes in that row.
Each column represents a subspace $E^k_\lambda$ and the number of boxes in that column equals the dimension of $E^k_\lambda$.
Every box is labelled by $ij$ where $i$ is the row number and $j$ is the column number.
There is one tableau for every distinct eigenvalue of $\varphi$ so every box is also labelled by an eigenvalue $\lambda$.
Let us denote the Young tableau of $\varphi$ by $\op{Yt} (\varphi)$.
Thus, $\op{Yt} (\varphi)$ is a collection of boxes labelled by $(\lambda, ij)$.

For example, if $\varphi$ has two distinct eigenvalues $\lambda_1, \lambda_2$ with multiplicities $8$ and $3$ respectively, four Jordan blocks of sizes $3,2,2,1$ in $E_{\lambda_1}$, and two Jordan blocks of sizes $2,1$ in $E_{\lambda_2}$, then the Young tableau $\op{Yt} (\varphi) = \op{Yt}_{1} (\varphi) \sqcup \op{Yt}_{2} (\varphi)$ of $\varphi$ has two subtableaux labelled by the eigenvalues $\lambda_1$ and $\lambda_2$ which are as follows.
\begin{center}
	\begin{ytableau}
		\none[\scriptstyle {\lambda_1}] & \none[\scriptstyle {\lambda_1}] & \none[\scriptstyle {\lambda_1}]\\
		{11} & {12} & {13} \\
		{21} & {22} \\
		{31} & {33} \\
		{41}
	\end{ytableau}
	\qquad \quad
	\begin{ytableau}
		\none[\scriptstyle {\lambda_2}] & \none[\scriptstyle {\lambda_2}] \\
		{11} & {12} \\
		{21}	\\
		\none	\\
		\none
	\end{ytableau}
\end{center}
When the total number of rows in the tableau is not very large, it is convenient to omit the row-number from the label $ij$ and use distinct colours for all the rows.
For example, the above Young tableau can be alternatively displayed as follows.
\begin{center}
	\begin{ytableau}
		\none[\scriptstyle {\lambda_1}] & \none[\scriptstyle {\lambda_1}] & \none[\scriptstyle {\lambda_1}]\\
		*(orange) 		{1} 
		&	*(orange) 		{2} 
		&	*(orange) 		{3}
		\\
		*(darkblue!75)	{1} 
		&	*(darkblue!75)	{2} 
		\\
		*(darkgreen)	{1} 
		&	*(darkgreen) 	{2} 
		\\
		*(red!75)		{1}
	\end{ytableau}
	\qquad \quad
	\begin{ytableau}
		\none[\scriptstyle {\lambda_2}] & \none[\scriptstyle {\lambda_2}] \\
		*(yellow)		{1} 
		&	*(yellow)		{2}
		\\
		*(violet!75)		{1}
		\\
		\none	\\
		\none
	\end{ytableau}
\end{center}

Now we are ready to define the complete invariant parameterising the isomorphism classes of $\sf{Fi}_\varphi (E)$.
We being with a less formal but more intuitive definition.
First, to \dfn{shuffle} the rows of a Young tableau is to combine all boxes into a single row in a way that preserves the order of the boxes that came from the same row.
When all eigenvalues are distinct, this is the same thing as ordering the eigenvalues.
More generally, a shuffle can arrange the indistinct eigenvalues out of order as long as the order coming from each row is maintained.
In addition, we consider two shuffles to be \dfn{equivalent} if they only differ by swapping pairs of rows of the same length in the same subtableau, and we call an equivalence class a \dfn{shuffled Jordan type}.

For example, some of the many possible ways to shuffle the six rows of the above Young tableau with eleven boxes are as follows.
\begin{center}
	\begin{ytableau}
		\none[\scriptstyle {\lambda_1}]
		&	\none[\scriptstyle {\lambda_1}]
		&	\none[\scriptstyle {\lambda_1}]
		&	\none[\scriptstyle {\lambda_1}]
		&	\none[\scriptstyle {\lambda_1}]
		&	\none[\scriptstyle {\lambda_1}]
		&	\none[\scriptstyle {\lambda_1}]
		&	\none[\scriptstyle {\lambda_1}]
		&	\none[\scriptstyle {\lambda_2}]
		&	\none[\scriptstyle {\lambda_2}]
		&	\none[\scriptstyle {\lambda_2}]
		\\
		*(orange) 		1
		& 	*(orange) 		2
		& 	*(orange) 		3
		&	*(darkblue!75)	1
		&	*(darkblue!75)	2
		& 	*(darkgreen)	1
		& 	*(darkgreen)	2
		& 	*(red!75)		1
		&	*(yellow)		1
		&	*(yellow)		2
		& 	*(violet!75)	1
	\end{ytableau}
	\qquad \quad
	\begin{ytableau}
		\none[\scriptstyle {\lambda_1}]
		&	\none[\scriptstyle {\lambda_1}]
		&	\none[\scriptstyle {\lambda_1}]
		&	\none[\scriptstyle {\lambda_2}]
		&	\none[\scriptstyle {\lambda_1}]
		&	\none[\scriptstyle {\lambda_1}]
		&	\none[\scriptstyle {\lambda_1}]
		&	\none[\scriptstyle {\lambda_1}]
		&	\none[\scriptstyle {\lambda_1}]
		&	\none[\scriptstyle {\lambda_2}]
		&	\none[\scriptstyle {\lambda_2}]
		\\
		*(darkblue!75)	1
		& 	*(orange) 		1
		& 	*(darkblue!75)	2
		& 	*(violet!75)	1
		& 	*(orange)		2
		& 	*(orange) 		3
		& 	*(darkgreen)	1
		& 	*(red!75)		1
		& 	*(darkgreen)	2
		&	*(yellow)		1
		&	*(yellow)		2
	\end{ytableau}
	\\
	\begin{ytableau}
		\none[\scriptstyle {\lambda_1}]
		&	\none[\scriptstyle {\lambda_1}]
		&	\none[\scriptstyle {\lambda_1}]
		&	\none[\scriptstyle {\lambda_1}]
		&	\none[\scriptstyle {\lambda_1}]
		&	\none[\scriptstyle {\lambda_1}]
		&	\none[\scriptstyle {\lambda_1}]
		&	\none[\scriptstyle {\lambda_1}]
		&	\none[\scriptstyle {\lambda_2}]
		&	\none[\scriptstyle {\lambda_2}]
		&	\none[\scriptstyle {\lambda_2}]
		\\
		*(orange) 		1
		& 	*(orange) 		2
		& 	*(orange) 		3
		& 	*(darkgreen)	1
		& 	*(darkgreen)	2
		&	*(darkblue!75)	1
		&	*(darkblue!75)	2
		& 	*(red!75)		1
		&	*(yellow)		1
		&	*(yellow)		2
		& 	*(violet!75)	1
	\end{ytableau}
	\qquad \quad
	\begin{ytableau}
		\none[\scriptstyle {\lambda_2}]
		&	\none[\scriptstyle {\lambda_1}]
		&	\none[\scriptstyle {\lambda_1}]
		&	\none[\scriptstyle {\lambda_2}]
		&	\none[\scriptstyle {\lambda_1}]
		&	\none[\scriptstyle {\lambda_1}]
		&	\none[\scriptstyle {\lambda_1}]
		&	\none[\scriptstyle {\lambda_2}]
		&	\none[\scriptstyle {\lambda_1}]
		&	\none[\scriptstyle {\lambda_1}]
		&	\none[\scriptstyle {\lambda_2}]
		\\
		*(yellow)		1
		&	*(red!75)		1
		& 	*(orange) 		1
		& 	*(darkgreen)	1
		& 	*(orange)		2
		& 	*(orange) 		3
		& 	*(darkgreen)	2
		& 	*(violet!75)	1
		&	*(darkblue!75)	1
		& 	*(darkblue!75)	2
		&	*(yellow)		2
	\end{ytableau}
	\\
	\begin{ytableau}
		\none[\scriptstyle {\lambda_1}]
		&	\none[\scriptstyle {\lambda_1}]
		&	\none[\scriptstyle {\lambda_1}]
		&	\none[\scriptstyle {\lambda_2}]
		&	\none[\scriptstyle {\lambda_2}]
		&	\none[\scriptstyle {\lambda_1}]
		&	\none[\scriptstyle {\lambda_1}]
		&	\none[\scriptstyle {\lambda_1}]
		&	\none[\scriptstyle {\lambda_1}]
		&	\none[\scriptstyle {\lambda_1}]
		&	\none[\scriptstyle {\lambda_2}]
		\\
		*(orange) 		1
		& 	*(orange) 		2
		& 	*(orange) 		3
		& 	*(yellow)		1
		& 	*(yellow)		2
		&	*(darkgreen)	1
		&	*(darkgreen)	2
		& 	*(red!75)		1
		&	*(darkblue!75)	1
		&	*(darkblue!75)	2
		& 	*(violet!75)	1
	\end{ytableau}
	\qquad \quad
	\begin{ytableau}
		\none[\scriptstyle {\lambda_1}]
		&	\none[\scriptstyle {\lambda_1}]
		&	\none[\scriptstyle {\lambda_1}]
		&	\none[\scriptstyle {\lambda_2}]
		&	\none[\scriptstyle {\lambda_1}]
		&	\none[\scriptstyle {\lambda_1}]
		&	\none[\scriptstyle {\lambda_1}]
		&	\none[\scriptstyle {\lambda_1}]
		&	\none[\scriptstyle {\lambda_1}]
		&	\none[\scriptstyle {\lambda_2}]
		&	\none[\scriptstyle {\lambda_2}]
		\\
		*(darkgreen)	1
		& 	*(orange) 		1
		& 	*(darkgreen)	2
		& 	*(violet!75)	1
		& 	*(orange)		2
		& 	*(orange) 		3
		& 	*(darkblue!75)	1
		& 	*(red!75)		1
		& 	*(darkblue!75)	2
		&	*(yellow)		1
		&	*(yellow)		2
	\end{ytableau}
\end{center}
Notice that the eigenvalue labels $\lambda_1, \lambda_2$ and the colours may occur in any order, but the order of the boxes within each colour is respected.
In contrast, any combination of boxes that contains, say, \begin{ytableau} \none[\scriptstyle \cdots] & *(orange) 2 & \none[\scriptstyle \cdots] & *(orange) 1 & \none[\scriptstyle \cdots] \end{ytableau} is disallowed.
The top-left and middle-left shuffles are equivalent because they can be obtained from one another by swapping the second and third rows in the $\lambda_1$ subtableau.
Similarly, the top-right and bottom-right shuffles are equivalent.
In contrast, the top-left and bottom-left shuffles are not equivalent because we disallow swapping rows from the $\lambda_1$-subtableau with rows from the $\lambda_2$-subtableau.
Obviously, there are too many possible shuffled Jordan types in this example to list them all (in fact, $415\,800$; see \autoref{251017162851} below).
For a much smaller example, consider $\varphi$ in $\CC^4$ with just one eigenvalue and two Jordan chains of lengths $1$ and $3$.
A quick count reveals that there are only four shuffled Jordan types which can be listed in full.
\begin{center}
	\begin{ytableau}
		\none[\scriptstyle {\lambda}] & \none[\scriptstyle {\lambda}] & \none[\scriptstyle {\lambda}] \\
		*(orange)			{1} 
		&	*(orange)			{2}
		&	*(orange)			{3}
		\\
		*(darkblue!75)		{1}
	\end{ytableau}
	\qquad \qquad
	\begin{ytableau}
		\none[\scriptstyle {\lambda}]
		&	\none[\scriptstyle {\lambda}]
		&	\none[\scriptstyle {\lambda}]
		&	\none[\scriptstyle {\lambda}]
		\\
		*(orange) 		1
		& 	*(orange) 		2
		& 	*(orange) 		3
		&	*(darkblue!75)	1
		\\
		\none
	\end{ytableau}
	\quad
	\begin{ytableau}
		\none[\scriptstyle {\lambda}]
		&	\none[\scriptstyle {\lambda}]
		&	\none[\scriptstyle {\lambda}]
		&	\none[\scriptstyle {\lambda}]
		\\
		*(orange) 		1
		& 	*(orange) 		2
		&	*(darkblue!75)	1
		& 	*(orange) 		3
		\\
		\none
	\end{ytableau}
	\quad
	\begin{ytableau}
		\none[\scriptstyle {\lambda}]
		&	\none[\scriptstyle {\lambda}]
		&	\none[\scriptstyle {\lambda}]
		&	\none[\scriptstyle {\lambda}]
		\\
		*(orange) 		1
		&	*(darkblue!75)	1
		& 	*(orange) 		2
		& 	*(orange) 		3
		\\
		\none
	\end{ytableau}
	\quad
	\begin{ytableau}
		\none[\scriptstyle {\lambda}]
		&	\none[\scriptstyle {\lambda}]
		&	\none[\scriptstyle {\lambda}]
		&	\none[\scriptstyle {\lambda}]
		\\
		*(darkblue!75)	1
		& 	*(orange) 		1
		& 	*(orange) 		2
		& 	*(orange) 		3
		\\
		\none
	\end{ytableau}
\end{center}
Even more concretely, in \autoref{251019183631} we encountered an automorphism $\varphi$ in $\CC^3$ with just one Jordan block of size $3$.
In this case, the Young tableau consists of just one row \begin{ytableau} *(orange) 1 & *(orange) 2 & *(orange) 3 \end{ytableau} and no other possible shuffles, and this is reflected in the fact that $\CC^3$ has a unique $\varphi$-invariant flag.
On the other hand, in \autoref{251019183645} we encountered an automorphism $\varphi$ in $\CC^3$ with two Jordan blocks of size $1$ and $2$, and we also saw that $\sf{Fi}_\varphi (\CC^3) / {\sim}$ consists of three elements corresponding to the three inequivalent groups of $\varphi$-invariant flags $F^c, F', F''$.
In this case, the Young tableau has two rows of lengths $2$ and $1$, and so there are only three possible shuffled Jordan types as follows.
\begin{center}
	\begin{ytableau}
		\none[\scriptstyle {1}] & \none[\scriptstyle {1}] \\
		*(orange)			{1} 
		&	*(orange)			{2}
		\\
		*(darkblue!75)		{1}
	\end{ytableau}
	\qquad \qquad
	\begin{ytableau}
		\none[\scriptstyle {1}]
		&	\none[\scriptstyle {1}]
		&	\none[\scriptstyle {1}]
		\\
		*(darkblue!75)	1
		& 	*(orange) 		1
		& 	*(orange) 		2
		\\
		\none
	\end{ytableau}
	\quad
	\begin{ytableau}
		\none[\scriptstyle {1}]
		&	\none[\scriptstyle {1}]
		&	\none[\scriptstyle {1}]
		\\
		*(orange) 		1
		&	*(darkblue!75)	1
		& 	*(orange) 		2
		\\
		\none
	\end{ytableau}
	\quad
	\begin{ytableau}
		\none[\scriptstyle {1}]
		&	\none[\scriptstyle {1}]
		&	\none[\scriptstyle {1}]
		\\
		*(orange) 		1
		& 	*(orange) 		2
		&	*(darkblue!75)	1
		\\
		\none
	\end{ytableau}
\end{center}

Notice that, up to relabelling using a different set of eigenvalues, the set of all shuffled Jordan types depends only on the Jordan type of $\varphi$.
That is to say, if $\varphi$ and $\varphi'$ have different eigenvalues but the same Jordan type, there is a one-to-one correspondence between the shuffles for $\varphi$ and the shuffles for $\varphi'$.

More formally, given $(E,\varphi)$, for every distinct eigenvalue $\lambda$, let $R_{\lambda,i}$ denote the ordered collection of boxes in row $i$ of the $\lambda$-subtableau $\op{Yt}_\lambda (\varphi) \subset \op{Yt} (\varphi)$, let $r_{\lambda,i}$ be the number of boxes in $R_{\lambda,i}$, and let $R$ be the total union of all the boxes:
\begin{eqn}
	R_{\lambda,i}
	\coleq \Big( \ytableausetup{smalltableaux,aligntableaux=bottom} 
	\raisebox{-3pt}{\begin{ytableau} \none[\scriptstyle \lambda] \\ {ij} \end{ytableau}} \Big)_{j = 1}^{r}
	= \Big( \raisebox{-3pt}{\begin{ytableau} \none[\scriptstyle \lambda] \\ {i1} \end{ytableau}} \,, 
	\raisebox{-3pt}{\begin{ytableau} \none[\scriptstyle \lambda] \\ {i2} \end{ytableau}} \,,
	\ldots,
	\raisebox{-3pt}{\begin{ytableau} \none[\scriptstyle \lambda] \\ {ir} \end{ytableau}}
	\Big)		
	\qtext{where}
	r = r_{\lambda,i}
	\qqtext{and}
	R \coleq \bigcup_{\lambda,i} R_{\lambda,i}.
\end{eqn}

\begin{definition}
	\label{251017143241}
	A \dfn{shuffle} of the Young tableau $\op{Yt} (\varphi)$ is a choice of total order on the set $R$ whose restriction to each $R_{\lambda, i} \subset R$ coincides with the natural order.
	We call two shuffles \dfn{equivalent} if one can be obtained from the other by a permutation of two sets $R_{\lambda, i}$ and $R_{\lambda, j}$ of the same length, or by a sequence of such permutations.
	A \dfn{shuffled Jordan type} on $(E, \varphi)$ is an equivalence class $\tau$ of shuffles.
	Note that if $\varphi$ has all-distinct eigenvalues, a shuffled Jordan type is the same thing as a total order on the eigenvalues.
\end{definition}

We denote the set of shuffles and shuffled Jordan types on $(E,\varphi)$ by
\begin{eqns}
	\sfop{Sh} \big( \op{Yt} (\varphi) \big) &\coleq \set{ \text{shuffles on $(E,\varphi)$} \Big.},
	\\
	\sfop{ShJt} (\varphi) &\coleq \set{ \text{shuffled Jordan types on $(E, \varphi)$} \Big.}
	= \sfop{Sh} \big( \op{Yt} (\varphi) \big) \big/ {\sim}.
\end{eqns}

To build a shuffled Jordan type for $\varphi$, start by drawing $n = \dim E$ empty boxes in a row.
\begin{center}
	\begin{ytableau}
		\none[\scriptstyle {\lambda_1}] & \none[\scriptstyle {\lambda_1}] & \none[\scriptstyle {\lambda_1}]\\
		*(orange) 		{1} 
		&	*(orange) 		{2} 
		&	*(orange) 		{3}
		\\
		*(darkblue!75)	{1} 
		&	*(darkblue!75)	{2} 
		\\
		*(darkgreen)	{1} 
		&	*(darkgreen) 	{2} 
		\\
		*(red!75)		{1}
	\end{ytableau}
	\qquad \quad
	\begin{ytableau}
		\none[\scriptstyle {\lambda_2}] & \none[\scriptstyle {\lambda_2}] \\
		*(yellow)		{1} 
		&	*(yellow)		{2}
		\\
		*(violet!75)		{1}
		\\
		\none	\\
		\none
	\end{ytableau}
	\hspace{2cm}
	\begin{ytableau}
		\none[]
		&	\none[]
		&	\none[]
		&	\none[]
		&	\none[]
		&	\none[]
		&	\none[]
		&	\none[]
		&	\none[]
		&	\none[]
		&	\none[]
		\\
		& 	
		& 	
		& 	
		& 	
		& 	
		& 	
		& 	
		& 	
		&	
		&	
		\\
		\none[]
		\\
		\none[]
		\\
		\none[]
	\end{ytableau}
\end{center}
Pick a row from the Young tableau of $(E,\varphi)$, let $r_1$ be the number of boxes in it, and fill in any $r_1$ of the $n$ empty boxes.
\begin{center}
	\begin{ytableau}
		\none[\scriptstyle {\lambda_1}] & \none[\scriptstyle {\lambda_1}] & \none[\scriptstyle {\lambda_1}]\\
		*(orange) 		{1} 
		&	*(orange) 		{2} 
		&	*(orange) 		{3}
		\\
		*(darkblue!75)	{1} 
		&	*(darkblue!75)	{2} 
		\\
		*(darkgreen)	{1} 
		&	*(darkgreen) 	{2} 
		\\
		*(red!75)		{1}
	\end{ytableau}
	\qquad \quad
	\begin{ytableau}
		\none[\scriptstyle {\lambda_2}] & \none[\scriptstyle {\lambda_2}] \\
		*(yellow)		{1} 
		&	*(yellow)		{2}
		\\
		*(violet!75)		{1}
		\\
		\none	\\
		\none
	\end{ytableau}
	\hspace{2cm}
	\begin{ytableau}
		\none[]
		&	\none[\scriptstyle {\lambda_1}]
		&	\none[]
		&	\none[]
		&	\none[\scriptstyle {\lambda_1}]
		&	\none[]
		&	\none[\scriptstyle {\lambda_1}]
		&	\none[]
		&	\none[]
		&	\none[]
		&	\none[]
		\\
		& 	*(orange) 		{1}
		& 	
		& 	
		& 	*(orange) 		{2}
		& 	
		& 	*(orange) 		{3}
		& 	
		& 	
		&	
		&	
		\\
		\none[]
		\\
		\none[]
		\\
		\none[]
	\end{ytableau}
\end{center}
Then pick a different row from the Young tableau, let $r_2$ be the number of boxes in it, and fill in any $r_2$ of the remaining $n - r_1$ empty boxes.
\begin{center}
	\begin{ytableau}
		\none[\scriptstyle {\lambda_1}] & \none[\scriptstyle {\lambda_1}] & \none[\scriptstyle {\lambda_1}]\\
		*(orange) 		{1} 
		&	*(orange) 		{2} 
		&	*(orange) 		{3}
		\\
		*(darkblue!75)	{1} 
		&	*(darkblue!75)	{2} 
		\\
		*(darkgreen)	{1} 
		&	*(darkgreen) 	{2} 
		\\
		*(red!75)		{1}
	\end{ytableau}
	\qquad \quad
	\begin{ytableau}
		\none[\scriptstyle {\lambda_2}] & \none[\scriptstyle {\lambda_2}] \\
		*(yellow)		{1} 
		&	*(yellow)		{2}
		\\
		*(violet!75)		{1}
		\\
		\none	\\
		\none
	\end{ytableau}
	\hspace{2cm}
	\begin{ytableau}
		\none[]
		&	\none[\scriptstyle {\lambda_1}]
		&	\none[]
		&	\none[\scriptstyle {\lambda_2}]
		&	\none[\scriptstyle {\lambda_1}]
		&	\none[]
		&	\none[\scriptstyle {\lambda_1}]
		&	\none[\scriptstyle {\lambda_2}]
		&	\none[]
		&	\none[]
		&	\none[]
		\\
		& 	*(orange) 		{1}
		& 	
		& 	*(yellow)		{1}
		& 	*(orange) 		{2}
		& 	
		& 	*(orange) 		{3}
		& 	*(yellow)		{2}
		& 	
		&	
		&	
		\\
		\none[]
		\\
		\none[]
		\\
		\none[]
	\end{ytableau}
\end{center}
And so on until we run out of rows in the Young tableau.
\begin{center}
	\begin{ytableau}
		\none[\scriptstyle {\lambda_1}] & \none[\scriptstyle {\lambda_1}] & \none[\scriptstyle {\lambda_1}]\\
		*(orange) 		{1} 
		&	*(orange) 		{2} 
		&	*(orange) 		{3}
		\\
		*(darkblue!75)	{1} 
		&	*(darkblue!75)	{2} 
		\\
		*(darkgreen)	{1} 
		&	*(darkgreen) 	{2} 
		\\
		*(red!75)		{1}
	\end{ytableau}
	\qquad \quad
	\begin{ytableau}
		\none[\scriptstyle {\lambda_2}] & \none[\scriptstyle {\lambda_2}] \\
		*(yellow)		{1} 
		&	*(yellow)		{2}
		\\
		*(violet!75)	{1}
		\\
		\none	\\
		\none
	\end{ytableau}
	\hspace{2cm}
	\begin{ytableau}
		\none[\scriptstyle {\lambda_1}]
		&	\none[\scriptstyle {\lambda_1}]
		&	\none[\scriptstyle {\lambda_1}]
		&	\none[\scriptstyle {\lambda_2}]
		&	\none[\scriptstyle {\lambda_1}]
		&	\none[\scriptstyle {\lambda_1}]
		&	\none[\scriptstyle {\lambda_1}]
		&	\none[\scriptstyle {\lambda_2}]
		&	\none[\scriptstyle {\lambda_1}]
		&	\none[\scriptstyle {\lambda_2}]
		&	\none[\scriptstyle {\lambda_1}]
		\\
		*(darkblue!75)	{1}
		& 	*(orange) 		{1}
		& 	*(darkgreen)	{1} 
		& 	*(yellow)		{1}
		& 	*(orange) 		{2}
		& 	*(darkblue!75)	{2}
		& 	*(orange) 		{3} 
		& 	*(yellow)		{2}
		& 	*(red!75)		{1}
		&	*(violet!75)	{1}
		&	*(darkgreen)	{2} 
		\\
		\none[]
		\\
		\none[]
		\\
		\none[]
	\end{ytableau}
\end{center}
We see that the total number of shuffles is nothing but a multinomial coefficient
\begin{eqn}
	\big| \sfop{Sh} \big( \op{Yt} (\varphi) \big) \big| =
	{n \choose r_1, r_2, \ldots, r_m}
	=
	\frac{n!}{r_1! r_2! \cdots r_m!},
\end{eqn}
where $m$ is the total number of rows in the Young tableau.
To take equivalence of shuffles into account, we must divide the above multinomial coefficient by the total number of allowed permutations of rows within each subtableau.
We summarise this as follows.

\begin{lemma}
	\label{251017162851}
	Let $(E, \varphi)$ be an $n$-dimensional linear system, and suppose that the Young tableau of $\varphi$ has $m$ rows with lengths $r_1, \ldots, r_m$.
	For every distinct eigenvalue $\lambda$ and every positive integer $r$, let $c_{\lambda, r}$ be the number of rows of length $r$ in the $\lambda$-subtableau.
	Let $\C \coleq \prod c_{\lambda, r}!$ where the product is taken over all distinct eigenvalues $\lambda$ and all integers $r \geq 1$.
	Then the total number of possible shuffled Jordan types on $(E,\varphi)$ is
	\begin{eqn}
		\big| \sfop{Sh} \big( \op{Yt} (\varphi) \big) /{\sim} \big| =
		\frac{1}{\C} {n \choose r_1, r_2, \ldots, r_m}.
	\end{eqn}
\end{lemma}

In the above example with six rows and eleven boxes, the total number of shuffles is $11!/(3! 2! 2! 1! 2! 1!) = 831\,600$.
There is just one pair of rows of length $2$ in the first subtableau, so $c_{\lambda_1,2} = 2$ and all other numbers $c_{\lambda, r}$ are either $1$ or $0$.
So $\C = 2!$ and therefore the total number of shuffled Jordan types is $415\,800$.

\subsection{Shuffled Jordan types from Jordan Basis}

The main use of a Young tableau in linear algebra is of course to construct a Jordan basis; i.e., a basis of generalised eigenvectors with respect to which $\varphi$ becomes a Jordan normal form.
Generalised eigenvectors in a Jordan basis are labelled by boxes of the Young tableau, and so any shuffle determines a total order on a Jordan basis.
Conversely, any permutation of a Jordan basis determines an shuffle so long as it does not scramble amongst themselves the vectors within each Jordan chain.
This leads to the following observation.

\begin{lemma}
	\label{251017230210}
	Let $(E,\varphi)$ be a linear system.
	Given a Jordan basis $\beta$ for $\varphi$, the set of permutations of $\beta$ that preserve the order within each Jordan chain are in one-to-one correspondence with the set of shuffles of the Young tableau of $\varphi$.
\end{lemma}

More concretely, given a linear system $(E, \varphi)$, choose any Jordan basis $\beta$ for $\varphi$.
Write it as a union of Jordan bases $\beta_i$ for each generalised eigenspace $E_i$, and furthermore write each $\beta_i$ as a union of maximal Jordan chains $\beta_{ij}$ in $E_i$:
\begin{eqn}
	\beta = \beta_1 \cup \cdots \cup \beta_d
	\qqtext{where}
	\beta_i = \beta_{i1} \cup \cdots \cup \beta_{in_i}.
\end{eqn}
Here, $d$ is the number of distinct eigenvalues of $\varphi$ and $n_i$ is the total number of Jordan blocks in the generalised $\lambda_i$-eigenspace.
Thus, $\beta$ determines a decomposition of $E$ into Jordan blocks:
\begin{eqn}
	E = \bigoplus_{i=1}^d E_i = \bigoplus_{i=1}^d \bigoplus_{j=1}^{n_i} E_{ij}
	\qtext{where}
	E_{ij} \coleq \inner{ \beta_{ij} }.
\end{eqn}

Let $\Sym (\beta) \cong W \subset G$ be the symmetric group on the Jordan basis $\beta$, canonically isomorphic to the subgroup of permutation matrices in $G$; i.e., $\Sym (\beta)$ is the set of all permutations of the basis vectors in $\beta$.
Any permutation $\tau \in \Sym (\beta)$ gives a new basis $\tau (\beta)$ of $E$ but which may not be a Jordan basis.
Similarly, $\Sym (\beta_{ij})$ is the symmetric group on each maximal Jordan chain $\beta_{ij}$: any permutation of $\beta_{ij}$ is still a basis of the Jordan block $E_{ij}$ but it is no longer a Jordan chain if the permutation is nontrivial.
Any permutation $\tau \in \Sym (\beta)$ induces a permutation $\tau_{ij} \in \Sym (\beta_{ij})$ on each Jordan chain $\beta_{ij}$ which records how $\tau$ reorders the elements of the set $\beta_{ij}$ amongst themselves.
Let $\sfop{Sh} (\beta) \subset \Sym (\beta)$ denote the subset of permutations of $\beta$ that preserve the order within each Jordan chain:
\begin{eqn}
	\sfop{Sh} (\beta) \coleq \set{ \tau \in \Sym (\beta) 
		~\Big|~ \tau_{ij} = 1 \in \Sym (\beta_{ij}) \quad \text{for all $i,j$} }.
\end{eqn}
Amongst all such permutations, let $\sf{sh}_0 (\beta) \subset \sfop{Sh} (\beta)$ be the subset of permutations which only swap a pair of Jordan chains in the same generalised eigenspace and leave the rest of the Jordan basis unchanged.
Let $\sfop{Sh}_0 (\beta) \subset \sfop{Sh} (\beta)$ be the subset consisting of all finite compositions of elements of $\sf{sh}_0 (\beta)$.
In symbols,
\begin{eqns}
	\sf{sh}_0 (\beta) &\coleq \set{ \tau \in \sfop{Sh} (\beta) ~\Big|~ \text{$\tau (\beta_{ij}) = \beta_{ij'}$ and $\tau = 1$ on $\beta \smallsetminus (\beta_{ij} {\cup} \beta_{ij'})$ for some $i,j,j'$}},
	\\
	\sfop{Sh}_0 (\beta) &\coleq \set{ \tau \in \sfop{Sh} (\beta) ~\Big|~ \tau = \tau_k \circ \cdots \circ \tau_1 \text{ for some } \tau_1, \ldots, \tau_k \in \sf{sh}_0 (\beta) }.
\end{eqns}
Note that if $\tau \in \sf{sh}_0 (\beta)$ and $\tau (\beta_{ij}) = \beta_{ij'}$, then $\beta_{ij}$ and $\beta_{ij'}$ necessarily have the same length and $\tau$ is an order-preserving bijection.
Finally, we use $\sfop{Sh}_0 (\beta)$ to introduce an equivalence relation on the set $\sfop{Sh} (\beta)$ by declaring two permutations of $\beta$ to be equivalent if they differ by an element of $\sfop{Sh}_0 (\beta)$; i.e., for all $\tau,\tau' \in \sfop{Sh}(\beta)$,
\begin{eqn}
	\tau \sim \tau'
	\quad\coliff\quad
	\tau' \circ \tau^{-1} \in \sfop{Sh}_0 (\beta).
\end{eqn}
Then we obtain the following sharper version of \autoref{251017230210}.

\begin{lemma}
	\label{251117154608}
	Let $(E,\varphi)$ be a linear system, and choose any Jordan basis $\beta$ for $\varphi$.
	Then the sets $\sfop{Sh} (\varphi)$ and $C (\varphi)$ of shuffles and shuffled Jordan types on $(E,\varphi)$ are in one-to-one correspondence with the set $\sfop{Sh} (\beta)$ of all permutations of $\beta$ that preserve the order within each Jordan chain and the set $\sfop{Sh} (\varphi) / \sfop{Sh}_0 (\varphi)$, respectively:
	\begin{eqn}
		\sfop{Sh} (\varphi) \cong \sfop{Sh} (\beta)
		\qtext{and}
		C (\varphi) \cong \sfop{Sh} (\beta) / \sfop{Sh}_0 (\beta).
	\end{eqn}
\end{lemma}

\subsection{Invariant Flags as shuffled Jordan types}

The main result of this subsection is the following proposition.

\begin{proposition}
	\label{251017181852}
	Let $(E,\varphi)$ be a linear system.
	Any shuffled Jordan type determines a complete $\varphi$-invariant flag uniquely up to a $\varphi$-equivariant automorphism.
	Conversely, any complete $\varphi$-invariant flag determines a shuffled Jordan type.
	In particular, if two complete $\varphi$-invariant flags $F,F'$ are related by a $\varphi$-equivariant automorphism $\psi \in \Aut (E)$ (i.e., such that $\psi (F_i) = F'_i$ and $\psi \varphi = \varphi \psi$), then $F,F'$ determine the same shuffled Jordan type.
	In other words, the set $\sf{Fi}_\varphi (E) / {\sim}$ is finite and in bijection with the set of shuffled Jordan types on $(E,\varphi)$:
	\begin{eqn}
		\sf{Fi}_\varphi (E) / {\sim} \cong \sfop{Jt}_W (\varphi).
	\end{eqn}
\end{proposition}

\begin{proof}
	Given a shuffled Jordan type, choose any Jordan basis and order it according to any shuffle that represents the given shuffled Jordan type.
	The resulting basis $(e_1, \ldots, e_n)$ defines a $\varphi$-invariant flag $F$ by $F_i \coleq \inner{e_1, \ldots, e_i}$.
	Any other choice of Jordan basis and/or shuffle results in a potentially different basis $(e'_1, \ldots, e'_n)$ and hence a different complete $\varphi$-invariant flag $F'$ by $F'_i \coleq \inner{e'_1, \ldots, e'_i}$.
	However, the automorphism $\psi \in \Aut (E)$ obtained as the change-of-basis matrix from $(e_i)$ to $(e'_i)$ moves the flag $F$ to $F'$, and $\psi$ is the product of several changes of basis corresponding to the individual swaps of pairs of rows of the same length within the same subtableau.
	By \autoref{251018210001}, any such change-of-basis matrix commutes with $\varphi$, so $\psi \varphi = \varphi \psi$.
	
	Conversely, suppose $F$ is a complete $\varphi$-invariant flag.
	We will construct a basis of generalised eigenvectors adapted to $F$ which is a permutation of a Jordan basis that respects the order within each Jordan chain.
	Since $E$ decomposes into a direct sum of generalised eigenspaces, it is enough to focus on each individual generalised eigenspace $E_\lambda$ and construct an ordered basis of generalised eigenvectors adapted to the restriction of $F$.
	So, without loss of generality, let us assume that $\varphi$ has only one eigenvalue $\lambda$ so that the whole $n$-dimensional space $E$ is the generalised $\lambda$-eigenspace.
	
	Observe that whilst $\varphi$ satisfies $\varphi (F_i) = F_i$ for each $i$, its nilpotent part $\nu \coleq \varphi - \lambda \id$ satisfies $\nu (F_i) \subsetneq F_i$ for each $i$.
	Rephrased in more elementary terms, the matrix of $\varphi$ in any basis adapted to $F$ is upper-triangular but the matrix of $\nu$ is strictly-upper triangular.
	Start by picking any vector $e_n \in F_n \smallsetminus F_{n-1}$.
	It generates a Jordan chain $(\nu^{k-1} e_n, \nu^{k-2} e_n, \ldots, \nu e_n, e_n)$ of some length $k$.
	These are linearly independent vectors which must therefore belong to different levels of the flag $F$ in the following sense.
	For each $i = 0, \ldots, k-1$, there is some $n_i$ such that $e_{n_i} \coleq \nu^{i} e_n \in F_{n_i}$ yet $e_{n_i} \notin F_{n_i - 1}$.
	Note that $0 < n_{k-1} < n_{k-2} < \ldots < n_1 < n_0 = n$.
	As a result, we have constructed generalised eigenvectors $e_{i} \in F_{i} \smallsetminus F_{i-1}$ for all $i \in \set{ n_{k-1}, n_{k-2}, \ldots, n_1, n_0} \subset \set{ 1, \ldots, n}$.
	
	Next, let $m < n$ be the largest integer not equal to any $n_i$.
	We repeat the above argument by picking a vector $e_m \in F_m \smallsetminus F_{m-1}$ and using it to generate another Jordan chain $(\nu^{l-1} e_m, \ldots, \nu e_m, e_m)$ of some length $l$.
	Define $e_{m_i} \coleq \nu^i e_m$ where $m_i$ is such that $e_{m_i} \in F_{m_i} \smallsetminus F_{m_i - 1}$ for all $i = 0, \ldots, l-1$.
	Again, we have $m_{l-1} < \ldots < m_1 < m_0 = m$
	Since Jordan chains are linearly independent, it follows that all vectors $e_{m_i}, e_{n_j}$ are linearly independent, and so none of the $m_i$ and $n_j$ are equal.
	We have therefore constructed vectors $e_{i} \in F_{i} \smallsetminus F_{i-1}$ for all $i \in \set{ m_{l-1}, \ldots, m_0, n_{k-1}, \ldots, n_0} \subset \set{ 1, \ldots, n}$.
	
	We continue this process until we exhaust the set $\set{ 1, \ldots, n}$.
	This process stops after a finite number of steps, resulting in the desired basis $(e_1, \ldots, e_n)$ of generalised eigenvectors adapted to $F$.
	Furthermore, every vector $e_i$ is a member of a Jordan chain of vectors belonging to the basis $(e_1, \ldots, e_n)$, and each Jordan chain respects the order of this basis.
	By \autoref{251017230210}, it determines a shuffle and hence a shuffled Jordan type.
	
	Thus, we have a bijection between the set $\sf{Fi}_\varphi (E) / {\sim}$ and the set of shuffled Jordan types $\sfop{Jt}_W (\varphi)$ of $\varphi$ which is finite by \autoref{251017162851}.
\end{proof}

\subsection{Shuffled Jordan Matrices}
\label{251119161726}

A \textit{Jordan matrix} is any block-diagonal matrix whose diagonal blocks are of the form $\lambda \idd + \N$ where $\lambda \in \CC^\ast$ and $\N$ is the nilpotent matrix that has $1$'s on the superdiagonal and $0$'s elsewhere.
Let $J \subset B$ denote the subset of all invertible Jordan matrices.
If $\varphi \in \Aut (E)$ is an automorphism of an $n$-dimensional vector space, let $J (\varphi) \subset J$ denote the subset of all Jordan matrices that have the same Jordan type as $\varphi$:
\begin{eqn}
	J (\varphi) \coleq \set{ \J \in J ~\big|~ \op{Jt} (\J) = \op{Jt} (\varphi) }.
\end{eqn}
Recall that $J (\varphi)$ is a finite set whose cardinality equals the number of all possible permutations of the distinct eigenvalues of $\varphi$ and of the Jordan blocks for the same eigenvalue.
In particular, if $\varphi$ has all-distinct eigenvalues, then $| J (\varphi) | = n!$.

A \dfn{shuffled Jordan matrix} is any upper-triangular matrix of the form $\Z \coleq \tau \J \tau^{-1}$ where $\J \in J$ and $\tau \in W$.
Given a Jordan matrix $\J \in J$, in general the conjugate $\Z$ remains upper-triangular only for a subset $W (\J) \subset W$ of permutations $\tau$, which we call \dfn{(Jordan) shuffles} of $\J$.
In particular, if $\J$ is diagonal, then $W (\J) = W$; otherwise, $W (\J) \subset W$ is a proper subset.
We denote the subset of all shuffled Jordan matrices by $\sfop{Jt}_W \subset B$, and the subset of all shuffled Jordan matrices with the same Jordan type as $\varphi$ by $\sfop{Jt}_W (\varphi) \subset \sfop{Jt}_W$.
In symbols,
\begin{eqns}
	W (\J) 
	&\coleq \set{ \tau \in W ~\Big|~ \Z = \tau \J \tau^{-1} \in B },
	\\
	\sfop{Jt}_W
	&\coleq \set{ \Z = \tau \J \tau^{-1} \in B ~\Big|~ \J \in J,~ \tau \in W (\J)},
	\\
	\sfop{Jt}_W (\varphi)
	&\coleq \set{ \Z = \tau \J \tau^{-1} \in B ~\Big|~ \J \in J (\varphi),~ \tau \in W (\J)}.
\end{eqns}
Equivalently, $\sfop{Jt}_W$ is the set of all pairs $(\J, \tau) \in J \times W$ such that $\J_\tau$ is upper-triangular, considered up to equivalence $(\J, \tau) \sim (\J', \tau') ~\coliff~ \tau \J \tau^{-1} = \tau' \J' \tau'^{-1}$:
\begin{eqn}
	\sfop{Jt}_W
	\cong \set{ (\J, \tau) \in J {\times} W ~\Big|~ \tau \J \tau^{-1} \in B } \big/ {\sim}.
\end{eqn}
Similarly, we define a \dfn{generalised shuffled Jordan matrix} (resp. \dfn{special-generalised}) to be an upper-triangular matrix of the form $\J_\omega = \omega \J \omega^{-1}$ for a Jordan matrix $\J \in J$ and a generalised permutation matrix $\omega \in N$ (resp. special-generalised permutation matrix $\omega \in N_0$).
We denote their subsets similarly:
\begin{eqns}
	\sfop{Jt}_N
	&\coleq \set{ \Z = \omega \J \omega^{-1} \in B ~\Big|~ \J \in J,~ \omega \in N (\J)}
	\cong \set{ (\J, \omega) \in J {\times} N ~\Big|~ \omega \J \omega^{-1} \in B } \big/ {\sim},
	\\
	\sfop{Jt}_{N_0}
	&\coleq \set{ \Z = \omega \J \omega^{-1} \in B ~\Big|~ \J \in J,~ \omega \in N_0 (\J)}
	\cong \set{ (\J, \omega) \in J {\times} N_0 ~\Big|~ \omega \J \omega^{-1} \in B } \big/ {\sim}.
\end{eqns}

The natural surjective map $\op{Jt} : \sfop{Jt}_W \to \sfop{Jt}$ that extracts the Jordan type of a shuffled Jordan matrix has finite fibres.
In particular, observe that $\sfop{Jt}_W (\varphi)$ is a finite set for any $\varphi$; in fact, fixing a Jordan matrix $\J \in J (\varphi)$ determines a bijection $\sfop{Jt}_W (\varphi) \cong W (\J)$.
The cardinality of the fibre is largest over any generic Jordan type (i.e., with all-distinct eigenvalues) where it equals to $n!$ (the number of all possible permutations of $n$ eigenvalues) and this fibre is (non-canonically) in bijection with $W$.
In contrast, the cardinality of the fibre is smallest over any Jordan type consisting of a single Jordan block of length $n$ where it equals $1$.

There is also a well-defined surjective map $\diag : \sfop{Jt}_W \to T$ that extracts the diagonal part $\diag (\J_\tau) \in T$ of a shuffled Jordan matrix.
This map has finite fibres too, and in fact it is a bijection onto the dense open subset of $T$ consisting of diagonal matrices with all-distinct diagonal entries.
On the other hand, restricting $\diag$ over the lower- and lower-dimensional subvarieties of $T$ with coincident diagonal entries, we find more and more connected components of $\sfop{Jt}_W$, enumerated by the possible Jordan types, and the restriction of $\diag$ to each of these components is a bijection onto the given subvariety of $T$.
Consequently, $\sfop{Jt}_W$ is a stratified smooth algebraic variety with the largest stratum of dimension $n$ equal to the dimension of $T$.
This structure is easiest seen in low-dimensional examples.

\begin{example}
	\label{251123113725}
	If $n = 2$, then the preimage $\sfop{Jt}_W^0 = \diag^{-1} (T_0)$ of the dense open subset $T_0 \subset T = \set{ \diag (t_1, t_2) } \cong (\CC^\times)^2$ given by $t_1 \neq t_2$ consists of the diagonal matrices $\diag (t_1, t_2)$ with $t_1 \neq t_2$.
	The map $\diag$ restricts to the obvious bijection $\sfop{Jt}_W^0 \cong T_0$.
	A cartoon of this situation is depicted in \autoref{251124174351}.
	On the other hand, the preimage of the $1$-dimensional subvariety $T_1 \subset T$ given by $t_1 = t_2 \eqcol t$ has two connected components $\sfop{Jt}_W^{10}, \sfop{Jt}_W^{11}$ consisting of all the matrices
	\begin{eqn}
		\mtx{ t & 0 \\ 0 & t}
		\qtext{and}
		\mtx{ t & 1 \\ 0 & t}.
	\end{eqn}
	The map $\diag$ restricts to give the obvious bijections $\sfop{Jt}_W^{10} \cong T_1$ and $\sfop{Jt}_W^{11} \cong T_1$.
	Furthermore, just as $T_1$ is contained in the closure of $T_0$, the subset $\sfop{Jt}_W^{10}$ (but not the subset $\sfop{Jt}_W^{11}$) is contained in the closure of $\sfop{Jt}_W^0$.
	This gives a description of $\sfop{Jt}_W$ as a stratified algebraic variety with the largest stratum $\sfop{Jt}_W^0$ of dimension $2$ isomorphic to $T_0$, and two $1$-dimensional strata $\sfop{Jt}_W^{10}, \sfop{Jt}_W^{11}$, each isomorphic to $T_1$:
	\begin{eqn}
		\sfop{Jt}_W = \sfop{Jt}_W^0 ~\coprod~ \Coprod_{i=0,1} \sfop{Jt}_W^{1i}.
		\tag*{\qedhere}
	\end{eqn}
\end{example}

\begin{figure}[t]
	\centering
	\begin{subfigure}{0.45\textwidth}
		\includegraphics[width=\textwidth]{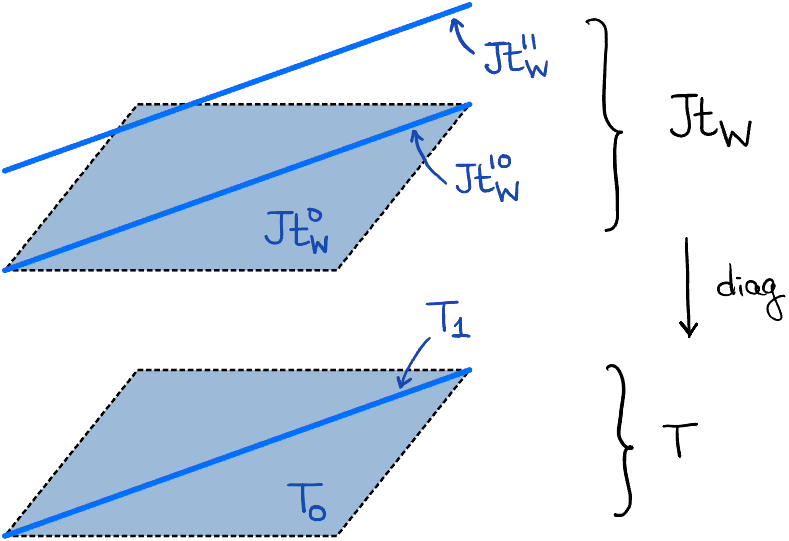}
		\caption{}
		\label{251124174351}
	\end{subfigure}
	\begin{subfigure}{0.45\textwidth}
		\includegraphics[width=\textwidth]{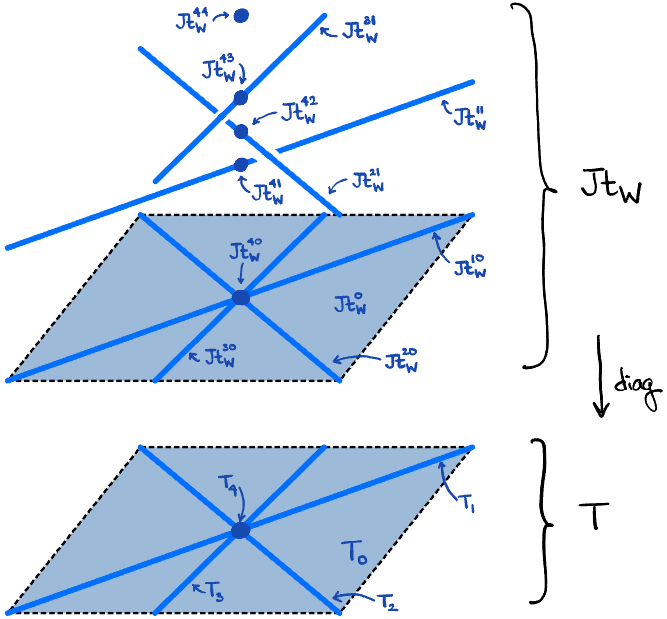}
		\caption{}
		\label{251124180859}
	\end{subfigure}
	\caption{}
	\label{251124181011}
\end{figure}

\begin{example}
	\label{251123113735}
	If $n = 3$, then $T = \set{ \diag (t_1, t_2, t_3) } \cong (\CC^\times)^3$ has the following relevant subsets; see \autoref{251124180859} for a cartoon.
	First, there is the dense open subset $T_0 \cong \set{ t_i \neq t_j } \subset T$.
	Second, there are three $2$-dimensional subsets $T_1 \cong \set{ t_2 = t_3 \eqcol t \neq t_1}$, $T_2 \cong \set{ t_1 = t_3 \eqcol t \neq t_2}$, and $T_3 \cong \set{ t_1 = t_2 \eqcol t \neq t_3}$.
	Finally, there is the $1$-dimensional subset $T_4 \cong \set{ t_1 = t_2 = t_3 }$.
	The preimage $\sfop{Jt}_W^0 \coleq \diag^{-1} (T_0)$ consists of the diagonal matrices $\diag (t_1, t_2, t_3)$ with $t_i \neq t_j$, hence $\sfop{Jt}_W^0 \cong T_0$.
	On the other hand, the preimages of $T_1, T_2, T_3$ have two connected components each.
	These preimages $\sfop{Jt}_W^{ij} = \diag^{-1} (T_i)$ with $j=0,1$ consist respectively of all matrices of the form
	\begin{eqn}
		\mtx{ t_1 & \ast & 0 \\ 0 & t & 0 \\ 0 & 0 & t }
		~,\qquad
		\mtx{ t & 0 & \ast \\ 0 & t_2 & 0 \\ 0 & 0 & t }
		~,\qquad
		\mtx{ t & 0 & 0 \\ 0 & t & \ast \\ 0 & 0 & t_3 }
		~,
	\end{eqn}
	where $\ast$ is $0$ or $1$, and $t, t_i$ have the same restrictions as those defining $T_1, T_2, T_3$ in each case.
	Thus, map $\diag$ determines a bijection $\sfop{Jt}_W^{ij} \cong T_i$.
	Finally, the preimage of $T_4$ has five connected components $\sfop{Jt}_W^{4j}, j = 0, \ldots, 4$, consisting of the matrices
	\begin{eqn}
		\mtx{ t & 0 & 0 \\ 0 & t & 0 \\ 0 & 0 & t }
		~,\qquad
		\mtx{ t & 1 & 0 \\ 0 & t & 0 \\ 0 & 0 & t }
		~,\qquad
		\mtx{ t & 0 & 1 \\ 0 & t & 0 \\ 0 & 0 & t }
		~,\qquad
		\mtx{ t & 0 & 0 \\ 0 & t & 1 \\ 0 & 0 & t }
		~,\qquad
		\mtx{ t & 1 & 0 \\ 0 & t & 1 \\ 0 & 0 & t }.
	\end{eqn}
	Notice that the matrix with $1$ in the $(1,3)$ entry is not a Jordan matrix, but it is a permutation of a Jordan matrix with two Jordan blocks of size $1$ and $2$.
	Again, the map $\diag$ restricts to give a bijection $\sfop{Jt}_W^{4j} \cong T_4$.
	Consequently, we obtain a description of $\sfop{Jt}_W$ as a stratified algebraic variety with the largest stratum $\sfop{Jt}_W^0$ of dimension $3$ isomorphic to $T_0$, and six $2$-dimensional strata $\sfop{Jt}_W^{ij}$ pairwise isomorphic to $T_1, T_2, T_3$, and five $1$-dimensional strata each isomorphic to $T_4$:
	\begin{eqn}
		\sfop{Jt}_W 
		= \sfop{Jt}_W^0 ~\coprod~ \Coprod_{\substack{i=1,2,3\\j=0,1}} \sfop{Jt}_W^{ij}
		~\coprod~ \Coprod_{\substack{j=0,\ldots,4}} \sfop{Jt}_W^{4j}.
		\tag*{\qedhere}
	\end{eqn}
\end{example}

Consequently, we can form the fibre product
\begin{eqn}
	\tilde{G} 
	\coleq G \tinyunderset{\times}{\op{Jt}} J_W
	= \set{ (\M, \Z) \in G \times J_W ~\Big|~ \op{Jt} (\M) = \op{Jt} (\Z) }
\end{eqn}
and conclude that the first projection $\tilde{G} \to G$ is a finite ramified covering of degree $n!$.

\subsection{$B$-Conjugacy Classes in $B$}

The of $G$-conjugacy classes in $G$ is in one-to-one correspondence with the set of Jordan types: $G / G \cong \sfop{Jt}$.
Similarly, we have the following result about conjugacy classes within the Borel subgroup $B \subset G$ of upper-triangular matrices.

\begin{proposition}
	\label{251122131732}
	The set $B / B$ of $B$-conjugacy classes in $B$ is in one-to-one correspondence with the set of all shuffled Jordan types $\sfop{Jt}_W$ in $G$.
	Similarly, the sets $B / U$ and $B / U$ of $U^\times$- and $U^\times$-conjugacy classes in $B$ are in one-to-one correspondence with the sets $\sfop{Jt}_N$ and $\sfop{Jt}_{N_0}$, respectively.
	\begin{eqn}
		B / B \cong \sfop{Jt}_W
		~,\qquad
		B / U \cong \sfop{Jt}_N
		~,\qquad
		B / U^\times \cong \sfop{Jt}_{N_0}.
	\end{eqn}
\end{proposition}


\printbibliography[title=References,heading=bibintoc]


\bigskip

\makeatletter
\newcommand\footnoteref[1]{\protected@xdef\@thefnmark{\ref{#1}}\@footnotemark}
\setcounter{footnote}{0}
\makeatother

\textsc{Benedetta Facciotti}\\
\small Universitat Polit\'ecnica de Catalunya%
\footnote{EPSEB, Avinguda del Dr Maranon 46-50,  08028 Barcelona, Spain.  benedetta.facciotti@upc.edu}

\smallskip

\textsc{Marta Mazzocco}\\
\small ICREA and Universitat Polit\'ecnica de Catalunya%
\footnote{EPSEB, Avinguda del Dr Maranon 46-50,  08028 Barcelona, Spain. marta.mazzocco@upc.edu}

\smallskip

\textsc{Nikita Nikolaev}\\
\small School of Mathematics, University of Birmingham%
\footnote{Watson Building, Edgbaston, Birmingham, B15 2TT, United Kingdom. n.nikolaev@bham.ac.uk}

\end{document}